%% file: RamC1.tex
\documentclass[10pt,a4paper]{article}

\input{packagesetc}
\usepackage{changebar}
\nochangebars

\usepackage{bm}
\usepackage{tabularx}
\usepackage{placeins}

\newcommand{\nocontentsline}[3]{}
\newcommand{\tocless}[2]{\bgroup\let\addcontentsline=\nocontentsline#1{#2}\egroup}

\makeatletter
\newcommand{\labitem}[2]{%
\def\@itemlabel{({#1})}
\item
\def\@currentlabel{#1}
\label{#2}}
\makeatother

\makeatletter

\newcommand{\Rmnum}[1]{\expandafter\@slowromancap\romannumeral #1@}
\makeatother

\newcommand{\aIna}{\llangle \aI n \rrangle}
\newcommand{\aIIna}{\llangle \aII n \rrangle}
\newcommand{\aIIIna}{\langle \aIII n \rangle}

\newcommand{\wY}{\widetilde{Y}}
\newcommand{\wX}{\widetilde{X}}

\usepackage{amssymb}
\renewcommand{\qed}{\qquad\hspace*{\fill}$\Box$}

\usepackage{hyperref}

\begin{document}

\begin{titlepage}
\begin{center}

\textsc{\LARGE \bf }\\[3cm]

{\huge \bf The Ramsey number of \\[12pt] mixed-parity cycles I}
\end{center}

\vspace{20mm}
\begin{center}
{\Large David G. Ferguson} 
\end{center}

\vspace{40mm}

\abstract{
\noindent 
Denote by $R(G_1, G_2, G_3)$ the minimum integer $N$ such that any three-colouring of the edges of the complete graph on $N$ vertices contains
a monochromatic copy of a graph $G_i$ coloured with colour~$i$ for some $i\in{1,2,3}$.
In a series of three papers of which this is the first, we consider the case where $G_1, G_2$ and $G_3$ are cycles of mixed parity. Specifically, in this and the subsequent paper, we consider~$R(C_n,C_m,C_{\ell})$, where $n$ and $m$ are even and $\ell$ is odd.
Figaj and \L uczak determined an asymptotic result for 
this case, which we improve upon to give an exact result. We prove that for~$n,m$ and $\ell$ sufficiently large
$$R(C_n,C_m,C_\ell)=\max\{2n+m-3, n+2m-3, \half n +\half m + \ell - 2\}.$$
In the case that the longest cycle is of even length, the proof in this paper is self-contained. However, in the case that the longest cycle is of odd length, we require an additional technical result, the proof of which makes up the majority of the subsequent paper.

}

\end{titlepage}


\pagenumbering{arabic}

\let\L\defaultL

 \renewcommand{\baselinestretch}{1.30}\small\normalsize
 
\setcounter{page}{2}


\renewcommand{\baselinestretch}{1.25}\small\normalsize

For graphs $G_1,G_2,G_3$, the Ramsey number $R(G_1,G_2,G_3)$ is the smallest integer~$N$ such that every edge-colouring of the complete graph on~$N$ vertices with up to three colours, results in the graph having, as a subgraph, a copy of~$G_{i}$ coloured with colour $i$ for some~$i$. 
We consider the case when $G_1,G_2$ and $G_3$ are~cycles.

In 1973, Bondy and Erd\H{o}s~\cite{BonErd} conjectured that, if~$n>3$ is odd, then $$R(C_{n},C_{n},C_{n})=4n-3.$$ 
Later, \L uczak~\cite{Lucz} proved that, for~$n$ odd,  $R(C_{n},C_{n},C_{n})=4n+o(n)$ as $n\rightarrow \infty$. Kohayakawa, Simonovits and Skokan~\cite{KoSiSk}, expanding upon the work of \L uczak, confirmed the Bondy-Erd\H{o}s conjecture for sufficiently large odd values of~$n$ by proving that there exists a positive integer~$n_0$ such that, for all odd $n,m,\ell>n_{0}$, 
$R(C_{n},C_{m},C_{\ell})=4 \max\{n,m,\ell\} -3.$

In the case where all three cycles are of even length, Figaj and \L uczak~\cite{FL2007} proved the following asymptotic. Defining $\llangle x\rrangle$ to be the largest even integer not greater than~$x$, they proved that, for all $\aI,\aII,\aIII>0$,
$$R(C_{\llangle \aI n \rrangle},C_{\llangle \aII n\rrangle},C_{\llangle \aIII n \rrangle})=\half\big(\aI + \aII +\aIII+\max \{\aI,\aII,\aIII\}\big)n+o(n),$$
as $n \rightarrow \infty$. 
Thus, in particular, for even~$n$, $R(C_{n},C_{n},C_{n})=2n+o(n),\text{ as }n\rightarrow \infty.$

Independently, Gy\'{a}rf\'{a}s, Ruszink\'{o}, S\'{a}rk\"{o}zy and Szem\'{e}redi~\cite{GyarSzem} proved a similar, but more precise, result for paths, namely that there exists a positive integer~$n_{1}$ such that, for $n>n_{1}$,
$$R(P_{n},P_{n},P_{n})=\begin{cases} 2n-1, & n \text{ odd,} \\ 2n-2, & n \text{ even.} \end{cases}$$
More recently, Benevides and Skokan~\cite{BenSko} proved that there exists~$n_{2}$ such that, for even $n>n_{2}$, 
$$R(C_{n},C_{n},C_{n})=2n.$$
We look at the mixed-parity case, for which, 
defining $\langle x \rangle$ to be the largest odd number not greater than~$x$, Figaj and \L uczak~\cite{FL2008} proved that, for all $\aI,\aII,\aIII>0$,
\begin{align*}
&\text{(i)}\ R(C_{\llangle \aI n \rrangle },C_{\llangle \aII n \rrangle },C_{\langle \aIII n\rangle  }) = \max \{2\aI+\aII, \aI+2\aII, \half\aI  + \half\aII +\aIII \}n +o(n),\\
&\text{(ii)}\ R(C_{\llangle \aI n \rrangle },C_{\langle \aII n \rangle  },C_{\langle \aIII n\rangle  }) = \max \{4\aI,\aI+2\aII, \aI  +2\aIII \}n +o(n),
\end{align*}
as $n\rightarrow \infty$.

Specifically, improving on their result, in the case when exactly one of the cycles is of odd length, we prove the following, which is the main result of this paper:

\phantomsection
\hypertarget{thA}
\phantomsection
\begin{thA}
\label{thA}

For every $\alpha_{1}, \alpha_{2}, \alpha_{3}>0$ such that $\aI \geq \aII$, there exists a positive integer $n_{A}=n_{A}(\aI,\aII,\aIII)$ such that, for $n> n_{A}$,
 \begin{align*}
 R(C_{\llangle \alpha_{1} n \rrangle},C_{\llangle \alpha_{2} n \rrangle}, C_{\langle \alpha_{3} n \rangle }) = \max\{ 2\llangle \alpha_{1} n \rrangle \!+\! \llangle \alpha_{2} n \rrangle  \!-\!\text{\:}3,\text{\:}\half\llangle  \alpha_{1} n \rrangle  \!+\! \half\llangle \alpha_{2} n \rrangle  \!+\! \langle \alpha_{3} n \rangle \!-\! \text{\:}2\}.
\end{align*}
\end{thA}

In order to prove Theorem~\hyperlink{thA}{A}, we require a new technical result, Theorem~\hyperlink{thB}{B}, which we will state in Section~\ref{s:struct}.
Owing to the length of the proof of Theorem~\hyperlink{thB}{B}, we postpone part of it to \cite{DF2}. However, this paper does contains a complete proof of Theorem~\hyperlink{thA}{A} in the case that $\aI\geq\aIII$.

In \cite{DF3}, we consider the case when exactly one of the cycles is of even length, proving the following result, which again improves upon the corresponding result of Figaj and \L uczak:

\phantomsection
\hypertarget{thC}
\phantomsection
\begin{thC}
\label{thC}

For every $\alpha_{1}, \alpha_{2}, \alpha_{3}>0$, there exists a positive integer $n_{C}=n_{C}(\aI,\aII,\aIII)$ such that, for $n> n_{C}$,
 \begin{align*}
 R(C_{\llangle \alpha_{1} n \rrangle},C_{\langle \alpha_{2} n \rangle}, C_{\langle \alpha_{3} n \rangle }) = \max\{
4\llangle \aI n \rrangle,
\llangle \aI n \rrangle+2\langle \aII n \rangle,
\llangle \aI n \rrangle+2\langle \aIII n \rangle\}-3
.
 \end{align*}
\end{thC}

\section{Lower bounds}
\label{ram:low}
\setlength{\parskip}{0.1in plus 0.05in minus 0.025in}

The first step in proving Theorem~\hyperlink{thA}{A} is to exhibit three-edge-colourings of the complete graph on $$ \max\left\{ 2\llangle \alpha_{1} n \rrangle + \llangle \alpha_{2} n \rrangle -\text{\:}4,\text{\:}\half\llangle \alpha_{1} n \rrangle + \half\llangle \alpha_{2} n \rrangle +\langle \alpha_{3} n \rangle-\text{\:}3\right\}$$ vertices which do not contain any of the relevant coloured cycles, thus proving that 
$$ R(C_{\llangle \alpha_{1} n \rrangle},C_{\llangle \alpha_{2} n \rrangle}, C_{\langle \alpha_{3} n \rangle }) \geq \max\{ 2\llangle \alpha_{1} n \rrangle + \llangle \alpha_{2} n \rrangle -\text{\:}3,\text{\:}\half\llangle \alpha_{1} n \rrangle + \half\llangle \alpha_{2} n \rrangle +\langle \alpha_{3} n \rangle-\text{\:}2\}.$$
For this purpose, the well-known colourings shown in Figure~\ref{th0} suffice.
\begin{figure}[!h]
\centering{
\mbox{
{\includegraphics[width=64mm, page=1]{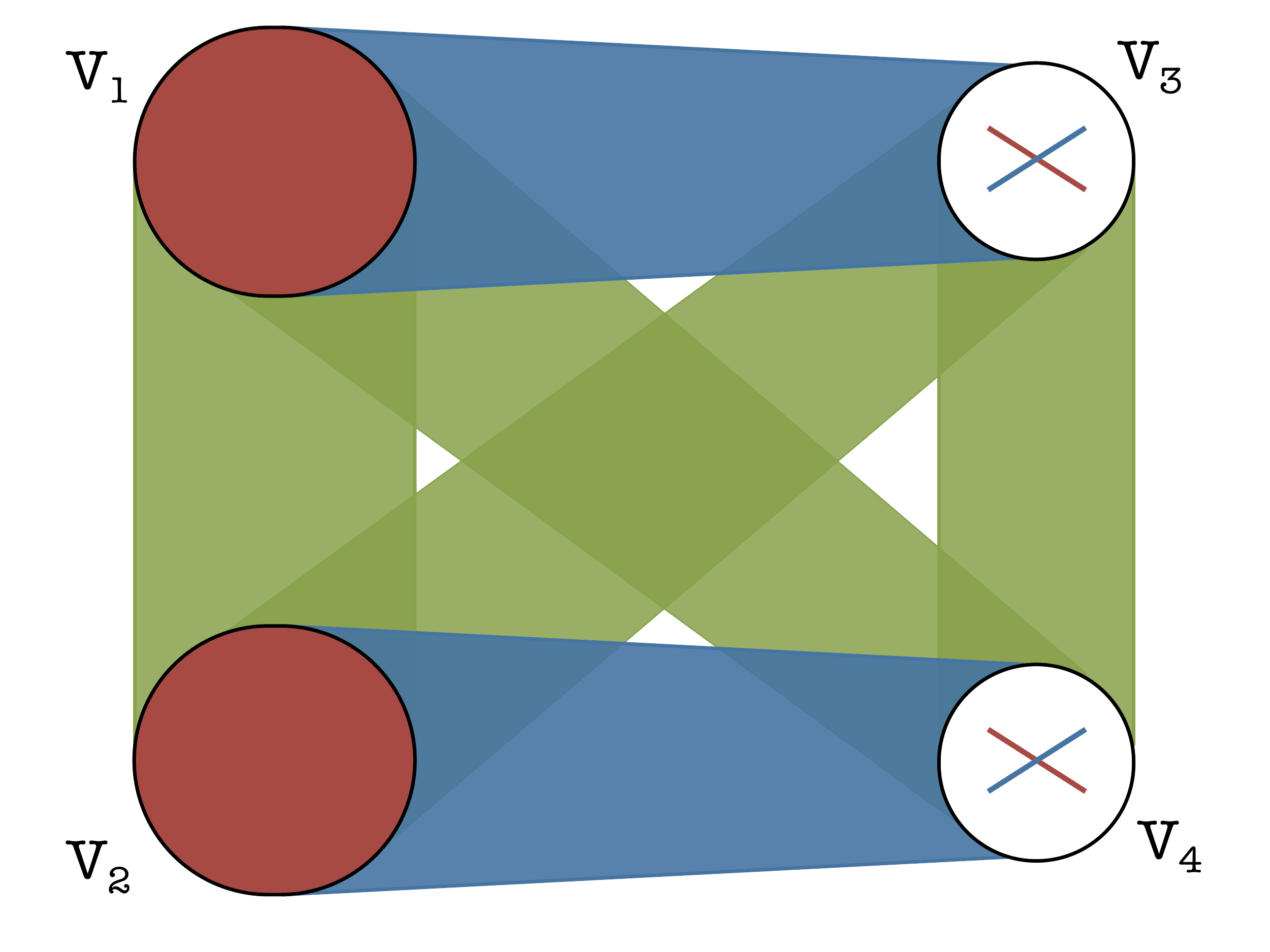}}\quad
{\includegraphics[width=64mm, page=2]{CH2-Figs.pdf}}}}
\caption{Extremal colouring for Theorem~A.}   
\label{th0}
\end{figure}

The first graph shown in Figure~\ref{th0} has $2\llangle \alpha_{1} n \rrangle + \llangle \alpha_{2} n \rrangle-4$ vertices divided into four classes $V_{1}, V_{2}, V_{3}$ and $V_{4}$, with 
$
|V_{1}|=|V_{2}|=\llangle \alpha_{1} n \rrangle-1$, 
$|V_{3}|=|V_{4}|=\half\llangle \alpha_{2} n \rrangle-1$,
such that all edges in $G[V_1]$ and $G[V_2]$ are coloured red; all edges in $G[V_1,V_3]$ and $G[V_2,V_4]$ are coloured blue; all edges in $G[V_1 \cup V_3,V_2 \cup V_4]$ are coloured green; and all edges in $G[V_3]$ and $G[V_4]$ are coloured red or blue.

The second graph shown in Figure~\ref{th0} has $\half\llangle \alpha_{1} n \rrangle + \half\llangle \alpha_{2} n \rrangle +\langle \alpha_{3} n \rangle- 3$ vertices, divided into three classes $V_{1}, V_{2}$ and $V_{3}$, with 
$
|V_{1}|=\half\llangle \alpha_{1} n \rrangle-1$, 
$|V_{2}|=\half\llangle \alpha_{2} n \rrangle-1$, 
$|V_{3}|=\langle \alpha_{3} n \rangle-1$,
such that all edges in $G[V_1]\cup G[V_1,V_3]$ are coloured red;
all edges in $G[V_2]\cup G[V_2,V_3]$ are coloured blue; and 
all edges in $G[V_1,V_2] \cup G[V_3]$ are coloured green.

Thus, it remains to prove the corresponding upper-bound. To do so, we combine regularity (as used in~\cite{Lucz},~\cite{FL2007},~\cite{FL2008}) with stability methods using a similar approach to~\cite{GyarSzem}, \cite{BenSko}, \cite{KoSiSk}, \cite{KoSiSk2}.

\section{Key steps in the proof}
\label{ram:key}

In order to complete the proof of Theorem~\hyperlink{thA}{A}, we must show that, for $n$ sufficiently large, any three-colouring of~$G$, the complete graph on $$N=\max\left\{ 2\llangle \alpha_{1} n \rrangle + \llangle \alpha_{2} n \rrangle -\text{\:}3,\text{\:}\half\llangle \alpha_{1} n \rrangle + \half\llangle \alpha_{2} n \rrangle +\langle \alpha_{3} n \rangle-\text{\:}2\right\}$$ vertices will result in either a red cycle on~$\llangle \aI n \rrangle$ vertices, a blue cycle on~$\llangle \aII n \rrangle$ or a green cycle on $\langle \aIII n \rangle$ vertices. 

The main steps of the proof are as follows: Firstly, we apply a version of the Regularity Lemma (Theorem~\ref{l:sze}) to give a partition $V_0\cup V_1\cup\dots\cup V_K$ of the vertices which is simultaneously regular for the red, blue and green spanning subgraphs of~$G$. Given this partition, we define the three-multicoloured reduced-graph $\cG$ on vertex set $V_1,V_2,\dots V_K$ whose edges correspond to the regular pairs. We colour the edges of the  reduced-graph with all those colours for which the corresponding pair has density above some threshold.
 \L uczak~\cite{Lucz} showed that, if the threshold is chosen properly, then the existence of a matching in a monochromatic connected-component of  the  reduced-graph implies the existence of a monochromatic cycle of the corresponding length in the original graph.

Thus, the key step in the proof of Theorem~\hyperlink{thA}{A} will be to prove a Ramsey-type stability result for so-called connected-matchings (Theorem~\hyperlink{thB}{B}). Defining a \textit{connected-matching} to be a matching with all its edges belonging to the same component, 
this result essentially says that, for every $\alpha_{1},\alpha_{2},\alpha_{3}>0$
such that $\aI\geq \aII$ and every sufficiently large $k$, every three-multicolouring of a graph $\cG$ on slightly fewer than $K=\max\{2\aI+\aII, \half\aI+\half\aII+\aIII\}k$ vertices with sufficiently large minimum degree results in either a connected-matching on at least~$\aI k$ vertices in the red subgraph of $\cG$, a connected-matching on at least~$\aII k$ vertices in the blue subgraph of $\cG$, a connected-matching on at least~$\aIII k$ vertices in a non-bipartite component of the green subgraph of $\cG$ or one of a list of particular structures which will be defined later.
In the case that $\cG$ contains a suitably large connected-matching in one of its coloured subgraphs, a blow-up result of Figaj and \L uczak (see Theorem~\ref{th:blow-up}) can be used to give a monochromatic cycle of the same colour in~$G$. If $\cG$ does not contain such a connected-matching, then the stability result gives us information about the structure of $\cG$. We then show that~$G$ has essentially the same structure, which we exploit to force the existence of a monochromatic cycle.

In the next section, given a three-colouring of the complete graph on~$N$ vertices, we will define its three-multicoloured  reduced-graph. We will also state and prove a version of the blow-up lemma of Figaj and \L uczak, which motivates our whole approach.
In Section~\ref{ram:defn}, we will deal with some notational formalities before proceeding in Section~\ref{s:struct} to define the structures we need and to give a precise formulation of the connected-matching stability result, which we shall call Theorem~\hyperlink{thB}{B}.
In Section~\ref{s:pre1}, we give a number of technical lemmas needed for the proofs of Theorem~\hyperlink{thA}{A} and Theorem~\hyperlink{thB}{B}. Among these is a decomposition result of Figaj and \L uczak,  which provides insight into the structure of the  reduced-graph. 
The hard work is done in Sections~\ref{s:stabp}--\ref{s:stabp2}, where we prove Theorem~\hyperlink{thB}{B}, and in Sections~\ref{s:p10}--\ref{s:p13}, where we translate this result for connected-matchings into one for cycles, thus completing the proof of Theorem~\hyperlink{thA}{A}. 

The proof of Theorem~\hyperlink{thB}{B} is divided into two parts according to the relative sizes of $\aI$ and $\aIII$. Section~\ref{s:stabp} deals with the case when $\aI \geq \aIII$, that is, the case when the longest cycle has even length. In that case, a combination of the decomposition lemma of Figaj and \L uczak and some careful counting of edges allows for a reasonably short proof. Section~\ref{s:stabp2} deals with the opposite case, which requires a longer proof utilising an alternative decomposition and extensive case analysis. 
The final part of the proof of Theorem~\hyperlink{thA}{A} is divided into four sub-parts, one dealing with the general setup and three further sections, each dealing with one of the structures that can occur. 

Note that Sections~\ref{ram:cmr}--\ref{s:stabp},~\ref{s:p10} and~\ref{s:p11} together give a complete proof for the case where the longest cycle is of even length, allowing the reader to postpone reading of sections~\ref{s:p12} and~\ref{s:p13}, which together with~\cite{DF2} complete the proof of Theorem~\hyperlink{thA}{A}.

\section{Cycles, matchings and regularity}
\label{ram:cmr}

 Recall that Szemer\'{e}di's Regularity Lemma~\cite{SzemRegu} asserts that any sufficiently large graph can be approximated by the union of a bounded number of random-like bipartite graphs. Recall also that, given a pair $(A,B)$ of disjoint subsets of the vertex set of a graph~$G$, we write $d(A,B)$ for the \textit{density} of the pair, that is, $d(A,B)=e(A,B)/|A||B|$. Finally, recall that we say such a pair  is $(\epsilon,G)$\textit{-regular} for some~$\epsilon>0$ if, for every pair $(A',B')$ with $A'\subseteq A$, $|A'|\geq \epsilon |A|$, $B' \subseteq B$, $|B'|\geq \epsilon |B|$, we have $ \left| d(A',B')-d(A,B) \right| <\epsilon.$
 
We will make use of a generalised version of Szemer\'edi's Regularity Lemma~in order to move from considering monochromatic cycles to considering monochromatic connected-matchings, the version below being a slight modification of one found, for instance, in~\cite{KomSim}: 

\begin{theorem}
\label{l:sze}
For every $\epsilon>0$ and every positive integer~$k_{0}$, there exists $K_{\ref{l:sze}}=K_{\ref{l:sze}}(\epsilon,k_{0})$ such that the following holds: For all graphs $G_{1},G_{2},G_{3}$ with $V(G_{1})=V(G_{2})=V(G_{3})=V$ and $|V|\geq K_{\ref{l:sze}}$, there exists a partition $\Pi =(V_{0},V_{1},\dots,V_{K})$ of $V$ such that
\begin{itemize}
\item [(i)]$k_{0}\leq K \leq K_{\ref{l:sze}}$;
\item [(ii)]$|V_0|\leq \epsilon |V|$;   
\item [(iii)]$|V_1|=|V_2|=\dots =|V_K|$; and
\item [(iv)] for each $i$, all but at most $\epsilon K$ of the pairs $(V_i,V_j)$, $1\leq i<j\leq K$, are simultaneously ($\epsilon,G_r$)-regular for $r=1,2,3$. 
\end{itemize}
\end{theorem}

Note that, given $\epsilon>0$ and graphs $G_1,G_2$ and $G_3$ on the same vertex set $V$, we call a partition $\Pi=(V_0,V_1,\dots,V_K)$, satisfying (ii)--(iv) \textit{$(\epsilon,G_1,G_2,G_3)$-regular}.

In what follows, given a three-coloured graph~$G$, we will use $G_1, G_2, G_3$ to refer to its monochromatic spanning subgraphs, that is, $G_1$ (resp. $G_2, G_3$) has the same vertex set as~$G$ and includes, as an edge, any edge which in~$G$ is coloured red (resp. blue, green).
Then, given a three-coloured graph~$G$, we can use Theorem~\ref{l:sze} to define a partition which is simultaneously regular for $G_1$, $G_2$, $G_3$ and then define the three-multicoloured  reduced-graph $\cG$ as follows: 

\begin{definition}
\label{reduced}
Given $\epsilon>0$, $\xi>0$, a three-coloured graph $G=(V,E)$  and an $(\epsilon,G_1,G_2,G_3)$-regular partition $\Pi =(V_{0},V_{1},\dots,V_{K})$, we define the three-multicoloured $(\epsilon,\xi,\Pi)$-reduced-graph $\cG=(\cV,\cE)$ by:
\begin{align*}
\mathcal{V}&=\{V_{1},V_{2},\dots ,V_{K}\}, \quad\quad\quad
\mathcal{E}&=\{ V_{i}V_{j} : (V_{i},V_{j}) \text{ is simultaneously } (\epsilon,G_{r})\text{-regular for }r=1,2,3\},
\end{align*}
where $V_{i}V_{j}$ is coloured with all colours~$r$ such that $d_{G_r}(V_i,V_j)\geq\xi$.

\end{definition}

One well-known fact about regular pairs is that they contain long paths. This is summarised in the following lemma, which is a slight modification of one found in~\cite{Lucz}:

\begin{lemma}
\label{longpath}
For every~$\epsilon$ such that $0\leq \epsilon < 1/600 $ and every $k\geq1/\epsilon$, the following holds: Let~$G$ be a bipartite graph with bipartition $V(G)=V_1\cup V_2$ such that $|V_1|,|V_2|\geq k$, the pair $(V_1,V_2)$ is $\epsilon$-regular and $e(V_1,V_2)\geq \epsilon^{1/2} |V_1||V_2|$. Then, for every integer~$\ell$ such that $1\leq \ell \leq k-2\epsilon^{1/2} k$ and every $v' \in V_1$, $v'' \in V_2$ such that $d(v'),d(v'')\geq \tfrac{2}{3}\epsilon^{1/2}k$,~$G$ contains a path of length $2\ell +1$ between~$v'$ and~$v''$.
\end{lemma}

A \textit{matching} is a collection of pairwise vertex-disjoint edges. Note that, in what follows, we will sometimes abuse terminology and, where appropriate, refer to a matching by its vertex set rather than its edge set.
We call a matching with all its vertices in the same component of~$G$ a \textit{connected-matching} and note that we say a connected-matching is \textit{odd} if the component containing the matching also contains an odd cycle. 

The following theorem makes use of the Lemma~above to blow up large connected-matchings in the  reduced-graph to cycles (of appropriate length and parity) in the original. This facilitates our approach to proving Theorem~\hyperlink{thA}{A} in that it allows us to shift our attention away from cycles to connected-matchings, which turn out to be somewhat easier to find.

Figaj and \L uczak~\cite[Lemma~3]{FL2008} proved a more general version of this theorem in a slightly different context (they considered any number of colours and any combination of parities and used a different threshold for colouring the reduced-graph):

\begin{theorem}
\label{th:blow-up}

For all $c_1,c_2,c_3,d, \eta>0$ such that
$0<\eta<\min\{0.01, (64c_1+64c_2+64c_3)^{-1}\}$, there exists $n_{\ref{th:blow-up}}=n_{\ref{th:blow-up}}(c_1,c_2,c_3,d,\eta)$ such that, for $n>n_{\ref{th:blow-up}}$, the following holds:

Given $\aI, \aII, \aIII$  such that $0<\aI,\aII,\aIII\leq2$, and $\xi$ such that $\eta\leq \xi\leq\tfrac{1}{3}$, a complete three-coloured graph $G=(V,E)$ on
$$N=c_{1}\llangle \aI n\rrangle+c_{2}\llangle \aII n \rrangle +c_{3}\langle\alpha_{3}n\rangle-d$$
vertices and an $(\eta^4,G_1,G_2,G_3)$-regular partition $\Pi =(V_{0},V_{1},\dots,V_{K})$ for some $K>8(c_1+c_2+c_3)^2/\eta$, letting $\cG=(\cV,\cE)$ be the three-multicoloured $(\eta^4,\xi,\Pi)$- reduced-graph of~$G$ on~$K$ vertices and letting~$k$ be an integer such that 
$$
c_{1}\alpha_{1}k+c_{2}\aII k+c_{3}\aIII k - \eta k \leq  K \leq c_{1}\alpha_{1}k+c_{2}\aII k+c_{3}\aIII k - \half\eta k,$$
\begin{itemize}
\item[(i)] if $\cG$ contains a red connected-matching on at least $\alpha_{1}k$ vertices, then~$G$ contains a red cycle on $\llangle \alpha_{1} n\rrangle$ vertices;

\item[(ii)] if $\cG$ contains a blue connected-matching on at least $\alpha_{2}k$ vertices, then~$G$ contains a blue cycle on $\llangle \alpha_{2} n\rrangle$ vertices;

\item[(iii)] if $\cG$ contains a green odd connected-matching on at least $\alpha_{3}k$ vertices, then~$G$ contains a green cycle on $\langle\alpha_{3}n\rangle $ vertices.
\end{itemize}
\end{theorem}

\section{Definitions and notation}
\label{ram:defn}
%

Recall that, given a three-coloured graph~$G$, we refer to the first, second and third colours as red, blue and green respectively and use $G_1, G_2, G_3$ to refer to the monochromatic spanning subgraphs of $G$. 
If~$G_1$ contains the edge $uv$, we say that $u$ and $v$ are \textit{red neighbours} of each other in $G$. Similarly, if $uv\in E(G_2)$, we say that $u$ and $v$ are \textit{blue neighbours} and, if $uv\in E(G_3)$, we say that that $u$ and $v$ are \textit{green neighbours}. For vertices~$u$ and~$v$ in a graph~$G$, we will say that the edge $uv$ is \emph{missing} if~$uv\notin E(G)$.

We say that a graph~$G=(V,E)$ on~$N$ vertices is $a$-\emph{almost-complete} for $0\leq a\leq N-1$ if its minimum degree~$\delta(G)$ is at least $(N-1)-a$. Observe that, if~$G$ is $a$-almost-complete and $X\subseteq V$, then $G[X]$ is also $a$-almost-complete. We say that a graph~$G$ on~$N$ vertices is $(1-c)$-\emph{complete} for $0\leq c\leq 1$ if it is $c(N-1)$-almost-complete, that is, if $\delta(G)\geq (1-c)(N-1)$. Observe that, for $c\leq \half$, any $(1-c)$-complete graph is connected.

We say that a bipartite graph~$G=G[U,W]$ is $a$-\emph{almost-complete} if every $u\in U$ has degree at least $|W|-a$ and every $w\in W$ has degree at least $|U|-a$. Notice that, if~$G[U,W]$ is $a$-almost-complete and $U_1\subseteq U, W_1\subseteq W$, then $G[U_1,W_1]$ is $a$-almost-complete. We say that a bipartite graph~$G=G[U,W]$ is $(1-c)$-\emph{complete} if every $u\in U$ has degree at least $(1-c)|W|$ and every $w\in W$ has degree at least $(1-c)|U|$. Again, notice that, for $c< \half$, any $(1-c)$-complete bipartite graph $G[U,W]$ is connected, provided that~$U,W\neq \emptyset$.

We say that a graph~$G$ on~$N$ vertices is $c$-\emph{sparse} for $0<c<1$ if its maximum degree is at most $c(N-1)$. We say a bipartite graph~$G=G[U,W]$ is $c$-\emph{sparse} if every $u\in U$ has degree at most $c|W|$ and every vertex $w\in W$ has degree at most $c|U|$.

\section {Connected-matching stability result}
\label{s:struct}

Before proceeding to state Theorem~\hyperlink{thB}{B}, 
we define the coloured structures we will need.

\begin{definition}
\label{d:H}

For $x_{1}, x_{2}, c_1,c_2$ positive, $\gamma_1,\gamma_2$ colours, let $\cH(x_{1},x_{2}, c_1,c_2, \gamma_1,\gamma_2)$ be the class of edge-multicoloured graphs defined as follows: 
A given two-multicoloured graph $H=(V,E)$ belongs to~$\cH$ if its vertex set can be partitioned into $X_{1}\cup X_{2}$ such that
\begin{itemize}
\item[(i)] $|X_{1}|\geq x_{1}, |X_{2}|\geq x_{2} $;
\item[(ii)] $H$ is $c_1$-almost-complete; and
\item[(iii)] defining $H_1$ to be the spanning subgraph induced by the colour $\gamma_1$ and $H_2$ to be the subgraph induced by the colour $\gamma_2$,
\begin{itemize}
\item[(a)] $H_1[X_{1}]$ is $(1-c_2)$-complete and $H_2[X_{1}]$ is $c_2$-sparse,
\item[(b)] $H_2[X_1,X_2]$ is $(1-c_2)$-complete and $H_1[X_1,X_2]$ is $c_2$-sparse.
\end{itemize}
\end{itemize}
\end{definition}

\begin{figure}[!h]
\vspace{-2mm}
\centering
\includegraphics[width=64mm, page=1]{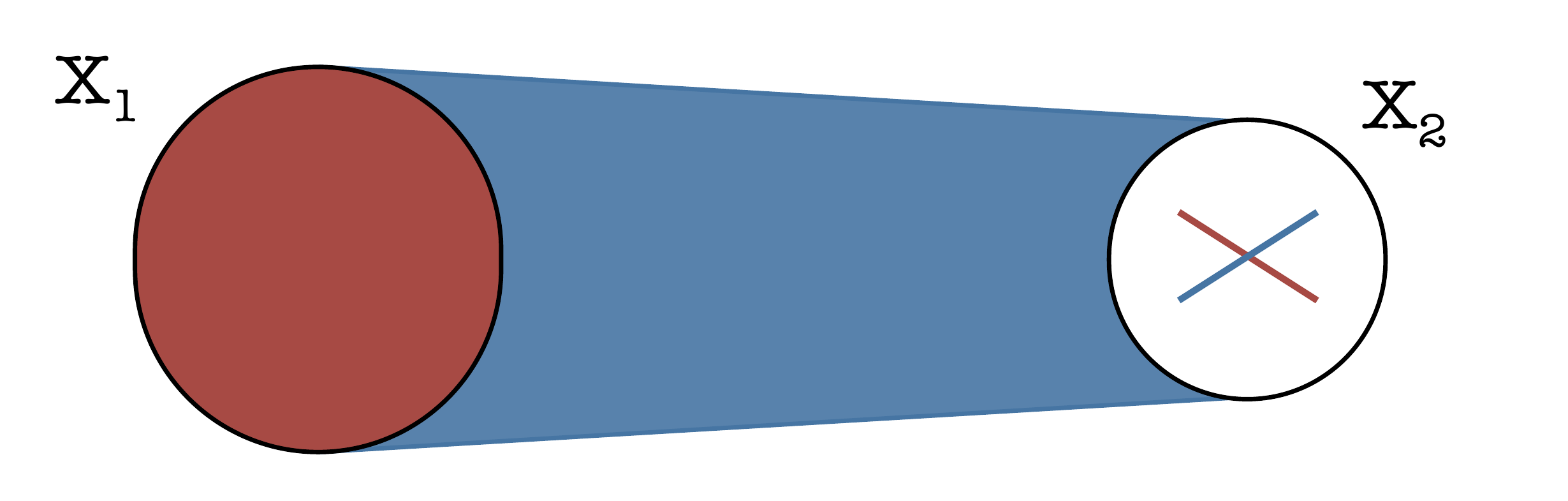}
\vspace{0mm}\caption{$H\in\cH(x_1,x_2,c_1,c_2,red,blue)$.}
\end{figure}

\begin{definition}
\label{d:K}

For $x_{1}, x_{2},x_{3}, c$ positive, let $\cK(x_{1},x_{2}, x_3,c)$ be the class of edge-multicoloured graphs defined as follows: 
A given three-multicoloured graph $H=(V,E)$ belongs to~$\cK$ if its vertex set can be partitioned into $X_{1}\cup X_{2}\cup X_{3}$ such that
\begin{itemize}
\item[(i)] $|X_{1}|\geq x_{1}, |X_{2}|\geq x_{2} , |X_{3}|\geq x_{3} $;
\item[(ii)] $H$ is $c$-almost-complete; 
\item[(iii)] all edges present in $H[X_1,X_3]$ are coloured exclusively red, all edges present in $H[X_2,X_3]$ are coloured exclusively blue and all edges present in $H[X_3]$ are coloured exclusively green. 
\end{itemize}
\end{definition}

\begin{definition}
\label{d:J}

For $x_{1}, x_{2},y_{1}, y_{2}, z, c$ positive, let $\cK^*(x_{1},x_{2}, y_1, y_2, z, c)$ be the class of edge-multicoloured graphs defined as follows: 
A given three-multicoloured graph $H=(V,E)$ belongs to~$\cK^*$, if its vertex set can be partitioned into $X_{1}\cup X_{2}\cup Y_{1}\cup Y_{2}$ such that
\begin{itemize}
\item[(i)] $|X_{1}|\geq x_{1}, |X_{2}|\geq x_{2} , |Y_{1}|\geq y_{2}, |Y_{2}|\geq y_{2}, |Y_1|+|Y_2|\geq z$;
\item[(ii)] $H$ is $c$-almost-complete; 
\item[(iii)] all edges present in $H[X_1,Y_1]$ and $H[X_2,Y_2]$ are coloured exclusively red, all edges present in $H[X_1,Y_2]$ and $H[X_2,Y_1]$ are coloured exclusively blue and all edges present in $H[X_1,X_2]$ and $H[Y_1,Y_2]$ are coloured exclusively green.
\end{itemize}
\end{definition}

\begin{figure}[!h]
\end{figure}
\vspace{-6mm}
\begin{figure}[!h]
\centering{
\mbox{
\hspace{-8mm}
{~}\quad
{\includegraphics[width=64mm, page=4]{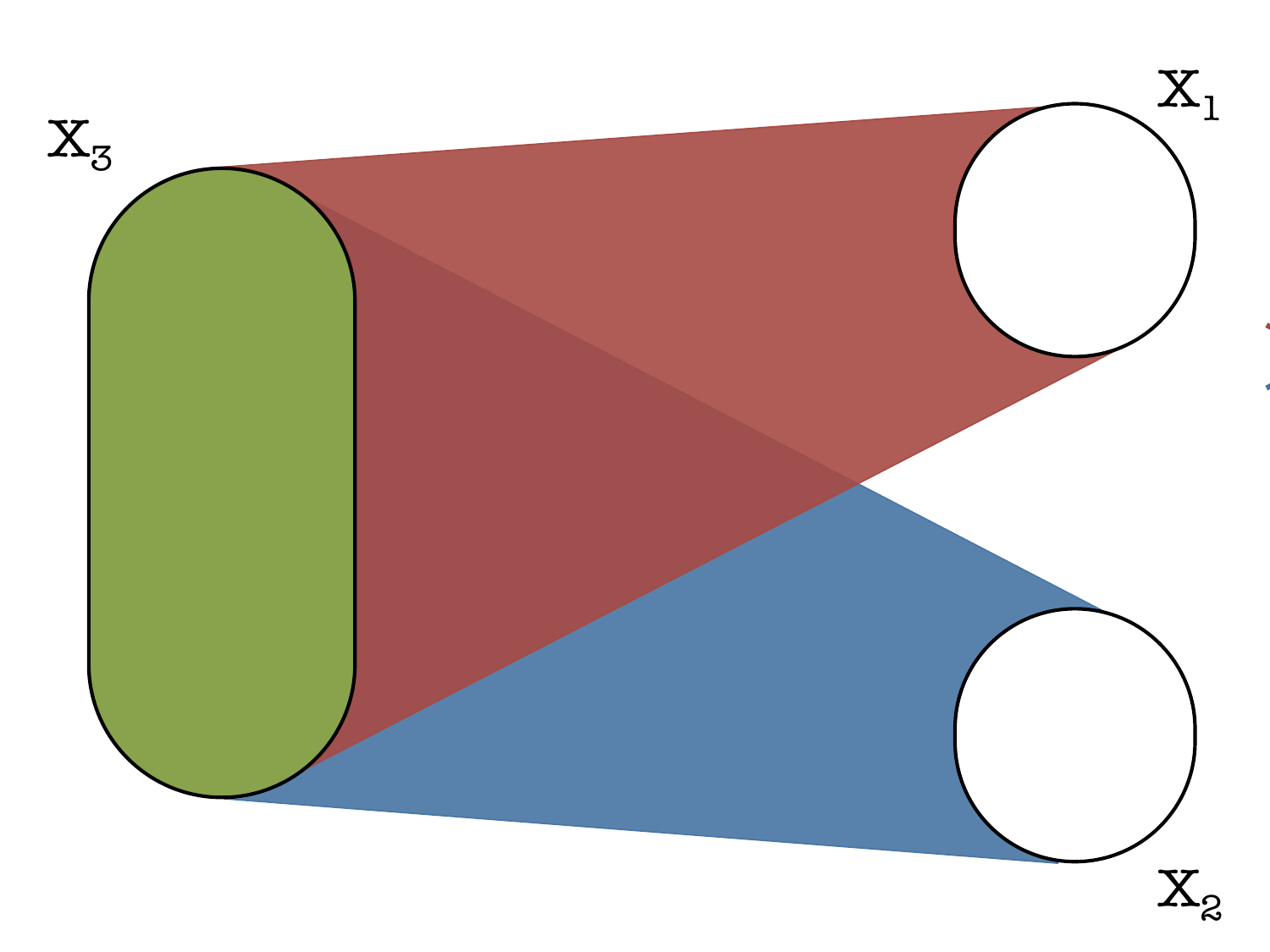}}\quad\quad\quad
{\includegraphics[width=64mm, page=5]{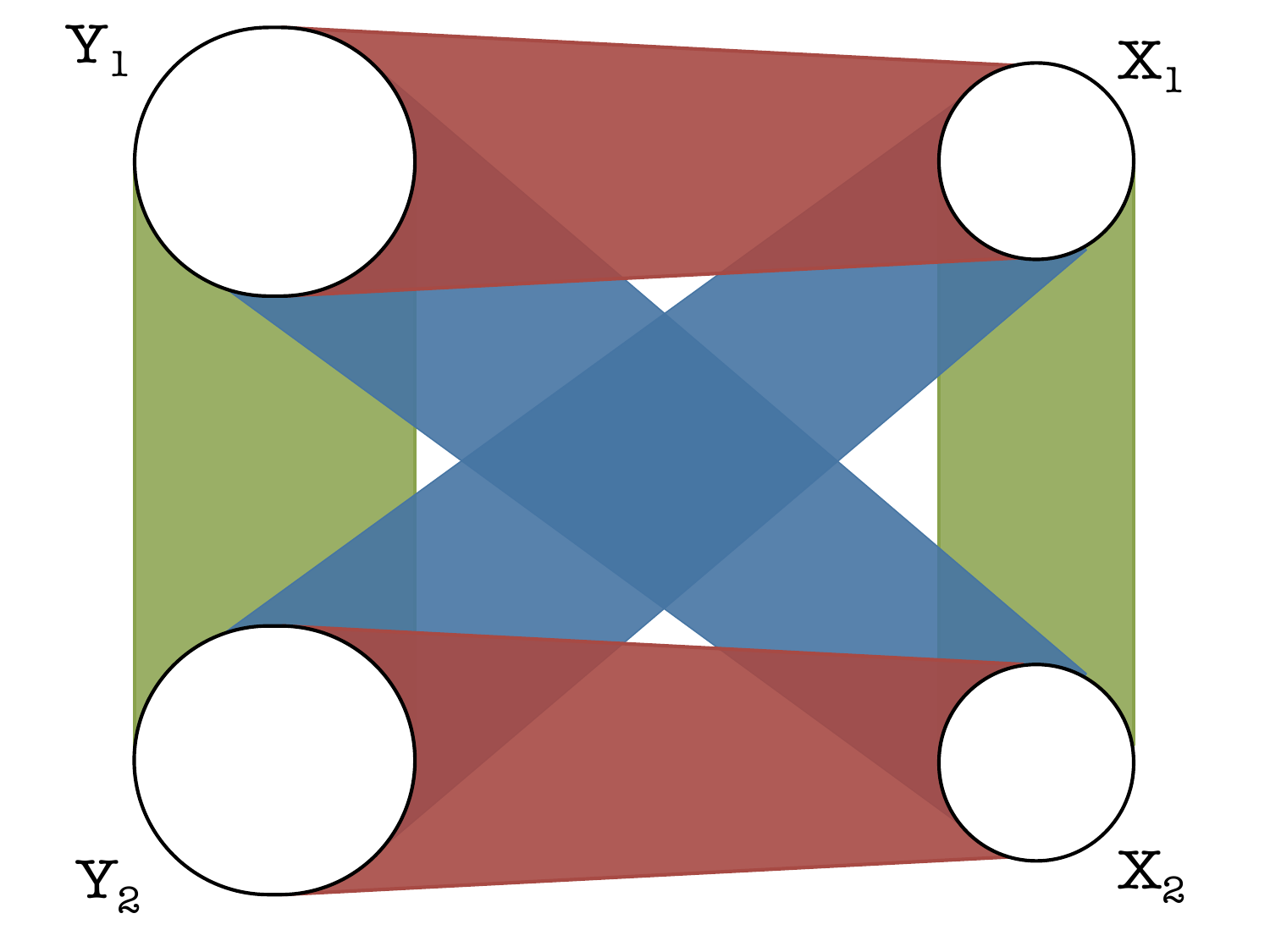}}}}
\setcounter{figure}{2}
\caption{$H\in \cK(x_1,x_2,x_3,c)$ and $H\in \cK^{*}(x_1,x_2,y_1,y_2,c)$.}   
\label{th1}
\setcounter{figure}{2}
\end{figure}
\setcounter{figure}{3}

\FloatBarrier
Having defined the coloured structures, we are in a position to state the main technical result of this chapter, that is, the connected-matching stability result. The proof of this result, which follows in Sections~\ref{s:stabp}--\ref{s:stabp2}, takes up the majority of this chapter.

\FloatBarrier
\phantomsection
\hypertarget{thB}
\phantomsection
\begin{thB}
\label{thB}
\label{th:stab}
For every $\alpha_{1},\alpha_{2},\alpha_{3}>0$ 
such that $\aI\geq \aII$, letting $$c=\max\{2\aI+\aII, \half\aI+\half\aII+\aIII\},$$
there exists $\eta_{B}=\eta_{B}(\aI,\aII,\aIII)$ and $k_{B}=k_{B}(\aI,\aII,\aIII,\eta)$ 
such that, for every $k>k_{B}$ and 
every~$\eta$ such that $0<\eta<\eta_{B}$, every 
three-multicolouring of~$G$, a $(1-\eta^4)$-complete graph on 
$$(c-\eta)k\leq K\leq(c-\half\eta)k$$ vertices,  results in the graph containing at least one of the following:
\vspace{-2mm}
\begin{itemize}
\item [(i)]	a red connected-matching on at least $\aI
 k$ vertices;
\item [(ii)]  a blue connected-matching on at least $\aII k$ vertices;
\item [(iii)]  a green odd connected-matching on at least $\aIII k$ vertices;
\item [(iv)]  
disjoint subsets of vertices $X$ and $Y$ such that $G[X]$ contains a two-coloured spanning subgraph~$H$ from $\cH_{1}\cup\cH_{2}$ and $G[Y]$ contains a two-coloured spanning subgraph $K$ from $\cH_{1}\cup\cH_{2}$, where
\vspace{-2mm}
\begin{align*}
\cH_1=&\cH\left((\aI-2\eta^{1/32})k,(\half\aII-2\eta^{1/32})k,3\eta^4 k,\eta^{1/32},\text{red},\text{blue}\right),\text{ } 
\\ \cH_2=&\cH\left((\aII-2\eta^{1/32})k,(\half\aI-2\eta^{1/32})k,3\eta^4 k,\eta^{1/32},\text{blue},\text{red}\right);
\end{align*}
 \item [(v)]  a subgraph from 
 \vspace{-2mm}
 $$\cK\left((\half\aI-14000\eta^{1/2})k, (\half\aII-14000\eta^{1/2})k, (\aIII-68000\eta^{1/2})k, 4\eta^4 k \right);$$
\item [(vi)]  a subgraph from $\cK_1^{*}\cup \cK_2^*$, where
\vspace{-2mm}
\begin{align*}
\cK_{1}^{*}=\cK^*\big((\half\aI-97\eta^{1/2})k, (\half\aI-97&\eta^{1/2})k, (\half\aI+102\eta^{1/2})k, \\&(\half\aI+102\eta^{1/2})k, (\aIII-10\eta^{1/2})k, 4\eta^4 k\big),\hphantom{l}
\end{align*}
\vspace{-12mm}
\begin{align*}
\cK_{2}^{*}=\cK^*\big((\half\aI-97\eta^{1/2})k, (\half\aII-97\eta^{1/2})k, (\tfrac{3}{4}&\aIII-140\eta^{1/2})k, \\& 100\eta^{1/2}k, (\aIII-10\eta^{1/2})k, 4\eta^4 k\big).
\end{align*}
\end{itemize}
Furthermore, 
\vspace{-2mm}
\begin{itemize}
\item[(iv)] occurs only if $\aIII\leq \tfrac{3}{2}\aI+\half\aII+14\eta^{1/2}$ with $H, K\in \cH_1$ unless $\aII\geq\aI-\eta^{1/16}$; and
\item[(v)] and (vi) occur only if $\aIII\geq \tfrac{3}{2}\aI+\half\aII-10\eta^{1/2}$.
\end{itemize}
\end{thB}

\setlength{\parskip}{0.1in plus 0.025in minus 0.025in}
This result forms a partially strengthened analogue of the main technical result of the paper of Figaj and \L uczak~\cite{FL2008}. In that paper, Figaj and \L uczak considered a similar graph but on slightly more than $\max\{2\aI+\aII, \half\aI+\half\aII+\aIII\}k$ vertices and proved the existence of a connected-matching, whereas we consider a graph on slightly fewer vertices and prove the existence of either a monochromatic connected-matching or a particular structure.

\section{{Tools}}
\label{s:pre1}
In this section, we summarise results that we shall use later in our proofs beginning with some results on Hamiltonicity including Dirac's Theorem, which gives us a minimum-degree condition for Hamiltonicity:

\begin{theorem}[Dirac's Theorem~\cite{Dirac52}]
\label{dirac}
If~$G$ is a graph on~$n\geq3$ vertices such that every vertex has degree at least $\half n$, then~$G$ is Hamiltonian, that is, $G$ contains a cycle of length exactly $n$.
\end{theorem}

Observe then that, by Dirac's Theorem, any $c$-almost-complete graph on $n$ vertices is Hamiltonian, provided that $c\leq \half n-1$. Then, since almost-completeness is a hereditary property, we may prove the following corollary:

\begin{corollary}
\label{dirac1a}
If $G$ is a $c$-almost-complete graph on $n$ vertices, then, for any integer~$m$ such that $2c+2\leq m\leq n$, $G$ contains a cycle of length $m$.
\end{corollary}

Dirac's Theorem may be used to assert the existence of Hamiltonian paths in a given graph as follows:

\begin{corollary}
\label{dirac2}
If $G=(V,E)$ is a simple graph on~$n\geq4$ vertices such that every vertex has degree at least $\half n+1$, then any two vertices of~$G$ are joined by a Hamiltonian path.
\end{corollary}

For balanced bipartite graphs, we make use of the following result of Moon and Moser:

\begin{theorem}[\cite{moonmoser}]
\label{moonmoser}
If~$G=G[X,Y]$ is a simple bipartite graph on~$n$ vertices such that $|X|=|Y|=\half n$ and $d(x)+d(y)\geq \half n+1$ for every $xy\notin E(G)$, then~$G$ is Hamiltonian. 
\end{theorem}

Observe that, by the above, any $c$-almost-complete balanced bipartite graph on $n$ vertices is Hamiltonian, provided that $c\leq \tfrac{1}{4}n-\half$. Then, since almost-completeness is a hereditary property, we may prove the following corollary:

\begin{corollary}
\label{moonmoser2}
If $G=G[X,Y]$ is $c$-almost-complete bipartite graph, then, for any even integer $m$ such that $4c+2\leq m\leq 2\min\{|X|,|Y|\}$, $G$ contains a cycle on $m$ vertices.
\end{corollary}

For bipartite graphs which are not balanced, we make use of the Lemma below:
\begin{lemma}
\label{bp-dir}
If $G=G[X_1,X_2]$ is a simple bipartite graph on $n\geq 4$ vertices such that $|X_1|>|X_2|+1$ and every vertex in~$X_2$ has degree at least $\half n +1$, then any two vertices $x_1,x_2$ in $X_1$ such that $d(x_2)\geq 2$ are joined by a path which visits every vertex of~$X_2$.
\end{lemma}

\begin{proof}
Observe that $\half n+1=\half|X_1|+\half|X_2|+1=|X_1|-(\half|X_1|-\half|X_2|-1)$ so any pair of vertices in~$X_2$ have at least $|X_1|-(|X_1|-|X_2|-2)$ common neighbours and, thus, at least $|X_1|-(|X_1|-|X_2|)\geq |X_2| $ common neighbours distinct from~$x_1,x_2$. 

Then, ordering the vertices of~$X_2$ such that the first vertex is a neighbour of~$x_1$ and the last is a neighbour of~$x_2$, greedily construct the required path from~$x_1$ to~$x_2$.
\end{proof} 

For graphs with 
a few vertices of small degree, we make use of the following result of Chv\'atal:
\begin{theorem}[\cite{Chv72}]
\label{chv}
If $G$ is a simple graph on $n\geq3$ vertices with degree sequence $d_1\leq d_2 \leq \dots \leq d_n$ such that $$d_k\leq k \leq \frac{n}{2} \implies d_{n-k} \geq n-k,$$ then~$G$ is Hamiltonian.
\end{theorem}

We also make extensive use of the theorem of Erd\H{o}s and Gallai:

\begin{theorem}[\cite{ErdGall59}]
\label{th:eg}
Any graph on~$K$ vertices with at least~$\frac{1}{2}(m-1)(K-1)+1$ edges, where $3\leq m \leq K$, contains a cycle of length at least~$m$.
\end{theorem}

Observing that a cycle on $m$ vertices contains a connected-matching on at least~$m-1$ vertices, the following is an immediate consequence of the above.

\begin{corollary}
\label{l:eg}
For any graph~$G$ on~$K$ vertices and any~$m$ such that $3 \leq m \leq K$, if the average degree~$d(G)$ is at least $m$, then~$G$ contains a connected-matching on at least~$m$ vertices.
\end{corollary}

The following decomposition lemma~of Figaj and \L uczak~\cite{FL2008} also follows from the theorem of Erd\H{o}s and Gallai and is crucial in establishing the structure of a graph not containing large connected-matchings of the appropriate parities:

\begin{lemma}[{\cite[Lemma~9]{FL2008}}]
\label{l:decomp}
For any graph~$G$ on~$K$ vertices and any~$m$ such that $3\leq m \leq K$, if no odd component of~$G$ contains a matching on at least~$m$ vertices, then there exists a partition $V=V'\cup V''$ such that
\begin{itemize}
\item [(i)] $G[V']$ is bipartite;
\item [(ii)] every component of $G''=G[V'']$ is odd;
\item [(iii)] $G[V'']$ has at most $\half m |V(G'')|$ edges; and
\item [(iv)] there are no edges in $G[V',V'']$.
\end{itemize}
\end{lemma}

We recall two more results of Figaj and \L uczak. The first, is the main technical result from~\cite{FL2007}. The second from~\cite{FL2008}, allows us to deal with graphs with a \emph{hole}, that is, a subset $W\subseteq V(G)$ such that no  edge of~$G$ lies inside $W$. Note that both of these results can be immediately extended to multicoloured graphs:

\begin{theorem}[{\cite[Lemma~8]{FL2007}}]
\label{l:largeW}
For every $\alpha_{1}, \alpha_{2}, \alpha_{3}>0$ and~$\eta$ such that 
$0<\eta<0.002\min\{\alpha_1^2, \alpha_2^2,\alpha_3^2\}$, there exists $k_{\ref{l:largeW}}=k_{\ref{l:largeW}}(\aI,\aII,\aIII,\eta)$ such that the following holds:

For every $k>k_{\ref{l:largeW}}$ and every $(1-\eta^4)$-complete graph~$G$ on $$K\geq \half\left(\alpha_{1}+\alpha_{2}+\alpha_{3}+\max\{ \alpha_{1}, \alpha_{2}, \alpha_{3} \} +18 \eta^{1/2}\right)k$$ vertices, for every three-colouring of the edges of~$G$, there exists a colour $i\in\{1,2,3\}$ such that~$G_i$ contains a connected-matching on at least $(\alpha_{i}+\eta)k$ vertices.
\end{theorem}

\begin{lemma}[{\cite[Lemma~12]{FL2008}}]
\label{l:hole}
For every $\alpha, \beta>0$, $v\geq0$ and $\eta$ such that
$0<\eta<0.01 \min\{ \alpha,\beta\}$, 
there exists $k_{\ref{l:hole}}=k_{\ref{l:hole}}(\alpha,\beta,v,\eta)$ such that, for every $k>k_{\ref{l:hole}}$, the following holds:

Let $G=(V,E)$ be a graph obtained from a $(1-\eta^4)$-complete graph on $$ K\geq\half\left(\alpha+\beta+\max \{ 2v, \alpha, \beta \} + 6\eta^{1/2} \right)k $$
vertices by removing all edges contained within a subset $W \subseteq V$ of size at most~$vk$. Then, every two-multicolouring of the edges of~$G$ results in {either} a red connected-matching on at least $(\alpha+\eta)k$ vertices {or} a blue connected-matching on at least $(\beta+\eta)k$ vertices.
\end{lemma}

The following lemmas allow us to find large connected-matchings in almost-complete bipartite graphs:

\begin{lemma}[{\cite[Lemma~10]{FL2008}}]
\label{l:ten}
Let $G=G[V_1,V_2]$ be a bipartite graph with bipartition $(V_{1},V_{2})$, where $|V_{1}|\geq|V_{2}|$, which has at least $(1-\epsilon)|V_{1}||V_{2}|$ edges for some $\epsilon$ such that $0<\epsilon<0.01$. Then,~$G$ contains a connected-matching on at least $2(1-3\epsilon)|V_{2}|$ vertices.
\end{lemma}

Notice that, if~$G$ is a $(1-\epsilon)$-complete bipartite graph with bipartition $(V_1,V_2)$, then we may immediately apply the above to find a large connected-matching in~$G$.

\begin{lemma}
\label{l:eleven}
Let $G=G[V_1,V_2]$ be a bipartite graph with bipartition $(V_1,V_2)$. If $\ell$ is a positive integer such that $|V_1|\geq|V_2|\geq \ell$ and~$G$ is $a$-almost-complete for some $a$ such that $0<a/\ell<0.5$, then~$G$ contains a connected-matching on at least $2|V_2|-2a$ vertices.
\end{lemma}

\begin{proof}
Observe that~$G$ is $(1-a/\ell)$-complete. Therefore, since $a/\ell<0.5$,~$G$ is connected. Thus, it suffices to find a matching of the required size. Suppose that we have found a matching with vertex set~$M$ such that $|M|=2k$ for some $k<|V_2|-a$ and consider a vertex $v_2\in V_2\backslash M$. Since $G$ is $a$-almost-complete, $v_2$ has at least $|V_1|-a$ neighbours in~$|V_1|$ and thus at least one neighbour in $v_1\in V_1\backslash M$. Then, the edge $v_1v_2$ can be added to the matching and thus, by induction, we may obtain a matching on $2|V_2|-2a$ vertices.
\end{proof}

We also make use of the following Lemma~from~\cite{KoSiSk2}, which is an extension of the two-colour Ramsey result for even cycles and which allows us to find, in any almost-complete two-multicoloured graph on~$K$ vertices, either a large matching or a particular structure.

\begin{lemma}[\cite{KoSiSk2}]
\label{l:SkB}

For every~$\eta$ such that $0<\eta<10^{-20}$, there exists $k_{\ref{l:SkB}}=k_{\ref{l:SkB}}(\eta)$ such that, for every $k>k_{\ref{l:SkB}}$ and every $\alpha,\beta>0$ such that $\alpha \geq \beta \geq 100\eta^{1/2}\alpha$, if $K>(\alpha + \half\beta-\eta^{1/2}\beta)k$ and $G=(V,E)$ is a two-multicoloured $\beta \eta^2 k$-almost-complete graph on~$k$ vertices, then at least one of the following occurs:
\begin{itemize}
\item[(i)]~$G$ contains a red connected-matching on at least $(1+\eta^{1/2})\alpha k$ vertices;
\item[(ii)]~$G$ contains a blue connected-matching on at least $(1+\eta^{1/2})\beta k$ vertices;
\item[(iii)] the vertices of~$G$ can be partitioned into three sets $W$, $V'$, $V''$ such that
\begin{itemize}
\item[(a)] $|V'| < (1+\eta^{1/2})\alpha k$, 
$|V''|\leq \half(1+\eta^{1/2})\beta k$,
$|W|\leq \eta^{1/16} k$,
\item[(b)] $G_1[V']$ is $(1-\eta^{1/16})$-complete and $G_2[V']$ is $\eta^{1/16}$-sparse,
\item[(c)] $G_2[V',V'']$ is $(1-\eta^{1/16})$-complete and $G_1[V',V'']$ is $\eta^{1/16}$-sparse;
\end{itemize}
\item[(iv)] we have $\beta > (1-\eta^{1/8})\alpha$ and the vertices of~$G$ can be partitioned into sets $W$, $V'$ and $V''$ such that
\begin{itemize}
\item[(a)] $|V'| < (1+\eta^{1/2})\beta k$,
$|V''|\leq \half(1+\eta^{1/8})\alpha k$,
$|W|\leq \eta^{1/16} k$,
\item[(b)] $G_2[V']$ is $(1-\eta^{1/16})$-complete and $G_1[V']$ is $\eta^{1/16}$-sparse, 
\item[(c)] $G_1[V',V'']$ is $(1-\eta^{1/16})$-complete and $G_2[V',V'']$ is $\eta^{1/16}$-sparse.
\end{itemize}
\end{itemize}
Furthermore, if $\alpha+ \half\beta \geq 2(1+\eta^{1/2})\beta$, then we can replace (i) with
\begin{itemize}
\item[(i')]~$G$ contains a red odd connected-matching on $(1+\eta^{1/2})\alpha k$ vertices.
\end{itemize}
\end{lemma}

We also make use of the following corollary of Lemma~\ref{l:SkB}:

\begin{corollary} 
\label{c:SkBe}
For every $0<\epsilon <10^{-12}$, there exists $k_{\ref{c:SkBe}}=k_{\ref{c:SkBe}}(\epsilon)$ such that, for every $k\geq k_{\ref{c:SkBe}}$, if $K>(1-\epsilon)k$ and $G=(V,E)$ is a two-multicoloured $\tfrac{27}{8}\epsilon^4 k$-almost-complete graph, then~$G$ contains at least one of the following:
\begin{itemize}
\item[(i)] a red connected-matching on $(\tfrac{2}{3}-7\epsilon^{1/8})k$ vertices;
\item[(ii)] a blue connected-matching on $(\tfrac{2}{3}-7\epsilon^{1/8})k$ vertices.
\end{itemize}
\end{corollary}

\begin{proof}
Setting $\eta=(\tfrac{3}{2}\epsilon)^2$, $\alpha=\beta=2/3$, provided $k\geq k_{\ref{l:SkB}}(\eta)$, we may apply Lemma~\ref{l:SkB}, which results in at least one of the following occurring:
\begin{itemize}
\item [(i)]~$G$ contains a red connected-matching on at least $(\tfrac{2}{3}+\epsilon)k$ vertices;
\item [(ii)]~$G$ contains a blue connected-matching on at least $(\tfrac{2}{3}+\epsilon)k$ vertices;
\item [(iii)] the vertices of~$G$ can be partitioned into three sets $W$, $V'$, $V''$ such that
\begin{itemize}
\item[(a)] $|V'| < (\tfrac{2}{3}+\epsilon)k$,
$|V''|\leq (\tfrac{1}{3}+\half\epsilon)k$,
$|W|\leq (\tfrac{3}{2}\epsilon)^{1/8} k$,
\item[(b)] $G_1[V']$ is $(1-(\tfrac{3}{2}\epsilon)^{1/8})$-complete and $G_2[V']$ is $(\tfrac{3}{2}\epsilon)^{1/8}$-sparse,
\item[(c)]$G_2[V',V'']$ is $(1-(\tfrac{3}{2}\epsilon)^{1/8})$-complete and $G_1[V',V'']$ is $(\tfrac{3}{2}\epsilon)^{1/8}$-sparse;
\end{itemize}
\item [(iv)] the vertices of~$G$ can be partitioned into three sets $W$, $V'$, $V''$ such that
\begin{itemize}
\item[(a)] $|V'| < (\tfrac{2}{3}+\epsilon)k$,
$|V''|\leq (\tfrac{1}{3}+\half\epsilon)k$,
$|W|\leq (\tfrac{3}{2}\epsilon)^{1/8} k$,
\item[(b)] $G_2[V']$ is $(1-(\tfrac{3}{2}\epsilon)^{1/8})$-complete and $G_1[V']$ is $(\tfrac{3}{2}\epsilon)^{1/8}$-sparse,
\item[(c)]$G_1[V',V'']$ is $(1-(\tfrac{3}{2}\epsilon)^{1/8})$-complete and $G_2[V',V'']$ is $(\tfrac{3}{2}\epsilon)^{1/8}$-sparse.
\end{itemize}
\end{itemize}

In the first two cases, the result is immediate.

In the third case, recalling that $K=|V'|+|V''|+|W|$, simple algebra yields $|V'|\geq (\tfrac{2}{3}-2\epsilon^{1/8})$ and $|V''|\geq (\tfrac{1}{3}-2\epsilon^{1/8}).$
Then, since $G_2[V',V'']$ is $(1-(\tfrac{3}{2}\epsilon)^{1/8})$-complete, there are at least 
$(1-(\tfrac{3}{2}\epsilon)^{1/8})(\tfrac{1}{3}-2\epsilon^{1/8})(\tfrac{2}{3}-2\epsilon^{1/8})k^2$ edges in $G_2[V',V'']$. Thus, by Lemma~\ref{l:ten}, $G[V',V'']$ contains a blue connected-matching on at least 
$2(1-3(\tfrac{3}{2}\epsilon)^{1/8})(\tfrac{1}{3}-2\epsilon^{1/8})k\geq(\tfrac{2}{3}-7\epsilon^{1/8})k$ vertices. 
In the fourth case, exchanging the roles of red and blue, an identical argument yields a red connected-matching on $(\tfrac{2}{3}-7\epsilon^{1/8})k$ vertices.
\end{proof}

It is a well-known fact that either a graph is connected or its complement is. We now prove
two
simple extensions of this fact for two-coloured almost-complete graphs, all of which can be immediately extended to two-multicoloured almost-complete graphs. 

\begin{lemma}\label{l:dgf0}
For every $\eta$ such that $0<\eta<1/3$ and every $K\geq 1/\eta$, if $G=(V,E)$ is a two-coloured $(1-\eta)$-complete graph on~$K$ vertices and~$F$ is its largest monochromatic component, then $|F|\geq (1-3\eta)K$.
\end{lemma}

\begin{proof}
If the largest monochromatic (say, red) component in~$G$ has at least $(1-3\eta)K$ vertices, then we are done. Otherwise, we may partition the vertices of~$G$ into sets~$A$ and~$B$ such that $|A|,|B|\geq3\eta K\geq 2$ such that there are no red edges between~$A$ and~$B$. Since~$G$ is $(1-\eta)$-complete, any two vertices in~$A$ have a common neighbour in~$B$, and any two vertices in~$B$ have a common neighbour in~$A$. Thus, $A\cup B$ forms a single blue component.
\end{proof}

The following lemmas form analogues of the above, the first concerns the structure of two-coloured almost-complete graphs with one hole and the second concerns the structure of two-coloured almost-complete graphs with two holes, that is, bipartite graphs. 

\begin{lemma}
\label{l:dgf1} 

For every $\eta$ such that $0<\eta<1/20$ and every $K\geq 1/\eta$, the following holds. For~$W$, any subset of~$V$ such that $|W|,|V\backslash W|\geq 4\eta^{1/2}K$, let $G_{W}=(V,E)$ be a two-coloured graph obtained from~$G$, a $(1-\eta)$-complete graph on~$K$ vertices with vertex set~$V$ by removing all edges contained entirely within~$W$. Let~$F$ be the largest monochromatic component of $G_W$ and define the following two sets:
\begin{align*} 
 W_{r}&= \{\text{$w \in W$ : $w$ has red edges to all but at most  $3\eta^{1/2} K$ vertices in $V \backslash W$}\};
\\ 
 W_{b}&=\{ w \in W : w \text{ has blue edges to all but at most } 3\eta^{1/2} K \text{ vertices in } V \backslash W\}.
\end{align*}
Then, at least one of the following holds:
\begin{itemize}
\item [(i)] $|F|\geq (1-2\eta^{1/2})K$; 
\item [(ii)] $|W_{r}|,|W_{b}|>0$.
\end{itemize}
\end{lemma}

\begin{proof}
Consider $G[V\backslash W]$. 
Since~$G$ is $(1-\eta)$-complete, $|V\backslash W|\geq 4\eta^{1/2}K$ and $\eta<1/20$, we see that every vertex in $G[V\backslash W]$ has degree at least $|V\backslash W|-\eta(K-1)\geq(1-\tfrac{1}{4}\eta^{1/2})(|V\backslash W|-1)$, that is, $G[V\backslash W]$ is $(1-\tfrac{1}{4}\eta^{1/2})$-complete. Thus, provided $4\eta^{1/2}K \geq 1/(\tfrac{1}{4}\eta^{1/2})$, that is, provided $K\geq 1/\eta$, we can apply Lemma~\ref{l:dgf0}, which tells us that the largest monochromatic component in $G[V\backslash W]$ contains at least $|V\backslash W|-\eta^{1/2} K$ vertices. We assume, without loss of generality, that this large component is red and call it~$R$.

Now,~$G$ is $(1-\eta)$-complete so either every vertex in~$W$ has a red edge to~$R$ (giving a monochromatic component of the required size) or there is a vertex $w\in W$ with at least $|R|-2\eta K$ {blue neighbours} in~$R$, that is, a vertex $w\in W_{b}$. Denote by~$B$ the set of $u\in R$ such that $uw$ is blue. Then, $|B|\geq |V\backslash W|-2\eta^{1/2} K$ and either every point in~$W$ has a blue edge to~$B$, giving a blue component of size at least $|B\cup W|>(1-2\eta^{1/2})K$, or there is a vertex $w_1\in W_{r}$.
\end{proof}

\section{Proof of the stability result -- Part I}
\label{s:stabp}

We begin with the case where $\aI\geq\aII, \aIII$ and wish to prove that any three-multicoloured graph on slightly fewer than $(2\aI+\aII)k$ vertices with sufficiently large minimum degree will contain a red connected-matching on at least~$\alpha_{1}k$ vertices, a blue connected-matching on at least~$\alpha_{2}k$ vertices or a green odd connected-matching on at least~$\alpha_{3}k$ vertices, or will have a particular structure.

Thus, given $\aI, \aII, \aIII$ such that $\aI \geq \aII, \aIII$, we choose  
$$\eta<\eta_{B1}=\min \left\{10^{-50}, 
\frac{\aII}{10^{20}}, \bigg(\frac{\aII}{120}\bigg)^8,  \bigg(\frac{\aII}{800\aI}\bigg)^2, \frac{\aIII}{10^4}
\right\}$$
and consider $G=(V,E)$, a $(1-\eta^4)$-complete graph on~$K\geq 72/\eta$ vertices, where
$$(2\aI + \aII - \eta)k \leq K \leq (2\aI + \aII - \half\eta)k$$ for some integer $k>k_{B1}$
, where $k_{B1}=k_{B1}(\aI,\aII,\aIII,\eta)$ will be defined implicitly during the course of this section, in that, on a finite number of occasions, we will need to bound~$k$ below in order to apply results from Section~\ref{s:pre1}.

Note that, since $\aI \geq \aII, \aIII$, the largest forbidden connected-matching is red and need not be odd. Note that, by scaling, we may assume that $\aI\leq1$. Notice, then, that $G$ is $3\eta^4k$-almost-complete and, thus, for any $X\subset V$, $G[X]$ is also $3\eta^4k$-almost-complete. 

In this section, we seek to prove that~$G$ contains at least one of the following:
\begin{itemize}
\item [(i)]	  a red connected-matching on at least $\aI k$ vertices;
\item [(ii)]  a blue connected-matching on at least $\alpha_{2}k$ vertices;
\item [(iii)]  a green odd connected-matching on at least $\alpha_{3}k$ vertices; 
\item [(iv)]  
disjoint subsets of vertices $X$ and $Y$ such that $G[X]$ contains a two-coloured spanning subgraph~$H$ from $\cH_{1}\cup\cH_{2}$ and $G[Y]$ contains a two-coloured spanning subgraph $K$ from $\cH_{1}\cup\cH_{2}$, where
\vspace{-2mm}
\begin{align*}
\cH_1=&\left((\aI-2\eta^{1/32})k,(\half\aII-2\eta^{1/32})k,3\eta^4 k,\eta^{1/32},\text{red},\text{blue}\right),\text{ } 
\\ \cH_2=&\left((\aII-2\eta^{1/32})k,(\half\aI-2\eta^{1/32})k,3\eta^4 k,\eta^{1/32},\text{blue},\text{red}\right).
\end{align*}
\end{itemize}

We begin by noting that, if~$G$ has a green odd connected-matching on at least $\alpha_{3}k$ vertices, then we are done. Thus, provided that $\aIII k\geq 3$, since $\aIII \leq \aI$, by Lemma~\ref{l:decomp}, we may partition the vertices of~$G$ into $W, X$ and $Y$ such that
\begin{itemize}
\item[(i)] $X$ and $Y$ contain only red and blue edges; 
\item[(ii)] $W$ has at most $\half \aIII k|W|\leq \half \aI k|W|$ green edges; and 
\item[(iii)] there are no green edges between $W$ and $X\cup Y$. 
\end{itemize}

\begin{figure}[!h]
\vspace{-2mm}
\centering
\includegraphics[width=64mm, page=6]{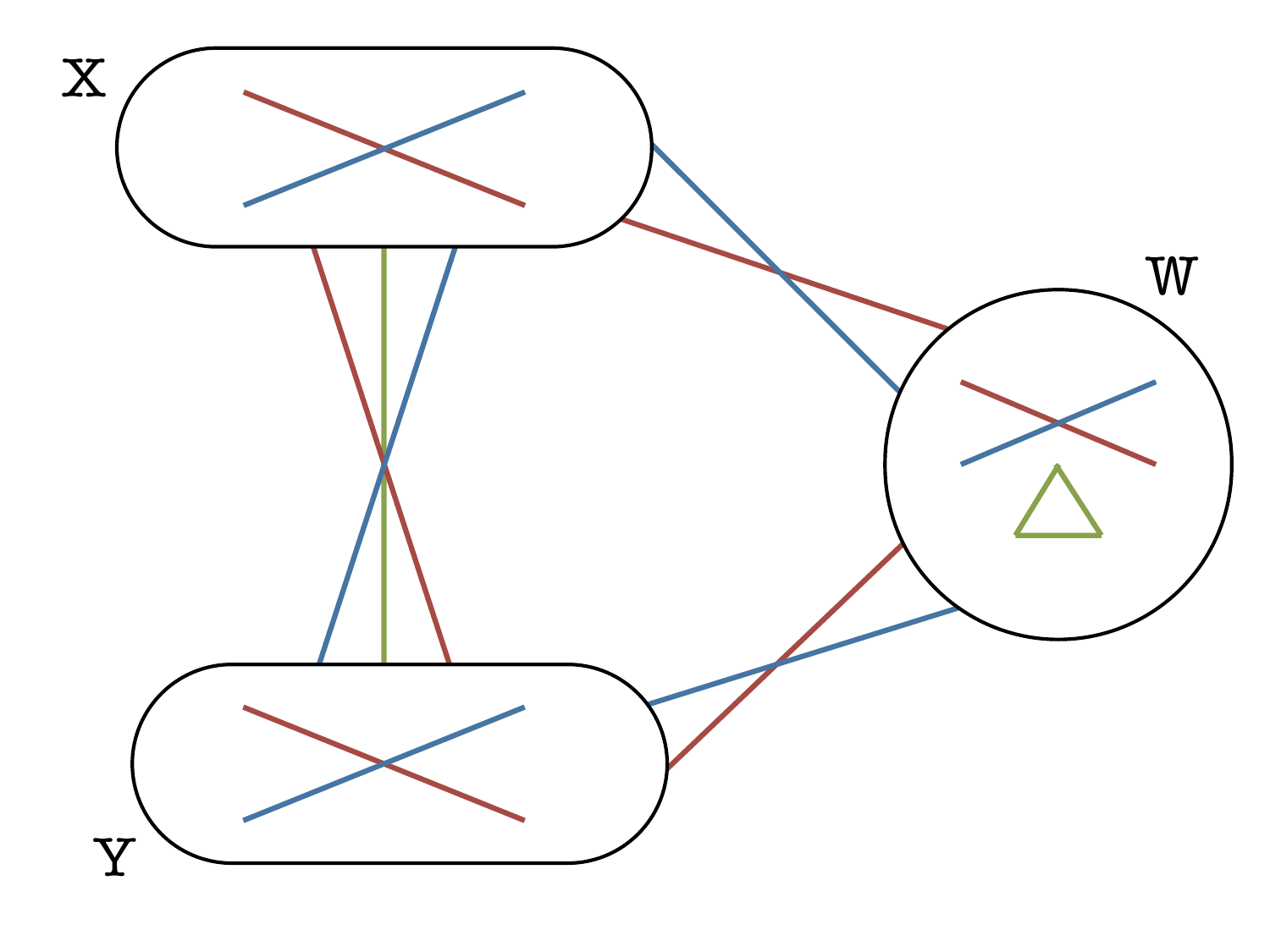}
\vspace{-3mm}\caption{Decomposition of the green graph.}
  
\end{figure}

By this decomposition, writing $w$ for $|W|/k$ and noticing that $e(G[X,Y])$ is maximised when~$X$ and~$Y$ are equal in size, we have
\begin{equation} 
\label{ub8}
e(G_{3}) = e(G_{3}[W])+e(G_{3}[X,Y]) \leq \half \aI wk^{2}+ \tfrac{1}{4} (2\aI +\aII-w)^{2}k^{2}. 
\end{equation}
Now, consider the average degree of the green graph~$d(G_3)$. Note that, since $K\leq 3k$, $\eta^4K \leq \eta k$. Thus, the average number of {missing} edges at each vertex is at most~$\eta k$. Thus, if $d(G_{3})~\leq~(\aI -2\eta)k$, then either $d(G_{1}) \geq \aI k$ or $d(G_{2}) \geq \alpha_{2}k$. If $d(G_1)\geq \aI k$, then, by Corollary~\ref{l:eg},~$G$ contains a red connected-matching on at least $\aI k$ vertices. Similarly, if $d(G_2)\geq \aII k$, then~$G$ contains a blue connected-matching on at least~$\aII k$ vertices. Thus, we may assume that $d(G_{3}) >(\aI -2\eta)k$, in which case
\begin{equation}
\label{lb8} 
e(G_{3}) >\half (\aI -2\eta)(2\aI +\alpha_{2}-\eta)k^{2}.
\end{equation}
Comparing~(\ref{ub8}) and~(\ref{lb8}), we obtain
\begin{equation*}
0  < w^{2} + w (-2\aI-2\aII) + 2\aI\aII+\aII^2+\eta(10\aI+4\aII).\end{equation*}
Since $1\geq \aI \geq \aII$ and $\eta<\aI/100$, this results in two cases:
\begin{itemize}
\item[(A)] $w > \aI +\aII + \sqrt{\aI^2 -(10\aI+4\aII)\eta}> 2\aI +\alpha_{2}-10\eta$; 
\item[(B)] $w < \aI +\aII - \sqrt{\aI^2 -(10\aI+4\aII)\eta}< \aII+10\eta$.
\end{itemize}

\hypertarget{CaseAht}{}

In {\bf Case A}, almost all the vertices of $G$ are contained in the odd green component. Since $\eta<\eta_{B1}
$, we have $2\aI +\alpha_{2}-10\eta>\aI +\half\aII+\half\aIII+9\eta^{1/2}$
and may apply Theorem~\ref{l:largeW} to obtain a connected-matching on $(\alpha_{i}+\eta)k$ vertices, provided that $k>k_{\ref{l:largeW}}(\aI,\aII,\aIII,\eta)$. Furthermore, the nature of the decomposition means that, if the connected-matching is green, then it is odd, thus completing the case.

Moving onto {\bf Case B}, we assume that $|X|\geq|Y|$ and consider the subgraph $G_1[X\cup W] \cup G_2[X\cup W]$, that is, the subgraph of~$G$ on $X\cup W$ induced by the red and blue edges.

Since $\eta\leq\eta_{B1}$, provided that $k>k_{\ref{l:hole}}(\aI,\aII,w,\eta)$, by Lemma~\ref{l:hole} (regarding $W$ as the hole), if 
\begin{equation*}
|X|+|W| \geq \half\left( \aI+\aII+\max \left\{ 2w, {\aI} \right\}+6\eta^{1/2} \right)k,
\end{equation*}
then we can obtain a red connected-matching on at least $(\aI+\eta)k$ vertices or a blue connected-matching on at least $(\alpha_{2}+\eta)k$ vertices.

We may therefore assume that
\begin{equation}
\label{e8}
|X|+|W|< \half\left( \aI+\aII+\max \left\{ 2w, {\aI} \right\}+6\eta^{1/2} \right)k.
\end{equation}
Since $K=|X|+|Y|+|W|$ and $|X|\geq|Y|$,
\begin{equation}
\label{e9}
|X|+|W|\geq \frac{K-|W|}{2}+|W|=\frac{(2\aI +\alpha_{2}-\eta)k+wk}{2} =  \left(\aI +\frac{\aII}{2}-\frac{\eta}{2}+\frac{w}{2} \right)k.
\end{equation}
Now, suppose that $w\leq \half\aI$, in which case, by~(\ref{e8}) and~(\ref{e9}), we have
\begin{align*}
(\aI+\half\aII-\tfrac{1}{2}\eta+\half w)k\leq|W|+|X| & < (\aI+\half\aII+3\eta^{1/2})k, \end{align*}
which results in a contradiction, unless $w<6\eta^{1/2}+\eta$, in which case almost all the vertices of~$G$ belong to~$X\cup Y$.
Eliminating~$|W|$, we find that
$$(\aI+\half\aII-4\eta^{1/2})k \leq |X| < (\aI+\half\aII+3\eta^{1/2})k.$$
Since $\eta\leq \eta_{B1}$, provided $k>k_{\ref{l:SkB}}(\eta^{1/2})$, we may apply Lemma~\ref{l:SkB} (with $\alpha=\aI, \beta=\aII$) to find that~$G[X]$ contains either a red connected-matching on at least~$\aI k$ vertices, a blue connected-matching on at least~$\alpha_{2}k$ vertices {or} 
a subgraph $H_1$ from $\cH_1\cup\cH_2$, where
\begin{align*}
\cH_1=&\cH\left((\aI-2\eta^{1/32})k,(\half\aII-2\eta^{1/32})k,3\eta^4 k,\eta^{1/32},\text{red},\text{blue}\right),\text{ } 
\\ \cH_2=&\cH\left((\aII-2\eta^{1/32})k,(\half\aI-2\eta^{1/32})k,3\eta^4 k,\eta^{1/32},\text{blue},\text{red}\right).
\end{align*}
Furthermore, unless $\aII\geq\aI-\eta^{1/16}$, $H_1\in\cH_1$.

Now, consider~$Y$. Since $|G|=|X|+|Y|+|W|$ and $|X|\geq|Y|$, we obtain
$$(\aI+\half\aII-10\eta^{1/2})k \leq |Y| < (\aI+\half\aII+3\eta^{1/2})k.$$
Then, provided $k\geq k_{\ref{l:SkB}}(\eta^{1/2})$, we may apply Lemma~\ref{l:SkB} to~$G[Y]$ to find that $G[Y]$ contains either a red connected-matching on at least~$\aI k$ vertices, a blue connected-matching on at least~$\alpha_{2}k$ vertices or a two-coloured subgraph $H_2$ belonging to $\cH_1 \cup \cH_2$. Furthermore, unless $\aII\geq\aI-\eta^{1/16}$, $H_2\in\cH_1$ which would be sufficient to complete the proof in this case. 

We may, therefore, assume that $\half\aI<w<\alpha_{2}+10\eta$, in which case, from~(\ref{e8}) and~(\ref{e9}), we have
\begin{align*}
\left( \aI+\half\aII+\half w-\tfrac{1}{2}\eta \right)k\leq |W|+|X| & < \left(\half\aI+\half\aII+w+3\eta^{1/2} \right)k.
\end{align*}
Since $\alpha_{2}>w-10\eta$, we obtain $|W|+|X|\geq (\aI+w-6\eta)k$.

Then, since $K=|X|+|Y|+|W|$ and $|X|\geq|Y|$, it follows that 
\begin{equation}
\label{e22}
\left.
\begin{aligned}
\,\,\,\quad\quad\quad\quad\quad\quad\quad\quad(\aI-\eta^{1/2})k \leq |X| < & \left(\half\aI+\half\aII+3\eta^{1/2} \right)k,\quad\quad\quad\quad\quad\,\,\,\,\,\\
(\aI-4\eta^{1/2})k  < |Y| < & \left(\half\aI+\half\aII+3\eta^{1/2} \right)k,\\
(\aI-8\eta^{1/2})k \leq |W| <  & \left(\half\aI+\half\aII+10\eta^{1/2} \right)k. 
\end{aligned}
\right\}\!
\end{equation}
The bounds for $|X|$ in (\ref{e22}) lead to a contradiction unless 
$\aI-\alpha_{2}\leq8\eta^{1/2}$. Therefore, we may only concern ourselves with the case where~$\aII\leq\aI\leq\alpha_{2}+8\eta^{1/2}$ and, therefore,~$X,Y$ and~$W$ each contain about a third of the vertices of~$G$.

Recall that there are no green edges contained within~$X$ or~$Y$. Then, since~$G$ is $3\eta^4k$-almost-complete, provided $k>k_{\ref{c:SkBe}}(\eta)$, we may apply Corollary~\ref{c:SkBe} 
to $G[X], G[Y]$ to find that each contains a monochromatic connected-matching on at least $(\tfrac{2}{3}\aI-8\eta^{1/8})k$ vertices. Thus, provided $\eta<(\aI/120)^8$, we may assume that each of $X$ and $Y$ contain a monochromatic connected-matching on at least $\tfrac{3}{5}\aI k$ vertices. Referring to these matchings as $M_1\subseteq G[X]$ and $M_2\subseteq G[Y]$, we consider three subcases:
\begin{itemize}
\item[(B.i)] \rm$M_{1}$ and $M_{2}$ are both red;
\item[(B.ii)]  \rm$M_{1}$ and $M_{2}$ are both blue;
\item[(B.iii)] \rm$M_{1}$ is red and $M_{2}$ is blue.
\end{itemize}

\label{Bistart}
In {\bf Case B.i}, suppose that there exists $r\in M_{1}, w\in W$ and $s\in M_{2}$ such that $rw$ and $ws$ are red. This would give a red connected-matching on at least $\tfrac{6}{5}\aI k$ vertices. Therefore, we may assume that every $w\in W$ has either~$rw$ blue or missing for all $r\in M_{1}$, or~$ws$ blue or missing for all $s\in M_{2}$. 

Thus, we may partition $W$ into $W_{1} \cup W_{2}$, where $W_1,W_2$ are defined as follows: 
\begin{align*}
W_{1}&=\{w\in W \text{ such that } w \text{ has no red edges to } M_1\};\\
W_{2}&=\{w\in W \text{ such that } w \text{ has no red edges to } M_2\}.
\end{align*}
\begin{figure}[!h]
\centering
\vspace{-6mm}
\includegraphics[width=67mm, page=24]{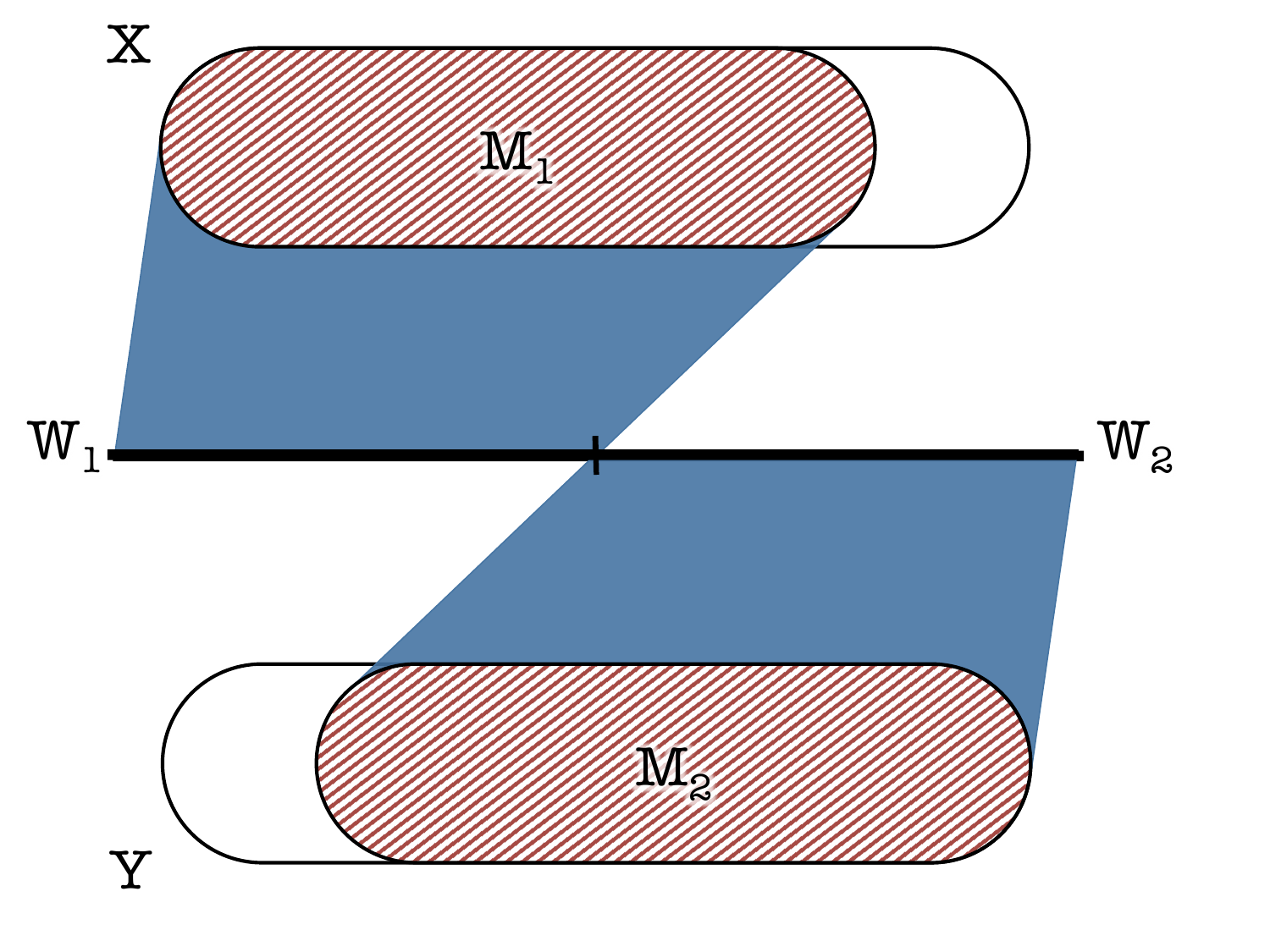}
\vspace{-3mm}\caption{Partition of $W$ into $W_1\cup W_2$.}
  
\end{figure}

Suppose that $|W_1|\geq (\half\aII +6\eta^4)k$. Then, since~$G$ is $3\eta^4 k$-almost-complete, so is $G[W,M_1]$. Since $\eta<(\aII/100)^2 \leq (\aII/60)^{1/4}$, $|M_1|\geq \frac{3}{5}\aI k \geq \frac{\aII}{2} k + 6\eta^4k$ and we may apply Lemma~\ref{l:eleven} with $\ell=(\half\aII +6\eta^4) k$ and $a=3\eta^4 k$ to give a blue connected-matching on at least~$\aII k$ vertices. The result is the same in the event that $|W_2|\geq(\half\aII +6\eta^4)k$.
Therefore, we may assume that $|W_{1}|, |W_{2}|\leq (\half\aII+6\eta^4) k$. In that case, we have $|W_{1}|=|W|-|W_{2}|\geq (\frac{\aI}{2}-9\eta^{1/2})k$ and, likewise, $|W_{2}|\geq (\frac{\aI}{2}-9\eta^{1/2})k$. Thus, since $\eta<(\aI/100)^2$, Lemma~\ref{l:eleven} gives a blue connected-matching on at least $({\aII}-20\eta^{1/2})k$ vertices in each of $G[M_1,W_1]$ and $G[M_2,W_2]$.

Then, suppose there exists a blue edge~$rw$ for $r\in M_{1}$, $w \in W_{2}$. This would connect these two blue connected-matchings, giving one on at least $(2\aI-40\eta^{1/2})k\geq \aII k$ vertices. Thus, all edges present in~$G[M_{1},W_{2}]$ are coloured exclusively red. By the same argument, so are all edges present in $G[M_2,W_1]$. 

Now, choose any set~$R_1$ of $10\eta^{1/2} k$ of the edges from the matching~$M_1$, let $M_1'=M\backslash R_1$ and consider $G[V(M_1'),W_2]$. Since $\eta<(\aI/100)^2$, we have $|V(M_1')|\geq (\half\aI  - 9\eta^{1/2})k$ and thus may apply Lemma~\ref{l:eleven} to $G[V(M_1'),W_2]$ to obtain a collection $R_2$ of edges from $G[V(M_1'),W_2]$ which form a red connected-matching on at least $(\aI-20\eta^{1/2})k$ vertices. Since $R_1$ and $R_2$ do not share any vertices but do belong to the same red-component of~$G$, the collection of edges $R_1\cup R_2$ forms a red connected-matching on at least $\aI k$ vertices, completing this case.

In {\bf Case B.ii}, exchanging the roles of red and blue (and where necessary $\aI$ and $\aII$), the proof follows the same steps as in Case B.i above.

Finally, we consider {\bf Case B.iii}. Without loss of generality, consider the case where~$M_{1}$ is red and~$M_{2}$ is blue. Since $\eta<(\aI/100)^2$, by (\ref{e22}), we have $|X|+|W|, |Y|+|W|\geq \half K$, so $G[X\cup W]$ and $G[Y\cup W]$ are each $(1-2\eta^4)$-complete.
Additionally, $|W|,|V\backslash W|\geq 4(2\eta^{4})^{1/2}|X\cup W|$. Thus, provided that $\half K\geq 1/2\eta^4$, we may apply Lemma~\ref{l:dgf1} separately to $G[X\cup W]$ and $G[Y\cup W]$ (regarding $W$ as the hole in each case) with the result being that at least one of the following occurs:
\begin{itemize}
\item[(a)] $X\cup W$ has a connected red 
component~$F$ on at least $|X \cup W| - \eta k$ vertices; 
\item[(b)] $X\cup W$ has a connected blue
 component~$F$ on at least $|X \cup W| - \eta k$ vertices; 
\item[(c)] $Y\cup W$ has a connected red
 component~$F$ on at least $|Y \cup W| - \eta k$ vertices;
\item[(d)] $Y\cup W$ has a connected blue
component~$F$ on at least $|Y \cup W| - \eta k$ vertices;
\item[(e)] there exist points $w_{1}, w_{2}, w_{3}, w_{4} \in W$ such that the following hold:
\begin{itemize}
\item[(i)] $w_1$ has red edges to all but at most $\eta k$ vertices in $X$,
\item[(ii)] $w_2$ has blue edges to all but at most $\eta k$ vertices in $X$,
\item[(iii)] $w_3$ has red edges to all but at most $\eta k$ vertices in $Y$,
\item[(iv)] $w_4$ has blue edges to all but at most $\eta k$ vertices in $Y$.
\end{itemize}
\end{itemize}

In case (a), we discard from~$W$ the, at most~$\eta k$, vertices not contained in~$F$ and consider $G[W,Y]$. Either there are at least $\tfrac{1}{5}\aI k$ mutually independent red edges present in $G[W,Y]$ (which can be used to augment $M_1$) or we may obtain $W'\subset W$, $Y' \subset Y$ with $|W'|,|Y'| \geq (\tfrac{4}{5} \aI  - 10\eta^{1/2})k$ such that all the edges present in $G[W',Y']$ are coloured exclusively blue. Notice that, since~$G$ is $3\eta^4 k$-almost-complete, so is $G_2[W',Y']$ and, since $\eta<(\aI/100)^2$, we may apply Lemma~\ref{l:eleven} (with $a=3\eta^4k$ and $\ell=\tfrac{3}{5}\aI k$) to obtain a blue connected-matching on at least $\aI k\geq \aII k$ vertices.

In case (b), suppose there exists a blue edge in $G[M_2,F]$. Then, at least $|M_2\cup W|-\eta k$ of the vertices of $M_2\cup W$ would belong to the same blue component in~$G$. We could then consider $G[W,X]$ and, by the same argument used in case (a) (with the roles of the colours reversed), find either a red connected-matching on at least $\aI k$ vertices or a blue connected-matching on at least~$\aII k$ vertices. Thus, we may instead, after discarding at most $\eta k$ vertices from $W$, assume that all edges present in $G[M_2,W]$ are coloured exclusively red and apply Lemma~\ref{l:eleven} to obtain a red connected-matching on at least~$\aI k$ vertices.

In case (c), the same argument as given in case (b) gives either a red connected-matching on at least~$\aI k$ vertices or a blue connected-matching on at least~$\aII k$ vertices.

In case (d), the same argument as given in case (a) gives gives either a red connected-matching on at least~$\aI k$ vertices or a blue connected-matching on at least~$\aII k$ vertices.

In case (e), there exist points $w_{1}, w_{2}, w_{3}, w_{4} \in W$ such that $w_1$ has red edges to all but at most $\eta k$ vertices in $X$, $w_2$ has blue edges to all but at most $\eta k$ vertices in $X$, $w_3$ has red edges to all but at most $\eta k$ vertices in $Y$, and $w_4$ has blue edges to all but at most $\eta k$ vertices in $Y$.
\begin{figure}[!h]
\centering
\vspace{2mm}
\includegraphics[width=70mm, page=26]{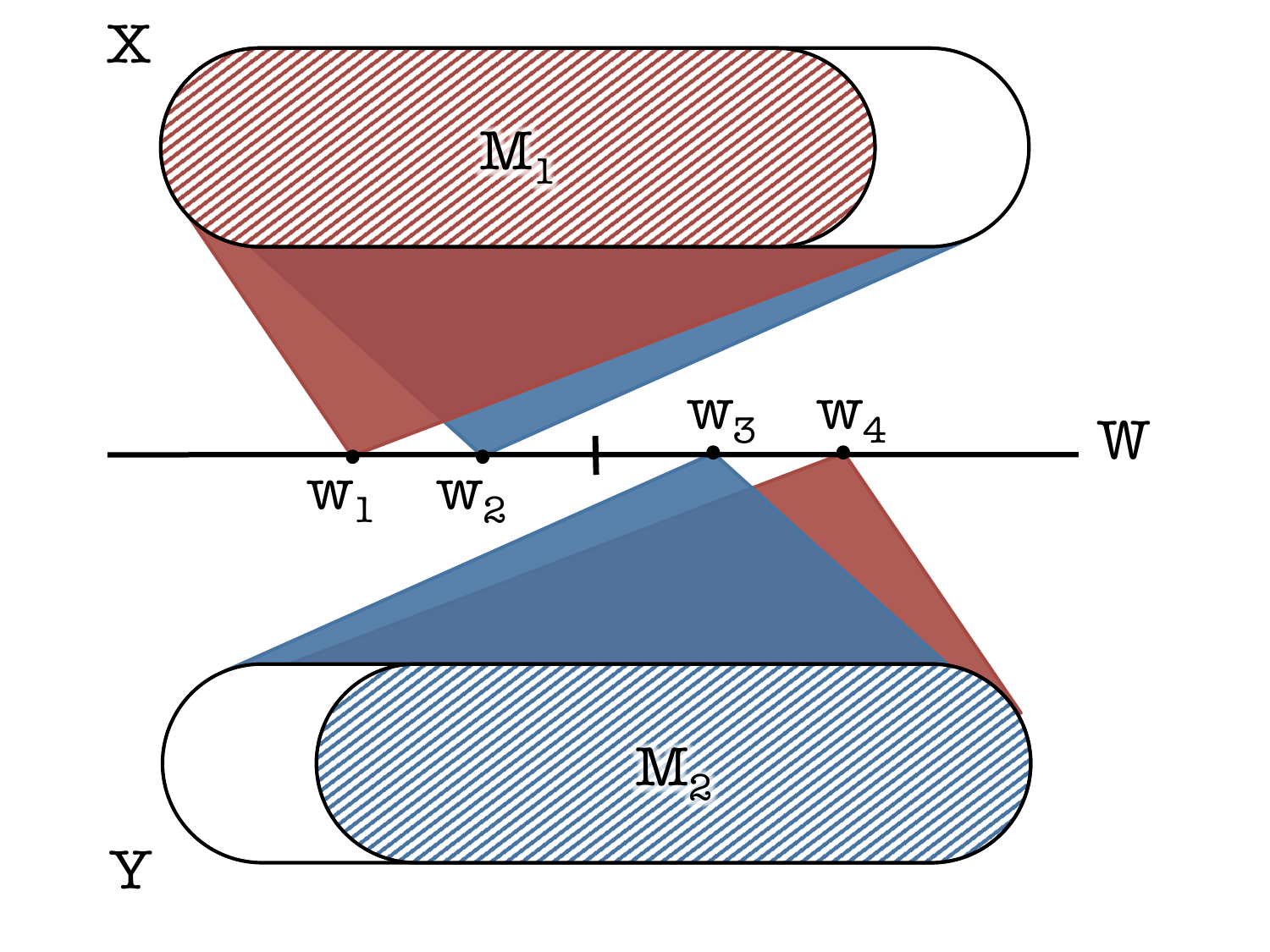}
\vspace{-3mm}\caption{Vertices $w_1,w_2,w_3$ and $w_4$ in case (e).}
\vspace{-3mm}
  
\end{figure}

Thus, defining
\begin{align*}
X_{S}&=\{x\in X \text{ such that } xw_1 \text{ is red and } xw_2 \text{ is blue}\},\\
Y_{S}&=\{y\in Y \text{ such that } yw_3 \text{ is red and } yw_4 \text{ is blue}\},
\end{align*}
by (\ref{e22}), we have $|X_S|,|Y_S|\geq(\aI-5\eta^{1/2})k$. Suppose there exists $x\in X_S$, $w\in W$,~$y\in Y_S$ such that~$xw$ and~$wy$ are red. In that case, $X_S\cup Y_S$ belong to the same red component of~$G$. Recall that $M_1$ contains a red matching on $\tfrac{3}{5} \aI k$ vertices and consider $G[W,Y_S]$. Either we can find $\tfrac{1}{5} \aI k$ mutually independent red edges in $G[W,Y_S]$ (which together with $M_1$ give a red connected-matching on at least~$\aI k$ vertices) or we may obtain $W'\subset W$, $Y' \subset Y_S$ with $|W'|,|Y'| \geq (\tfrac{4}{5}\aI  - 10\eta^{1/2})k$ such that all the edges present in $G[W',Y']$ are coloured exclusively blue. Then, as in case (a), we may apply Lemma~\ref{l:eleven} to obtain a blue connected-matching on at least $\aI k\geq \aII k$ vertices.

Thus, we assume no such triple exists and, similarly, we may assume there exists no triple $x\in X_S$, $w\in W$, $y\in Y_S$ such that $xw$ and $wy$ are blue. Thus, we may partition~$W$ into $W_1 \cup W_2$ such that all edges present in $G[W_1,X_S]$ and $G[W_2,Y_S]$ are coloured exclusively red and all edges present in $G[W_1,Y_S]$ and $G[W_2,X_S]$ are coloured exclusively blue. Thus, we may assume that $|W_1|,|W_2|\leq(\half\aI+\eta^{1/2})k$ (else Lemma~\ref{l:eleven} could be used to give a red connected-matching on at least~$\aI k$ vertices) and therefore also that $|W_1|,|W_2|\geq(\half\aI-9\eta^{1/2})k$, in which case, the same argument as in the last paragraph of case (a) gives a red connected-matching on~$\aI k$ vertices.

This concludes Case B and, thus, Part I of the proof of Theorem~\hyperlink{thB}{B}.

Note that the preceding section together with Section~\ref{s:p10} and Section~\ref{s:p11} forms a complete proof of Theorem~\hyperlink{thA}{A} in the case that $\aI\geq\aII,\aIII$. \label{Biiiend}
\qed

\section{Proof of the stability result  -- Part II}
\label{s:stabp2}

In \cite{DF2},
we provide the proof of Theorem~\hyperlink{thB}{B} in the case that $\aIII\geq\aI$. There, we prove that any three-multicoloured $(1-\eta^4)$-complete graph on  slightly fewer than $\max\{2\aI+\aII, \half\aI +\half\aII +\aIII\}k$ vertices will have a red connected-matching on at least~$\alpha_{1}k$ vertices, a blue connected-matching on at least~$\alpha_{2}k$ vertices, a green odd connected-matching on at least~$\alpha_{3}k$ vertices or will have one of the coloured structures described in Theorem~\hyperlink{thB}{B}.

\section{Proof of the main result -- Setup}
\label{s:p10}

For $\aI,\aII,\aIII>0$ such that $\aI\geq\aII$,  we set $c=\max\{2\aI+\aII,\half\aI+\half\aII+\aIII\}$, set
$$ \eta=\frac{1}{2}\min \left\{\eta_{B1}(\aI,\aII,\aIII),\eta_{B2}(\aI,\aII,\aIII),10^{-50},
 \left(\frac{\aII}{2500}\right)^{2},
\left(\frac{\aII}{200}\right)^{128}\right\}$$ 
and let~$k_0$ be the smallest integer such that
$$k_0\geq \max \left\{\left(c-\half{\eta}\right)k_{B1}(\aI,\aII,\aIII,\eta), \left(c-\half{\eta}\right)k_{B2}(\aI,\aII,\aIII,\eta), \frac{72}{\eta}\right\}.$$
We let
$$N=\max\left\{2\llangle \aI n\rrangle + \llangle \aII n \rrangle - 3,{~}\half\llangle \aI n\rrangle + \half\llangle \aII n \rrangle + \langle \aIII n \rangle - 2\right\},$$ for some integer $n$ such that $N \geq  K_{\ref{l:sze}}(\eta^4,k_0)$ and
$n>\max\{n_{\ref{th:blow-up}}(2,1,0,\eta), n_{\ref{th:blow-up}}(\half,\half,1,\eta)\}$
and consider a three-colouring of~$G=(V,E)$, the complete graph on~$N$ vertices.

In order to prove Theorem~\hyperlink{thA}{A}, we must prove that $G$ contains either a red cycle on~$\llangle \aI n \rrangle$ vertices, a blue cycle on~$\llangle \aII n \rrangle$ vertices or a green cycle on $\langle \aIII n \rangle$ vertices.

Recall that we use $G_1, G_2, G_3$ to refer to the monochromatic spanning subgraphs of $G$. That is,~$G_1$ (resp. $G_2, G_3$) has the same vertex set as~$G$ and includes as an edge any edge which in~$G$ is coloured red (resp. blue, green).

By Theorem~\ref{l:sze}, there exists an $(\eta^4,G_1,G_2,G_3)$-regular partition $\Pi=(V_0,V_1,\dots,V_K)$ for some~$K$ such that $k_0\leq K \leq K_{\ref{l:sze}}(\eta^4,k_0)$. Given such a partition, we define the $(\eta^4,\eta,\Pi)$-reduced-graph~$\cG=(\cV,\cE)$ on~$K$ vertices as in Definition~\ref{reduced}. The result is a three-multicoloured graph $\cG=(\cV,\cE)$ with
\begin{align*}
\cV&=\{V_1,V_2,\dots,V_K\}, &
\cE&=\{V_iV_j : (V_i,V_j) \text{ is } (\eta^4,G_r)\text{-regular for }r=1,2,3\},
\end{align*}
such that a given edge $V_iV_j$ of~$\cG$ is coloured with every colour for which there are at least $\eta|V_i||V_j|$ edges of that colour between~$V_i$ and~$V_j$ in~$G$.

In what follows, we will use $\cG_1, \cG_2, \cG_3$ to refer to the monochromatic spanning subgraphs of the reduced graph $\cG$. That is,~$\cG_1$ (resp. $\cG_2, \cG_3$) has the same vertex set as~$\cG$ and includes as an edge any edge which in~$G$ is coloured red (resp. blue, green).
Note that, by scaling, we may assume that either $\aII,\aIII\leq\aI=1$ or $\aII\leq\aI\leq1\leq\aIII\leq 2$. Thus, since $K\geq k_0 \geq 72/\eta$, we may fix an integer~$k$ such that 
\vspace{-2mm}
\begin{align}
\left(c-\eta\right)k \leq K \leq \left(c-\half\eta\right)k. 
\label{sizeK}
\end{align}
and may assume that $k\leq K\leq 3k$, $n\leq N\leq 3n$. 
Notice, also, that since the partition is~$\eta^4$-regular, we have $|V_0|\leq \eta^4 N$ and, for $1\leq i \leq K$,
\begin{equation}
\label{NK}
(1-\eta^4)\frac{N}{K}\leq |V_i|\leq \frac{N}{K}.
\end{equation}

\hypertarget{reB}
Applying Theorem~\hyperlink{thB}{B}, we find that $\cG$ contains at least one of the following:
\begin{itemize}
\item [(i)]	a red connected-matching on at least $\aI
 k$ vertices;
\item [(ii)]  a blue connected-matching on at least $\aII k$ vertices;
\item [(iii)]  a green odd connected-matching on at least $\aIII k$ vertices;
\item [(iv)]  
disjoint subsets of vertices $\cX$ and $\cY$ such that $G[\cX]$ contains a two-coloured spanning subgraph~$H$ from $\cH_{1}\cup\cH_{2}$ and $G[\cY]$ contains a two-coloured spanning subgraph $K$ from $\cH_{1}\cup\cH_{2}$, where
\vspace{-2mm}
\begin{align*}
\cH_1=&\cH\left((\aI-2\eta^{1/32})k,(\half\aII-2\eta^{1/32})k,3\eta^4 k,\eta^{1/32},\text{red},\text{blue}\right),\text{ } 
\\ \cH_2=&\cH\left((\aII-2\eta^{1/32})k,(\half\aI-2\eta^{1/32})k,3\eta^4 k,\eta^{1/32},\text{blue},\text{red}\right);
\end{align*}
 \item [(v)]  a subgraph from 
 $$\cK\left((\half\aI-14000\eta^{1/2})k, (\half\aII-14000\eta^{1/2})k, (\aIII-68000\eta^{1/2})k, 4\eta^4 k \right);$$
\item [(vi)]  a subgraph from $\cK_1^{*}\cup \cK_2^*$, where
\begin{align*}
\cK_{1}^{*}=\cK^*\big((\half\aI-97\eta^{1/2})k, (\half\aI-97&\eta^{1/2})k, (\half\aI+102\eta^{1/2})k, \\&(\half\aI+102\eta^{1/2})k, (\aIII-10\eta^{1/2})k, 4\eta^4 k\big),\hphantom{l}
\end{align*}
\vspace{-12mm}
\begin{align*}
\cK_{2}^{*}=\cK^*\big((\half\aI-97\eta^{1/2})k, (\half\aII-97\eta^{1/2})k, (\tfrac{3}{4}&\aIII-140\eta^{1/2})k, \\& 100\eta^{1/2}k, (\aIII-10\eta^{1/2})k, 4\eta^4 k\big).
\end{align*}
\end{itemize}
\vspace{-2mm}
Furthermore, 
\begin{itemize}
\item[(iv)] occurs only if $\aIII\leq \tfrac{3}{2}\aI+\half\aII+14\eta^{1/2}$, with $H, K\in \cH_1$ unless $\aII\geq\aI-\eta^{1/16}$; 
\item[(v)] and (vi) occur only if $\aIII\geq \tfrac{3}{2}\aI+\half\aII-10\eta^{1/2}$.
\end{itemize}

Since $n>\max\{n_{\ref{th:blow-up}}(2,1,0,\eta), n_{\ref{th:blow-up}}(\half,\half,1,\eta)\}$ and $\eta<10^{-50}$, in cases (i)--(iii), Theorem~\ref{th:blow-up} gives a cycle of appropriate length, colour and parity to complete the proof. 

Thus, we need only concern ourselves with cases (iv)--(vi). We divide the remainder of the proof into three parts, each corresponding to one of the possible coloured structures.

\section{Proof of the main result -- Part I -- Case (iv)}
\label{s:p11}

Suppose that $\cG$ contains disjoint subsets of vertices $\cX$ and $\cY$ such that $G[\cX]$ contains a two-coloured spanning subgraph~$H$ from $\cH_{1}\cup\cH_{2}$ and $G[\cY]$ contains a two-coloured spanning subgraph $K$ from $\cH_{1}\cup\cH_{2}$, where
\vspace{-2mm}
\begin{align*}
\cH_1=&\left((\aI-2\eta^{1/32})k,(\half\aII-2\eta^{1/32})k,3\eta^4 k,\eta^{1/32},\text{red},\text{blue}\right),\text{ } 
\\ \cH_2=&\left((\aII-2\eta^{1/32})k,(\half\aI-2\eta^{1/32})k,3\eta^4 k,\eta^{1/32},\text{blue},\text{red}\right).
\end{align*}
In this case, additionally, from Theorem~\hyperlink{reB}{B} we may assume that
\begin{equation}
\label{a3notbig} 
\aIII\leq\tfrac{3}{2}\aI+\half\aII+14\eta^{1/2}\leq 2\aI+14\eta^{1/2}.
\end{equation}

We divide the proof that follows into three sub-parts depending on the colouring of the subgraphs~$H$~and~$K$, that is, whether $H$ and $K$ belong to $\cH_1$ or $\cH_2$:

\subsection*{Part I.A: $H, K\in \cH_1$.}

We have a natural partition of the vertex set $\cV$ of $\cG$ into $\cX _1 \cup \cX _2 \cup \cY _1 \cup \cY _2\cup \cZ$ where $\cX_1\cup \cX_2$ is the partition of the vertices of $H$ given by Definition~\ref{d:H} and $\cY_1\cup\cY_2$ is the corresponding partition of the vertices of $K$. Thus we have
\begin{equation}
\left.
\label{IA0a}
\begin{aligned}
\quad\quad\quad\quad\quad\quad\quad\quad\,\,\,
(\aI-2\eta^{1/32})k\leq|\cX _1|&=|\cY _1|=p\leq \aI k, 
\quad\quad\quad\quad\quad\quad\quad\,\,\,\\
(\half\aII-2\eta^{1/32})k\leq|\cX _2|&=|\cY _2|=q\leq \half\aII k, 
\end{aligned}
\right\}\!
\end{equation}
such that
\begin{itemize}
\labitem{HA1}{HA1} $\cG_1[\cX _1], \cG_1[\cY _1$] are each $(1-2\eta^{1/32})$-complete (and thus connected); 
\labitem{HA2}{HA2} $\cG_2[\cX_1,\cX_2], \cG_2[\cY_1,\cY_2]$ are each $(1-2\eta^{1/32})$-complete (and thus connected);
\labitem{HA3}{HA3} $\cG[\cX_1\cup\cX_2\cup\cY_1\cup\cY_2]$ is $3\eta^4 k$-almost-complete (and thus connected);
\labitem{HA4}{HA4} $\cG[\cX _1], \cG[\cY _1$] are each $2\eta^{1/32}$-sparse in blue and contain no green edges; and
\labitem{HA5}{HA5} $\cG[\cX_1,\cX_2], \cG[\cY_1,\cY_2]$ are each $2\eta^{1/32}$-sparse in red and contain no green edges.
\end{itemize}

\begin{figure}[!h]
\centering
\includegraphics[width=64mm, page=1]{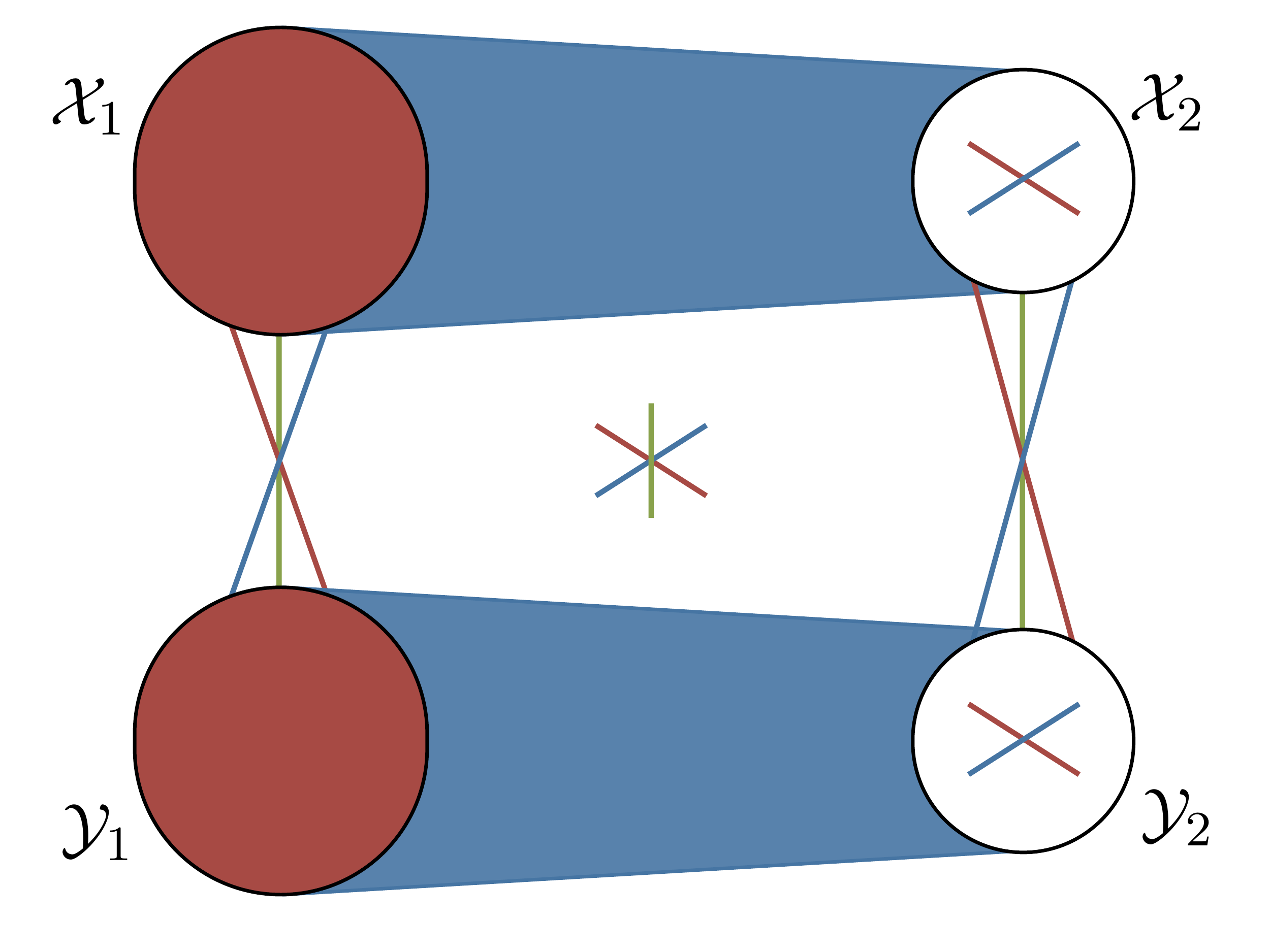}
\vspace{0mm}\caption{Coloured structure of the  reduced-graph in Part I.A.}
  
\end{figure}
The remainder of this section focuses on showing that the original graph must have a similar structure which can then be exploited in order to force a cycle of appropriate length, colour and parity.
By definition, each vertex $V_{i}$ of $\cG=(\cV,\cE)$ represents a class of vertices of $G=(V,E)$. In what follows, we will refer to these classes as \emph{clusters} (of vertices of~$G$). Additionally, recall, from \ref{NK}, that
$$(1-\eta^4)\frac{N}{K}\leq |V_i|\leq \frac{N}{K}.$$ 

Since $n> \max\{n_{\ref{th:blow-up}}(2,1,0,\eta), n_{\ref{th:blow-up}}(\half,\half,1,\eta)\}$, we can (as in the proof of Theorem~\ref{th:blow-up}) prove that $$
|V_i|\geq \left(1+\frac{\eta}{24}\right)\frac{n}{k}> \frac{n}{k}.$$

We partition the vertices of~$G$ into sets $X_{1}, X_{2}, Y_{1}, Y_{2}$ and $Z$ corresponding to the partition of the vertices of~$\cG$ into $\cX_1, \cX_2, \cY_1,\cY_2$ and $\cZ$. Then,~$X_1, Y_1$ each contain~$p$ clusters of vertices and $X_2, Y_2$ each contain~$q$ clusters and, recalling (\ref{IA0a}), we have
\begin{equation}
\left.
\label{IA1} 
\begin{aligned}
\,\,\,\quad\quad\quad\quad\quad\quad\quad\quad\quad\,\,\,\,
|X_1|,|Y_1| & = p|V_1| \geq (\aI-2\eta^{1/32})n,
\quad\quad\quad\quad\quad\quad\quad\quad
\\
|X_2|,|Y_2| & = q|V_1|\geq (\half \aII-2\eta^{1/32})n.
\end{aligned}
\right\}\!
\end{equation}

In what follows, we will \textit{remove} vertices from $X_1\cup X_2\cup Y_1\cup Y_2$ by moving them into~$Z$ in order to show that, in what remains, $G[X_1\cup X_2\cup Y_1 \cup Y_2]$ has a particular coloured structure.  We begin by proving the below claim which essentially tells us that  $G$ has similar coloured structure to $\cG$:

\begin{claim}
\label{G-struct}
We can \textit{remove} at most $9\eta^{1/64}n$ vertices from each of~$X_1$ and~$Y_1$ and at most $4\eta^{1/64}n$ vertices from each of~$X_2$ and~$Y_2$ so that the following pair of conditions~hold.
 \end{claim}
\begin{itemize}
\labitem{HA6}{HA6} $G_1[X_1]$ and $G_1[Y_1]$ are each $8\eta^{1/64}n$-almost-complete; and
\labitem{HA7}{HA7} $G_2[X_1,X_2]$ and $G_2[Y_1,Y_2]$ are each $4\eta^{1/64}n$-almost-complete.
\end{itemize}

\begin{proof} Consider the complete three-coloured graph $G[X_1]$ and recall, from (\ref{HA1}), (\ref{HA3}) and (\ref{HA4}), that~$\cG[\cX_1]$ contains only red and blue edges and is $3\eta^4k$-almost-complete. Given the structure of~$\cG$, we can bound the number of non-red edges in $G[X_1]$ as follows:

Since regularity provides no indication as to the colours of the edges contained within each cluster, these could potentially all be non-red. There are~$p$ clusters in $X_1$, each with at most $N/K$ vertices. Thus, there are at most $$p\binom{N/K}{2}$$ non-red edges in~$X_1$ within the clusters.

Now, consider a pair of clusters $(U_1,U_2)$ in $X_1$. If $(U_1, U_2)$ is not $\eta^4$-regular, then we can only trivially bound the number of non-red edges in $G[U_1,U_2]$ by $|U_1||U_2|\leq (N/K)^2$. However, by~(\ref{HA3}), there are at most $3\eta^4 |\cX_1| k$ such pairs in $\cG$. Thus, we can bound the number of non-red edges coming from non-regular pairs by
$$3\eta^4 pk \left(\frac{N}{K}\right)^2.$$

If the pair is regular and~$U_1$ and~$U_2$ are joined by a blue edge in the  reduced-graph, then, again, we can only trivially bound the number of non-red edges in $G[U_1,U_2]$ by $(N/K)^2$. However, by~(\ref{HA4}),~$\cG_2[X_1]$ is $2\eta^{1/32}$-sparse, so there are at most $2\eta^{1/32}\binom{p}{2}$ blue edges in~$\cG[X_1]$ and, thus, there are at most $$2\eta^{1/32}\binom{p}{2}\left(\frac{N}{K}\right)^2$$ non-red edges in $G[X_1]$ corresponding to such pairs of clusters.

Finally, if the pair is regular and~$U_1$ and~$U_2$ are not joined by a blue edge in the  reduced-graph, then the blue density of the pair is at most~$\eta$ (since, if the density were higher, they would be joined by a blue edge). Likewise, the green density of the pair is at most~$\eta$ (since there are no green edges in $\cG[X_1]$). Thus, there are at most $$2\eta\binom{p}{2}\left(\frac{N}{K}\right)^2$$ non-red edges in $G[X_1]$ corresponding to  such pairs of clusters.

Summing the four possibilities above gives an upper bound of 
$$p \binom{N/K}{2} + 3\eta^4 pk \left(\frac{N}{K}\right)^2+2\eta^{1/32}\binom{p}{2}\left(\frac{N}{K}\right)^{2}+2\eta\binom{p}{2}\left(\frac{N}{K}\right)^{2} $$ non-red edges in $G[X_1]$.

Since $K\geq k, \eta^{-1}$, $N\leq 3n$ and $p\leq \aI k \leq k$, we obtain
$$e(G_2[X_1])+e(G_3[X_1])\leq [4.5\eta + 27\eta^{4} +18\eta^{1/32}+9 \eta]n^2\leq 24\eta^{1/32}n^2.$$
Since $G[X_1]$ is complete and contains at most $24\eta^{1/32}n^2$ non-red edges, there are at most~$6\eta^{1/64}n$ vertices with red degree at most $|X_1|-8\eta^{1/64}n$. Removing these vertices from $X_1$, that is, re-assigning these vertices to~$Z$ gives a new~$X_1$  
such that every vertex in $G[X_1]$ has red degree at least $|X_1|-8\eta^{1/64}n$. The same argument works for $G[Y_1]$, thus completing the proof of (\ref{HA6}).

Now, consider $G[X_1,X_2]$. In a similar way to the above, we can bound the number of non-blue edges in $G[X_1,X_2]$ by 
$$3\eta^4 pk \left(\frac{N}{K}\right)^2+2\eta^{1/32}pq\left(\frac{N}{K}\right)^{2}+2\eta pq\left(\frac{N}{K}\right)^{2}. $$
Where the first term bounds the number of non-blue edges between non-regular pairs, the second bounds the number of non-blue edges between pairs of clusters that are joined by a red edge in the  reduced-graph and the second bounds the number of non-blue edges between pairs of clusters that are not joined by a red edge in the  reduced-graph.

Since $K\geq k$, $N\leq 3n$, $p \leq \aI k \leq k$ and $q\leq \half\aII k\leq\tfrac{1}{2}k$, we obtain
$$e(G_1[X_1,X_2])+e(G_3[X_1,X_2])\leq 16 \eta^{1/32}n^2.$$
Since $G[X_1,X_2]$ is complete and contains at most $16\eta^{1/32}n^2$ non-blue edges, there are at most $4\eta^{1/64}n$ vertices in~$X_1$ with blue degree to~$X_2$ at most $|X_2|-4\eta^{1/64}n$ and at most~$4\eta^{1/64}n$ vertices in~$X_2$ with blue degree to~$X_1$ at most $|X_1|-4\eta^{1/64}n$. Removing these vertices results in every vertex in~$X_1$ having degree in $G_2[X_1,X_2]$ at least $|X_2|-4\eta^{1/64}n$ and every vertex in~$X_2$ having degree in $G_2[X_1,X_2]$ at least $|X_1|-4\eta^{1/64}n$.

We repeat the above for~$G[Y_1,Y_2]$, removing vertices such that every (remaining) vertex in~$Y_1$ has degree in $G_2[Y_1,Y_2]$ at least $|Y_2|-4\eta^{1/64}n$ and every (remaining) vertex in~$Y_2$ has degree in $G_2[Y_1,Y_2]$ at least $|Y_1|-4\eta^{1/64}n$, thus completing the proof of (\ref{HA7}).
\end{proof}

Having discarded some vertices, recalling~(\ref{IA1}), we  have
\begin{align}
\label{IA2}
|X_1|,|Y_1| & \geq (\aI-10\eta^{1/64})n,
& |X_2|,|Y_2| & \geq (\half \aII-5\eta^{1/64})n.
\end{align}
We now proceed to the {end-game}: Notice that, given the colouring found thus far, $G[X_1]$ and $G[Y_1]$ each contain a red Hamiltonian cycle. Similarly, $G[X_1,X_2]$ and $G[Y_1,Y_2]$ each contain a blue cycle of length twice the size of the smaller part. Essentially, we will show that it is possible to augment each of $X_1, X_2, Y_1, Y_2$ with vertices from~$Z$ while maintaining these properties and, then, considering the sizes of each part, show that there must, in fact, be a cycle of appropriate length, colour and parity to complete the proof. 

\begin{figure}[!h]
\centering
\includegraphics[width=64mm, page=5]{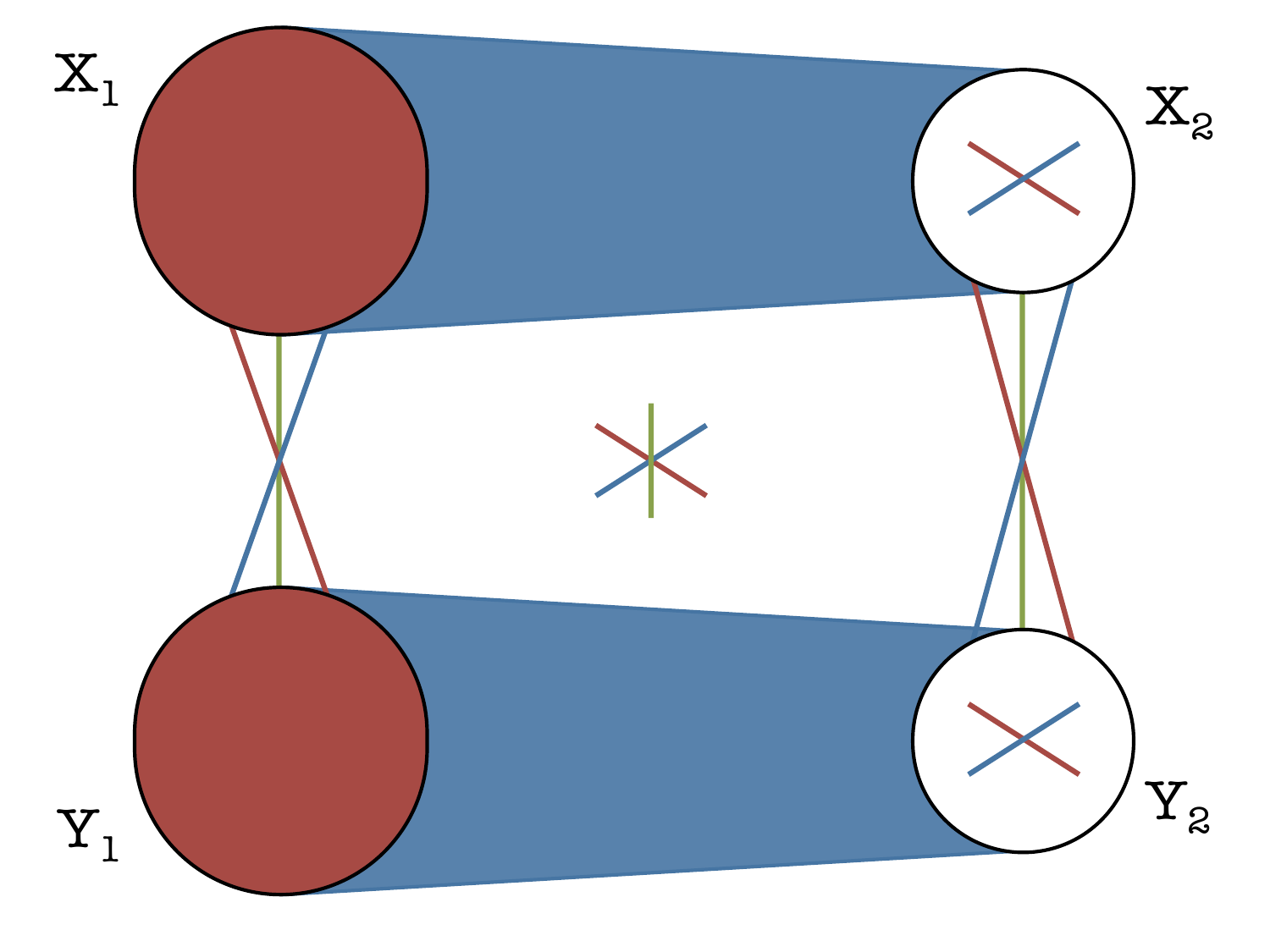}
\vspace{0mm}\caption{Colouring of $G$ after Claim~\ref{G-struct}.}
\end{figure}
The following pair of claims tell us that after removing a small number of vertices from~$X_1$ and $Y_1$, we may assume that  all edges in $G[X_1,Y_1]$ are coloured green:

\begin{claim}
\label{nogreen}
\hspace{-2.8mm} {\rm \bf a.} If there exist distinct vertices $x_1,x_2 \in X_1$ and $y_1, y_2\in Y_1$ such~that~$x_1y_1$ and~$x_2y_2$ are coloured red, then~$G$ contains a red cycle of length exactly~$\llangle \aI n \rrangle$.

{\rm \bf Claim~\ref{nogreen}.b.} If there exist distinct vertices $x_1,x_2 \in X_1$ and $y_1, y_2\in Y_1$ such~that~$x_1y_1$ and~$x_2y_2$ are coloured blue, then~$G$ contains a blue cycle of length exactly~$\llangle \aII n \rrangle$.
\end{claim}

\begin{proof}
(a) Suppose there exist distinct vertices $x_1,x_2\in X_1$ and  $y_1,y_2\in Y_1$ such that the edges $x_1y_1$ and $x_2y_2$ are coloured red. Then, let $\widetilde{X}_1$ be any set of $\half \llangle \aI n \rrangle $ vertices in~$X_1$ such that $x_1,x_2 \in \widetilde{X}_1$. 

By (\ref{HA6}), every vertex in $\widetilde{X}_1$ has degree at least $|\widetilde{X}_1|-8\eta^{1/64}n$ in $G_1[\widetilde{X}_1]$. Since $\eta\leq(\aI/100)^{64}$, we have $|\widetilde{X}_1|-8\eta^{1/64}n \geq \half |\widetilde{X}_1| +2$. So, by Corollary~\ref{dirac2}, there exists a Hamiltonian path in $G_1[\widetilde{X}_1]$ between $x_1, x_2$, that is, there exists a red path between~$x_1$ and $x_2$ in $G[X_1]$ on exactly $\half \llangle \aI n \rrangle$ vertices. 

Likewise, given any two vertices $y_1,y_2$ in~$Y_1$, there exists a red path between~$y_1$ and~$y_2$ in~$G[Y_1]$ on exactly $\half \llangle \aI n \rrangle$ vertices.  
Combining the edges $x_1y_1$ and $x_2y_2$ with the red paths gives a red cycle on exactly~$\llangle \aI n \rrangle$ vertices.

(b) Suppose there exist distinct vertices $x_1,x_2\in X_1$ and  $y_1,y_2\in Y_1$ such that $x_1y_1$ and $x_2y_2$ are coloured blue. Then, let $\widetilde{X}_2$ be any set of 
$$\ell_1=\left\lfloor \frac{\llangle \aII n \rrangle-2}{4} \right\rfloor \geq 4\eta^{1/64} n +2 $$ vertices from $X_2$. By~(\ref{HA7}), $x_1$ and $x_2$ each have at least two neighbours in $\widetilde{X}_2$ and, since $\eta\leq(\aI/100)^{64}$, every vertex in $\widetilde{X}_2$ has degree at least 
$|X_1|-4\eta^{1/64}n\geq \half|X_1| +\half|\widetilde{X}_2| +1$ in $G[X_1,\widetilde{X}_2]$. Since $|X_1|>\ell_1+1$, by Lemma~\ref{bp-dir},  $G_2[X_1,\widetilde{X}_2]$ contains a path on exactly $2\ell_1+1$ vertices from~$x_1$ to~$x_2$.

Likewise, given $y_1, y_2 \in Y_1$, for  any set $\widetilde{Y}_2$ of 
$$\ell_2=\left\lceil \frac{\llangle \aII n \rrangle-2}{4} \right\rceil \geq 4\eta^{1/64} n +2 $$ vertices from $Y_2$, $G_2[Y_1,\widetilde{Y}_2]$ contains a a path on exactly $2\ell_2+1$ vertices from~$y_1$ to~$y_2$.

Combining the edges~$x_1y_1$,~$x_2y_2$ with the blue paths found gives a blue cycle on exactly $2\ell_1+2\ell_2+2=\llangle \aII n \rrangle$ vertices, completing the proof of the claim. \end{proof}

The existence of red cycle on~$\aIna$ vertices or a blue cycle on~$\aIIna$ vertices would be sufficient to complete the proof of Theorem~\hyperlink{thA}{A}. Thus, there cannot exist such a pair of vertex-disjoint red edges or such a pair of vertex-disjoint blue edges in $G[X_1,Y_1]$. 
Similarly, there cannot be a pair of vertex-disjoint blue edges in $G[X_1,Y_2]$ or in $G[X_2,Y_1]$. Thus, after removing at most three vertices from each of~$X_1$ and~$Y_1$ and one vertex from from each of~$X_2$ and~$Y_2$, we may assume that 
\begin{itemize}
\labitem{HA8a}{HA8a} the green graph $G_3[X_1,Y_1]$ is complete; and
\labitem{HA8b}{HA8b} the are no blue edges in $G[X_1,Y_2]\cup G[X_2,Y_1]$.
\end{itemize}

Then, recalling~(\ref{IA2}), we have
\begin{align}
\label{IA3}
|X_1|, |Y_1|&\geq (\aI-11\eta^{1/64})n,
& |X_2|, |Y_2|&\geq (\half \aII - 6\eta^{1/64})n.
\end{align}
We now consider $Z$. Defining $Z_G$ to be the set of vertices in $Z$ having a green edge to both $X_1$ and $Y_1$, we prove the following claim which allows us to assume that $Z_G$ is empty:

\begin{figure}[!h]
\centering
\vspace{-1mm}\hspace{2mm}
{\includegraphics[height=50mm, page=10]{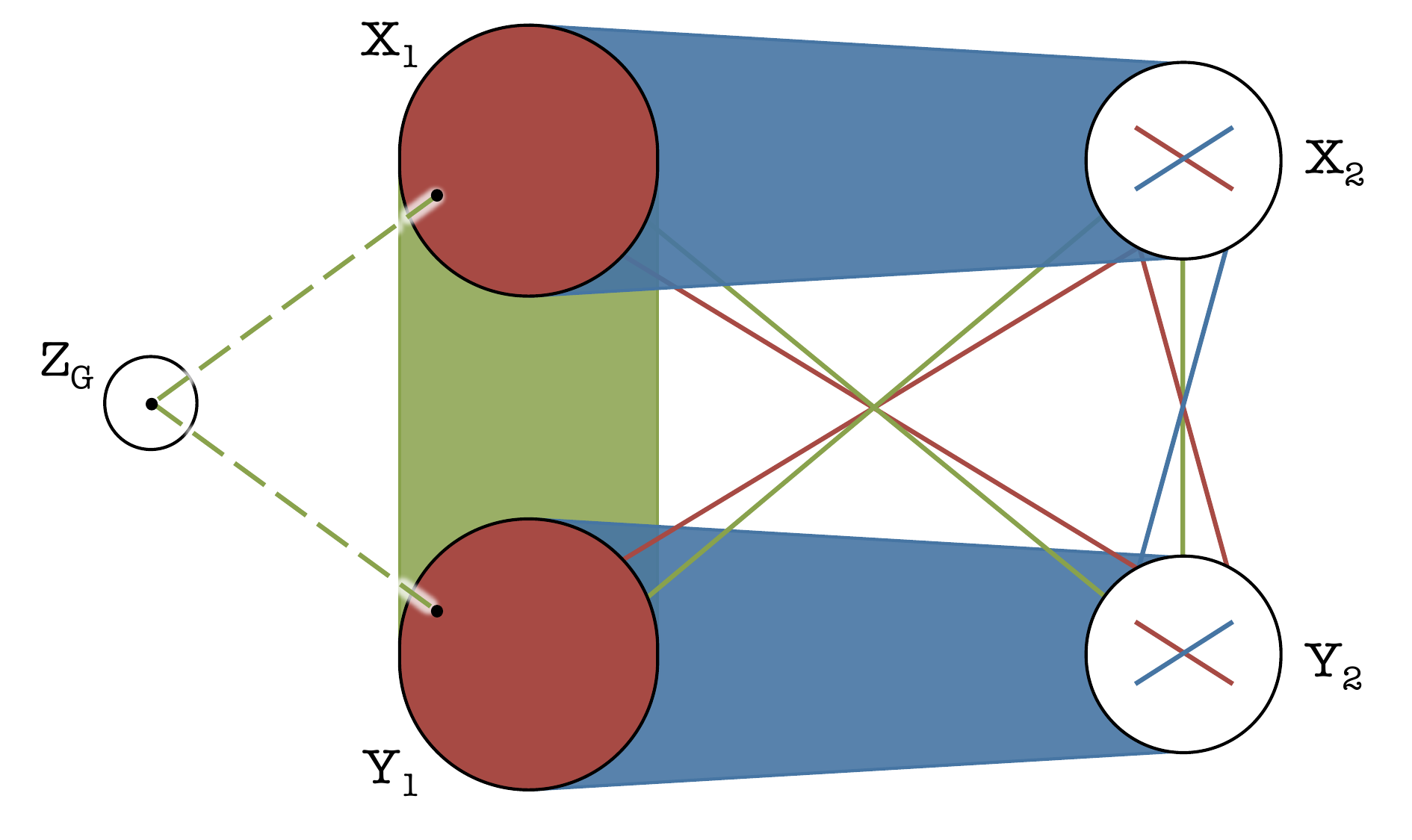}\hspace{18mm}}
\vspace{-1mm}\caption{Using $z\in Z_G$ to construct an odd green cycle. }
  
\end{figure}

\begin{claim}
\label{greenorred}
If $Z_G$ is non empty, then $G$ contains either a red cycle on exactly~$\llangle \aI n \rrangle$ vertices or a green cycle on exactly $\langle \aIII n \rangle$ vertices. 
\end{claim}

\begin{proof}
Suppose that $Z_G$ is non-empty. Then, there exists $x\in X_1$, $y\in Y_1$, $z\in Z$ such that $xz$ and $yz$ are both coloured green. Recalling (\ref{HA8a}), $G_3[X_1,Y_1]$ is complete, thus, we may obtain a cycle of any odd length up to $|X_1|+|Y_1|+1$ in $G[X\cup Y\cup Z_G]$. Therefore, to avoid having a a green cycle on exactly $\langle \aIII n \rangle$ vertices, we may assume that $|X_1|+|Y_1|+1<\langle \aIII n \rangle$. 
Then, considering~(\ref{a3notbig}) and~(\ref{IA3}), we have
$$2\aI-22\eta^{1/64}\leq \aIII\leq 2\aI+2\eta^{1/64},$$
and 
$|X_1|+|Y_1|\geq(2\aI-22\eta^{1/64})n\geq(\aIII-24\eta^{1/64})n.$

In that case, suppose that, for some $x_a, x_b\in X_1$ and $y_a, y_b\in Y_1$, there exist green paths 
\begin{align*}
\quad\quad\quad\quad&P_1\text{ from }x_a\in X_1\text{ to }x_b\in X_1&&\text{ on }2\lceil12\eta^{1/64}n\rceil+1\text{ vertices }\text{in } G[X_1,Y_2],
\quad\quad\quad\quad\\
\quad\quad\quad\quad&P_2\text{ from }y_a\in Y_1\,\text{ to }y_b\in Y_1&&\text{ on }2\lceil12\eta^{1/64}n\rceil+1\text{ vertices }\text{in } G[Y_1,Y_2].
\quad\quad\quad\quad
\end{align*}
Then, since $G_3[X_1,Y_1]$ is complete and $Z_G$ is non-empty, $P_1$ and $P_2$ could be used along with edges from $G[X_1,Y_1]$ and $G[z,X_1\cup Y_1]$ to give an odd green cycle on exactly $\langle \aIII n\rangle$ vertices. Therefore, without loss of generality, we may assume that $G[X_1,Y_2]$ does not contain a green path on $2\lceil12\eta^{1/64}n\rceil+1$ vertices. Thus, by Theorem~\ref{th:eg}, $G[X_1,Y_2]$ contains at most $16\eta^{1/64}n^2$ green edges.

\begin{figure}[!h]
\centering
\includegraphics[height=50mm, page=11]{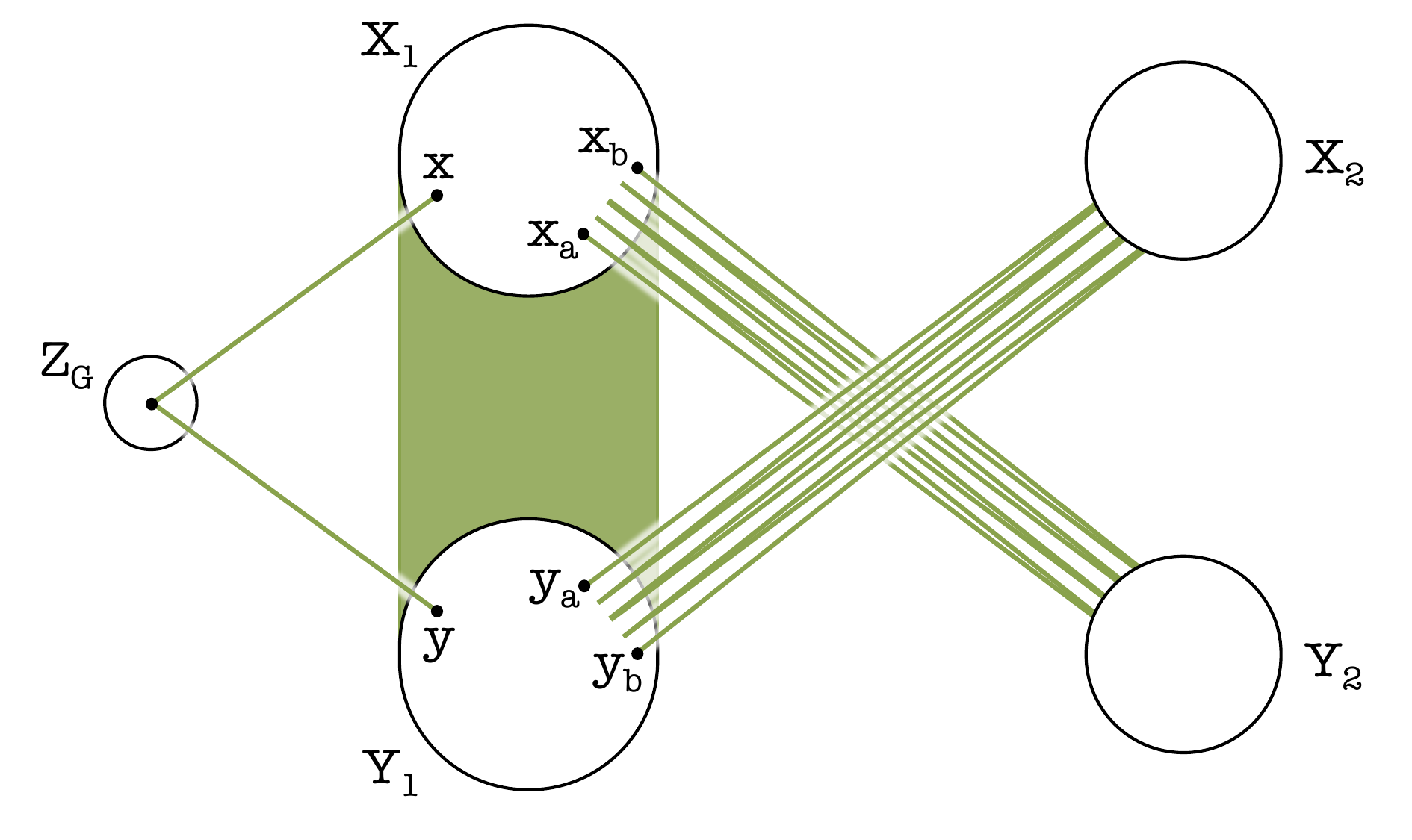}
\vspace{-1mm}\caption{Construction of a long odd green cycle.}
\end{figure}

Recalling (\ref{HA8b}), we know that $G[X_1,Y_2]$ is complete, contains no blue edges and contains at most~$16\eta^{1/64}n^2$ green edges. Thus, there are at most $4\eta^{1/128}n$ vertices in~$X_1$ with red degree to~$Y_2$ at most $|Y_2|-4\eta^{1/128}n$ and at most $4\eta^{1/128}n$ vertices in~$Y_2$ with red degree to~$X_1$ at most $|X_1|-4\eta^{1/128}n$. Removing these vertices from~$X_1\cup Y_2$ results in every vertex in~$X_1$ having red degree  at least $|Y_2|-4\eta^{1/128}n$ in $G[X_1,Y_2]$ and every vertex in~$Y_2$ having red degree  at least $|X_1|-4\eta^{1/128}n$ in $G[X_1,Y_2]$. Thus, $G_1[X_1,Y_2]$ is $4\eta^{1/128}n$-almost-complete.

Recalling~(\ref{IA3}), since $\eta<10^{-50}$, having discarded these vertices, we have 
\begin{align}
\label{IA5}
|X_1|&\geq (\aI-7\eta^{1/128})n,
& |Y_2|&\geq (\half \aII - 7\eta^{1/128})n.
\end{align}
Then, given this bound for $X_1$, there exist disjoint subsets $X_{L},X_{S}\subseteq X_1$ such that 
\begin{align}
\label{IA6}
|X_L|&=\llangle \aI n \rrangle -2\lfloor 8\eta^{1/128} n \rfloor -1, & |X_S|&= \lfloor 8\eta^{1/128} n \rfloor.
\end{align} 

\begin{figure}[!h]
\centering
\includegraphics[height=50mm, page=7]{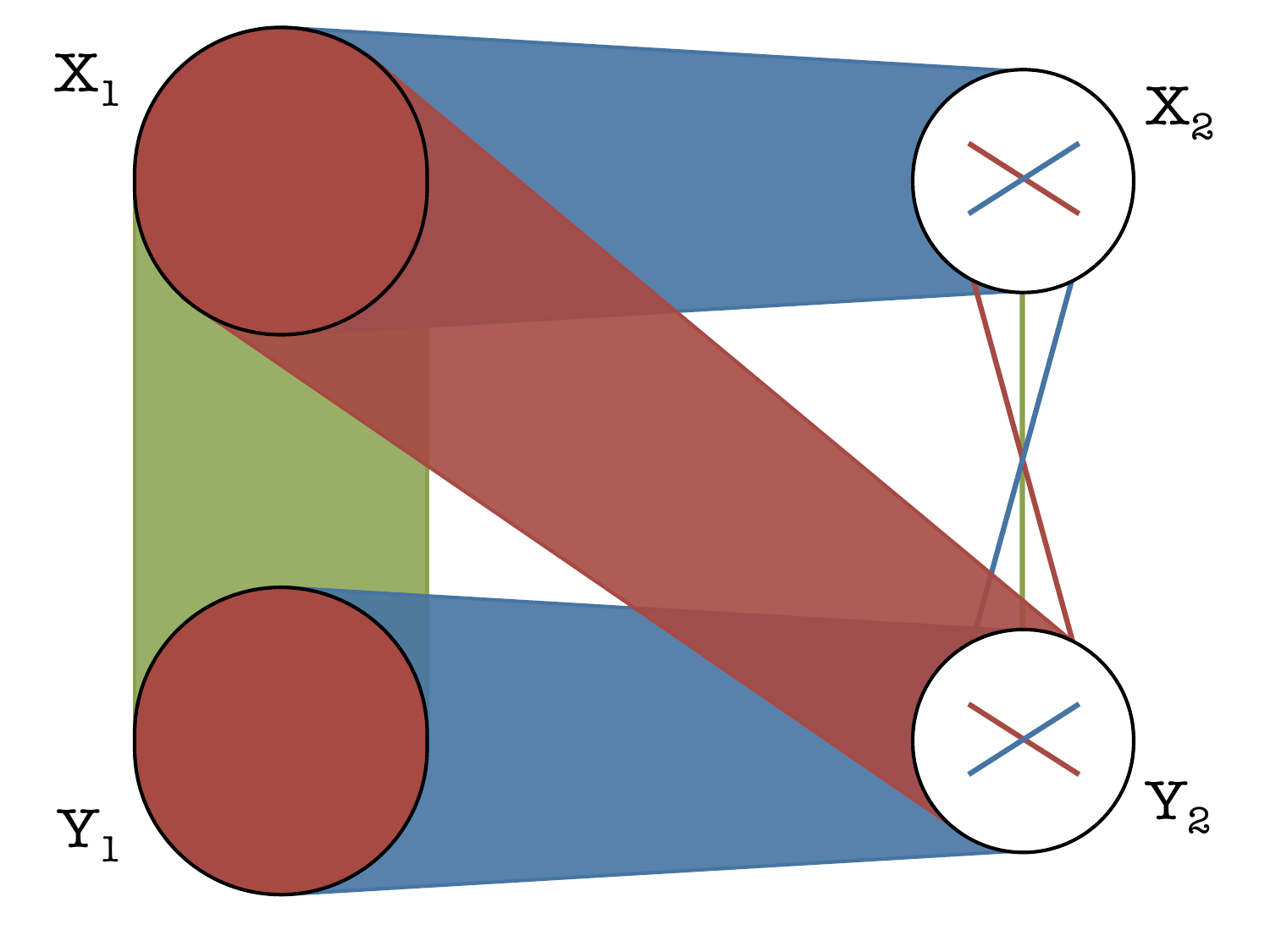}
\vspace{0mm}\caption{Colouring of $G[X_1,Y_2]$ in Claim~\ref{greenorred}. }
\end{figure}

Let $x_1, x_2$ be distinct vertices in $X_L$. By (\ref{HA6}), every vertex in $X_1$ has red degree  at least $(|X_1|-1)-8\eta^{1/64}n$ in $G[X_1]$. Thus, since $\eta\leq(\aI/100)^{64}$, every vertex in $X_L$ has red degree in $G[X_L]$ at least~$\half|X_L|+1$ and so, by Corollary~\ref{dirac2}, there exists a red Hamiltonian  path $R_1$ in $G[X_L]$ from $x_1$ to~$x_2$. 

We now consider $G[Y_2,X_S]$. Since $G_1[X_1,Y_2]$ is $4\eta^{1/128}n$-almost-complete, there exist disjoint vertices $y_1,y_2\in Y_2$, such that $x_1y_1$ and $x_2y_2$ are red and $d(y_1),d(y_2)\geq 2$.
Since $\eta\leq(\aII/100)^{128}$, considering~(\ref{IA5}) and~(\ref{IA6}), we have $|Y_2|>|X_S|+1$ and know that every vertex in $X_S$ has red degree at least $|Y_2|-4\eta^{1/128}n\geq \half(|Y_2|+\tfrac{1}{2}\aII n -13\eta^{1/128}n)\geq \half(|Y_2|+|X_S|+1)$ in $G[Y_2,X_S]$. Thus, by Lemma~\ref{bp-dir}, there exists a red path $R_2$ in~$G[Y_2,X_S]$ from $y_1$ to $y_2$ which visits every vertex of $X_S$.

Together, the paths $R_1, R_2$ and the edges $x_1y_1$ and $x_2y_2$ form a red cycle on exactly~$\llangle \aI n\rrangle$ vertices in $G[X_1]\cup G[X_1,Y_2]$, thus completing the proof of the claim.
\end{proof}

The existence of a red cycle on exactly~$\llangle \aI n \rrangle$ vertices or a green cycle on exactly $\langle \aIII n \rangle$ vertices as offered by Claim~\ref{greenorred} would be sufficient to complete the proof of Theorem~\hyperlink{thA}{A}. We may, therefore, assume that $Z_G$ is empty.
Thus, defining $Z_X$ to be the set of vertices in $Z$ having no green edges to $X_1$ and $Z_Y$ to be the set of vertices in $Z$ having no green edges to $Y_1$, we see that  $Z_X\cup Z_Y$ is a partition of $Z$.
We thus have a partition of $V(G)$ into $X_1\cup X_2\cup Y_1\cup Y_2\cup Z_X\cup Z_Y$. Then, since $$|V(G)|\geq2\llangle\aI n \rrangle +\llangle\aII n \rrangle-3,$$ without loss of generality, we may assume that 
$$|X_1 \cup X_2\cup Z_X|\geq \llangle \aI n \rrangle + \half \llangle \aII n \rrangle -1$$ since, if not, then $Y_1 \cup Y_2\cup Z_Y$ is that large instead. 
%
Given (\ref{HA6}) and (\ref{HA7}), we can obtain upper bounds on $|X_1|$, $|X_2|$, $|Y_1|$ and $|Y_2|$ as follows: By Corollary~\ref{dirac1a}, for every integer $m$ such that $16\eta^{1/64}n+2\leq m \leq |X_1|$, we know that $G[X_1]$ contains a red cycle of length $m$. Thus, in order to avoid having a red cycle on exactly~$\aIna$ vertices, we may assume that $|X_1|<\aIna$.
By Corollary~\ref{moonmoser2}, for every even integer $m$ such that $16\eta^{1/64}n+2\leq m\leq 2\min\{|X_1|, |X_2|\}$, we know that $G[X_1,X_2]$ contains a blue cycle on $m$ vertices. Recalling~(\ref{IA3}), we have $|X_1|\geq(\aI-11\eta^{1/64})n\geq\half\aII n$, thus, in order to avoid having a blue cycle on exactly~$\aIIna$ vertices, we may assume that $|X_2|<\half\aIIna$.
In summary, we then have
\begin{equation}
\label{IA7}
\left.
\begin{aligned}
\,\,\,\,\quad\quad\quad\quad\quad\quad\quad\quad\quad\,\,\,\,
(\aI-9\eta^{1/64})n&\leq |X_1|< \aIna, 
\quad\quad\quad\quad\quad\quad\quad\quad\quad\\
(\half \aII - 6\eta^{1/64})n&\leq |X_2|< \half\aIna.
\end{aligned}
\right\}\!
\end{equation}

\begin{figure}[!h]
\centering
{\includegraphics[height=50mm, page=12]{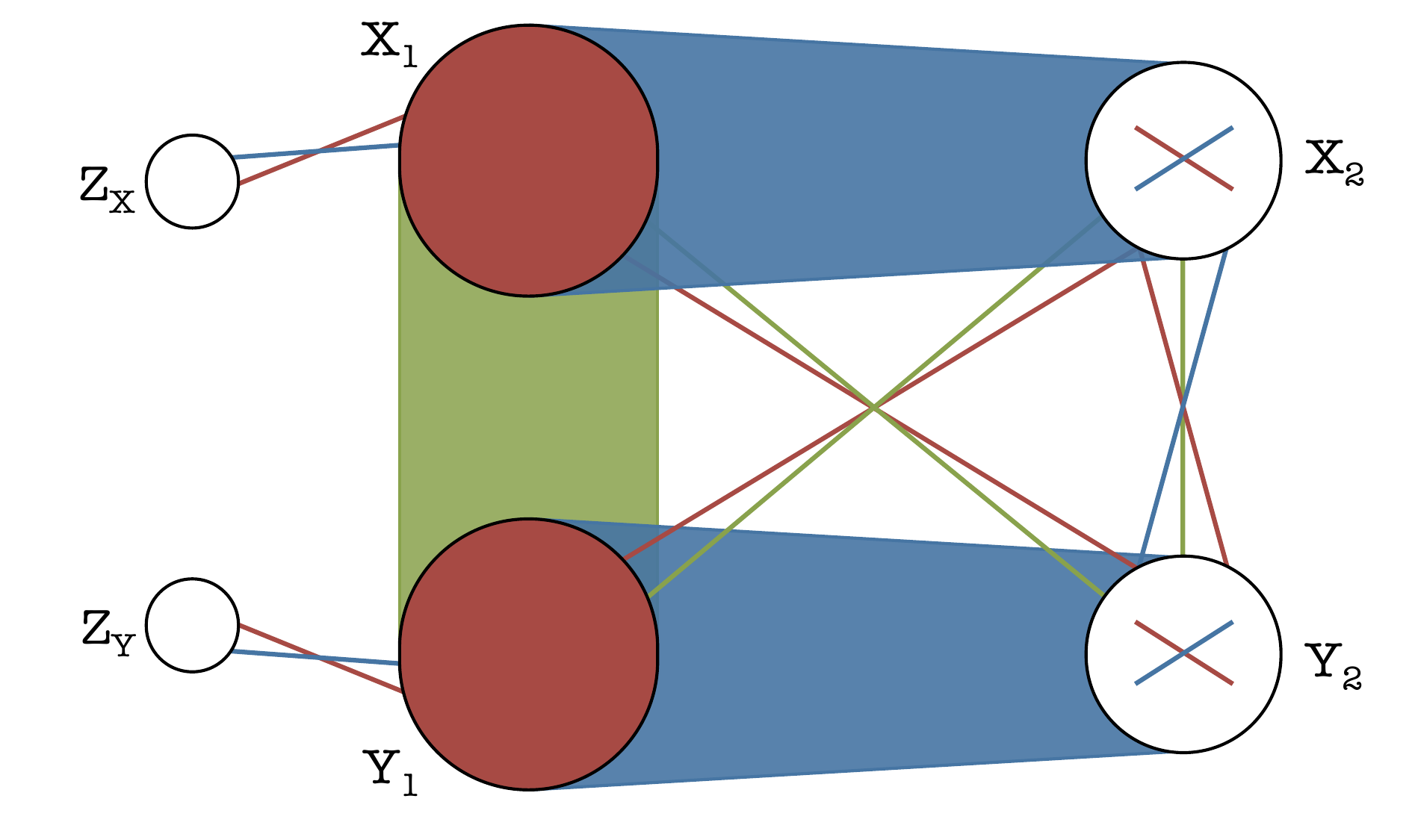}\hspace{18mm}}
\vspace{0mm}\caption{Partition of $Z$ into $Z_X\cup Z_Y$.}
\end{figure}

Letting
\begin{align*}
Z_B&=\{ z\in Z_X \text{ such that } z \text{ has at least } |X_1|-64\eta^{1/64} \text{ blue neighbours in } X_1\};\text{ and}
\\ Z_R&=Z\backslash Z_B=\{ z\in Z_X \text{ such that } z \text{ has at least } 64\eta^{1/64} \text{ red neighbours in } X_1\},
\end{align*}
we have $Z=Z_R\cup Z_B$. Thus, either $|X_1\cup Z_R|\geq\llangle \aI n \rrangle$ or $|X_2\cup Z_B|\geq\half\llangle \aII n \rrangle$.

\begin{figure}[!h]
\centering
\includegraphics[height=22mm, page=9]{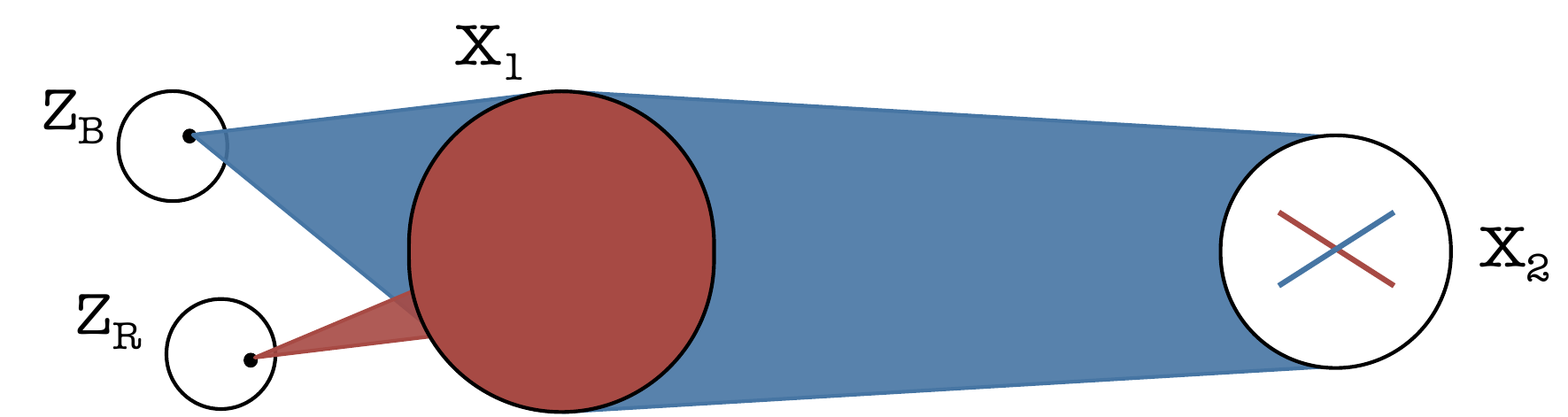}
\vspace{0mm}\caption{Partition of $Z_X$ into $Z_B\cup Z_R$.}
\end{figure}

If $|X_1\cup Z_R|\geq\llangle \aI n \rrangle$, then we show that $G_1[X_1\cup Z_R]$ contains a long red cycle as follows: Let~$X$ be any set of~$\llangle \aI n \rrangle$ vertices from $X_1\cup Z_R$ consisting of every vertex from~$X_1$ and~$\llangle \aI n \rrangle-|X_1|$ vertices from~$Z_R$. By (\ref{HA6}) and~(\ref{IA7}), the red graph~$G_1[X]$ has at least $\llangle \aI n \rrangle-11\eta^{1/64}n$ vertices of degree at least $|X|-20\eta^{1/64}n$ and at most $11\eta^{1/64}$ vertices of degree at least $64\eta^{1/64}n$. Thus, by Theorem~\ref{chv},~$G[X]$ contains a red cycle on exactly~$\llangle \aI n \rrangle$ vertices.

Thus, we may, instead, assume that $|X_2\cup Z_B|\geq\half\llangle \aII n \rrangle$, in which case, we consider the blue graph $G_2[X_1,X_2\cup Z_B]$. Given the relative sizes of~$X_1$ and~$X_2\cup Z_B$ and the large minimum-degree of the graph, we can use Theorem~\ref{moonmoser} to give a blue cycle on exactly~$\llangle \aII n \rrangle$ vertices. Indeed, by~(\ref{IA7}), we have $|X_1|\geq\half\llangle\aII n \rrangle$ and may choose subsets $\widetilde{X}_1\subseteq X_1$, $\widetilde{X}_2\subseteq X_2\cup Z_B$ such that $\wX_2$ includes every vertex of $X_2$ and
\begin{align*}
|\widetilde{X}_1|&=|\widetilde{X}_2|=\half\llangle\aII n \rrangle, & |\wX_2\cup Z_B|\leq6\eta^{1/64}n.
\end{align*}
Recall that $G_2[X_1,X_2]$ is $4\eta^{1/64}k$-almost-complete and that, all vertices in $Z_B$ have blue degree at least $|\widetilde{X}_1|-64\eta^{1/64}n$ in $G[\widetilde{X}_1,\widetilde{X}_2]$. Thus, since $|\wX_2\cup Z_B|\leq 6\eta^{1/64}n$ and $\eta\leq(\aII/200)^{64}$, for any pair of vertices $x_1\in\widetilde{X}_1$ and $x_2\in\widetilde{X}_2$, we have $d(x_1)+d(x_2)\geq |\widetilde{X}_1|+|\widetilde{X}_2|-74\eta^{1/64}n\geq\half\llangle\aII n \rrangle+1.$ 

Therefore, by Theorem~\ref{moonmoser}, $G_2[\wX_1,\wX_2]$ contains a blue cycle on exactly~$\llangle \aII n \rrangle$ vertices, thus completing this part of the proof.

\subsection*{Part I.B: $H, K \in \cH_2$.}
By Theorem~\hyperlink{reB}{B}, this case only occurs when \begin{equation}
\tag{HB0}
\label{HB0}
\aII \leq \aI \leq \aII + \eta^{1/16}.
\end{equation}
Recalling that $\cG_1, \cG_2$ and $\cG_3$ are the monochromatic spanning subgraphs of the  reduced-graph $\cG$, the vertex set~$\cV$ of~$\cG$ has a natural partition into $\cX _1 \cup \cX _2 \cup \cY _1 \cup \cY _2\cup \cZ$ where $\cX_1\cup \cX_2$ is the partition of the vertices of $H$ given by Definition~\ref{d:H} and $\cY_1\cup\cY_2$ is the corresponding partition of the vertices of~$K$. Thus we have
\begin{align}
\left.
\label{IB0a}
\begin{aligned}
\quad\quad\quad\quad\quad\quad\quad\quad\,\,\,
(\aII-2\eta^{1/32})k\leq|\cX _1|&=|\cY _1|=p\leq \aII k, 
\quad\quad\quad\quad\quad\quad\quad\,\,\,\\
(\half\aI-2\eta^{1/32})k\leq|\cX _2|&=|\cY _2|=q\leq \half\aI k, 
\end{aligned}
\right\}
\end{align}
and know that
\begin{itemize}
\labitem{HB1}{HB1} $\cG_2[\cX _1], \cG_2[\cY _1$] are each $(1-2\eta^{1/32})$-complete (and thus connected); 
\labitem{HB2}{HB2} $\cG_1[\cX_1,\cX_2], \cG_1[\cY_1,\cY_2]$ are each $(1-2\eta^{1/32})$-complete (and thus connected);
\labitem{HB3}{HB3} $\cG[\cX_1\cup\cX_2\cup\cY_1\cup\cY_2]$ is $3\eta^4 k$-almost-complete (and thus connected);
\labitem{HB4}{HB4} $\cG[\cX _1], \cG[\cY _1$] are each $2\eta^{1/32}$-sparse in red and contain no green edges; and
\labitem{HB5}{HB5} $\cG[\cX_1,\cX_2], \cG[\cY_1,\cY_2]$ are each $2\eta^{1/32}$-sparse in blue and contain no green edges.
\end{itemize}

\begin{figure}[!h]
\centering
\includegraphics[width=64mm, page=2]{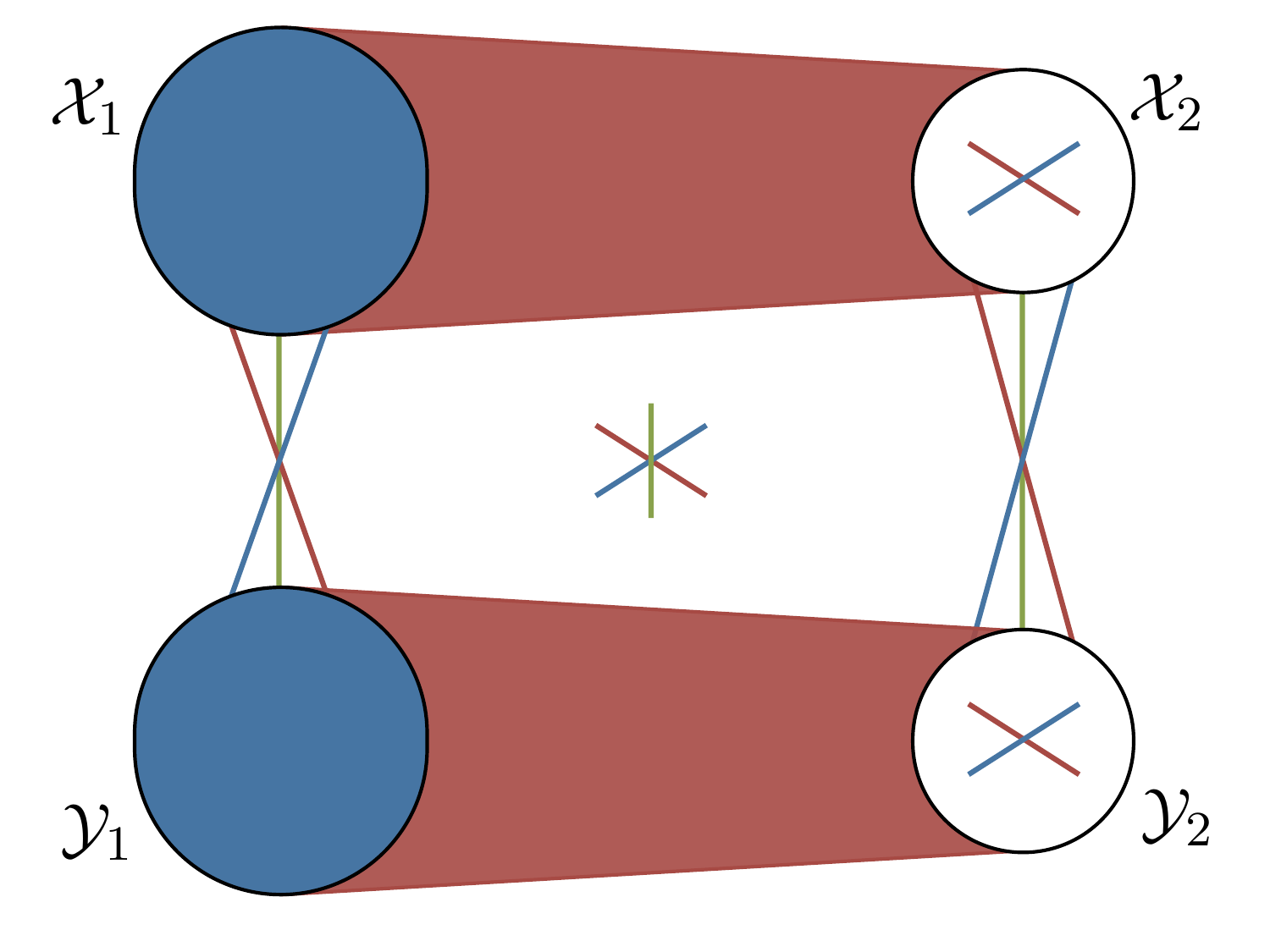}
\vspace{0mm}\caption{Coloured structure of the  reduced-graph in Part I.B.}
\end{figure}

The proof in this case is essentially identical to that in Part I.A. However we include the key steps here for completeness.
As in Part I.A, each vertex~$V_{i}$ of~$\cG$ represents a class of vertices of~$G$. We partition the vertices of~$G$ into sets $X_{1}, X_{2}, Y_{1}, Y_{2}$ and~$Z$ corresponding to the partition of the vertices of $\cG$ into $\cX_1, \cX_2, \cY_1, \cY_2$ and $\cZ$. Then~$X_1$,~$Y_1$ each contain~$p$ {clusters} of vertices, $X_2, Y_2$ each contain~$q$ clusters and we have
\begin{align}
\label{IB1}
|X_1|,|Y_1| & = p|V_1| \geq (\aII-2\eta^{1/32})n,
&|X_2|,|Y_2| & = q|V_1|\geq (\half \aI-2\eta^{1/32})n.
\end{align}

By the following claim, $G$ has essentially the same colouring as the  reduced-graph:
\begin{claim}
\label{G-struct-b}
Given~$G$ as described, we can \textit{remove} at most $9\eta^{1/64}n$ vertices from each of $X_1$ and $Y_1$ and at most $4\eta^{1/64}n$ vertices from each of~$X_2$ and~$Y_2$ so that the following pair of conditions~hold.
 \end{claim}
\begin{itemize}
\labitem{HB6}{HB6} $G_2[X_1]$ and $G_2[Y_1]$ are each $8\eta^{1/64}n$-almost-complete;
\labitem{HB7}{HB7} $G_1[X_1,X_2]$ and $G_1[Y_1,Y_2]$ are each $4\eta^{1/64}n$-almost-complete.
\end{itemize}
\begin{proof}
Identical to that of Claim~\ref{G-struct} with the roles of red and blue exchanged. 
\end{proof}

Having discarded some vertices, recalling~(\ref{IB1}), we  have
\begin{align}
\label{IB2}
|X_1|,|Y_1| & \geq (\aII-10\eta^{1/64})n,
& |X_2|,|Y_2| & \geq (\half \aI-5\eta^{1/64})n,
\end{align}
 and can proceed to the {end-game}. 
  
 The following pair of claims allow us to determine the colouring of $G[X_1,X_2]$:
\begin{claim}
\label{nogreen2}
\hspace{-2.8mm} {\rm \bf a.} If there exist distinct vertices $x_1,x_2 \in X_1$ and $y_1, y_2\in Y_1$ such~that~$x_1y_1$ and~$x_2y_2$ are coloured blue, then~$G$ contains a blue cycle of length exactly~$\llangle \aII n \rrangle$.

{\rm \bf Claim~\ref{nogreen2}.b.} If there exist distinct vertices $x_1,x_2 \in X_1$ and $y_1, y_2\in Y_1$ such~that~$x_1y_1$ and~$x_2y_2$ are coloured red, then~$G$ contains a red cycle of length exactly~$\llangle \aI n \rrangle$.
\end{claim}

\begin{proof}
Identical to that of Claim~\ref{nogreen} with the roles of red and blue exchanged and also the roles of $\aI$ and $\aII$. Note that, when needed in (b), the fact that $|X_1|> \ell_1+1, |Y_1|> \ell_2+1$, follows from (\ref{HB0}).
\end{proof}

The existence of a red cycle on~$\aIna$ vertices or a blue cycle on~$\aIIna$ vertices would be sufficient to complete the proof of Theorem~\hyperlink{thA}{A}. Thus, there cannot exist such a pair of vertex-disjoint red edges or such a pair of vertex-disjoint blue edges in $G[X_1,Y_1]$. Similarly, there cannot be a pair of vertex-disjoint red edges in $G[X_1,Y_2]$ or in $G[X_2,Y_1]$. Thus, after removing at most three vertices from each of~$X_1$ and~$Y_1$ and one vertex from from each of~$X_2$ and~$Y_2$, we may assume that
\begin{itemize}
\labitem{HB8a}{HB8a} the green graph $G_3[X_1,Y_1]$ is complete; and
\labitem{HB8b}{HB8b} the are no red edges in $G[X_1,Y_2]\cup G[X_2,Y_1]$.
\end{itemize}

Then, recalling~(\ref{IB2}), we have 
\begin{align}
\label{IB3}
|X_1|, |Y_1|&\geq (\aII-11\eta^{1/64})n,
& |X_2|, |Y_2|&\geq (\half \aI - 6\eta^{1/64})n.
\end{align}
We now consider $Z$. Defining $Z_G$ to be the set of vertices in $Z$ having a green edge to both $X_1$ and $Y_1$, the following claim allows us to assume that $Z_G$ is empty:

\begin{claim}
\label{greenorred2}
If $Z_G$ is non empty, then $G$ contains either a blue cycle on exactly~$\llangle \aII n \rrangle$ vertices or a green cycle on exactly $\langle \aIII n \rangle$ vertices. 
\end{claim}

\begin{proof}
Identical to that of Claim~\ref{nogreen} with the roles of red and blue and also the roles of $\aI$ and $\aII$ exchanged. 
\end{proof}

Since $Z_G$ is empty, defining $Z_X$ to be the set of vertices in $Z$ having no green edges to $X_1$ and $Z_Y$ to be the set of vertices in $Z$ having no green edges to $Y_1$, we see that  $Z_X\cup Z_Y$ is a partition of $Z$. We may assume, without loss of generality that 
$$|X_1 \cup X_2 \cup Z_X|\geq \llangle \aI n \rrangle + \half \llangle \aII n \rrangle -1.$$ 
Given (\ref{HB6}) and (\ref{HB7}), we can obtain upper bounds on $|X_1|$, $|X_2|$, $|Y_1|$ and $|Y_2|$ as follows: By Corollary~\ref{dirac1a}, for every integer $m$ such that $16\eta^{1/64}n+2\leq m\leq |X_1|$, we know that $G_2[X_1]$ contains a blue cycle of length $m$. Thus, in order to avoid having a blue cycle of length~$\aIIna$, we may assume that $|X_1|<\aIIna$. 
By Corollary~\ref{moonmoser2}, for every even integer $m$ such \mbox{that $16\eta^{1/64}n+2\leq m \leq 2\min\{|X_1|,|X_2|\}$}, we know that $G_1[X_1,X_2]$ contains a red cycle of length $m$. Recalling (\ref{HB0}) and (\ref{IB3}), we have $|X_1|\geq(\aII-11\eta^{1/64})n\geq\half\aI n$. Thus, in order to avoid having a red cycle on exactly~$\aIna$ vertices, we may assume that $|X_2|<\half\aIna$.
In summary, we~have
\begin{equation}
\label{IB7}
\left.
\begin{aligned}
\,\,\,\,\quad\quad\quad\quad\quad\quad\quad\quad\quad\,\,\,\,
(\aII-11\eta^{1/64})n&\leq |X_1|< \aIIna, 
\quad\quad\quad\quad\quad\quad\quad\quad\quad\\
(\half \aI - 6\eta^{1/64})n&\leq |X_2|< \half\aIna.
\end{aligned}
\right\}\!
\end{equation}
Let
\begin{align*}
Z_R&=\{ z\in Z_X \text{ such that } z \text{ has at least } |X_1|-64\eta^{1/64} \text{ red neighbours in } X_1\};\text{ and}
\\ Z_B&=W\backslash Z_R=\{ z\in Z_X \text{ such that } z \text{ has at least } 64\eta^{1/64} \text{ blue neighbours in } X_1\}.
\end{align*}

\begin{figure}[!h]
\centering
\includegraphics[height=22mm, page=10]{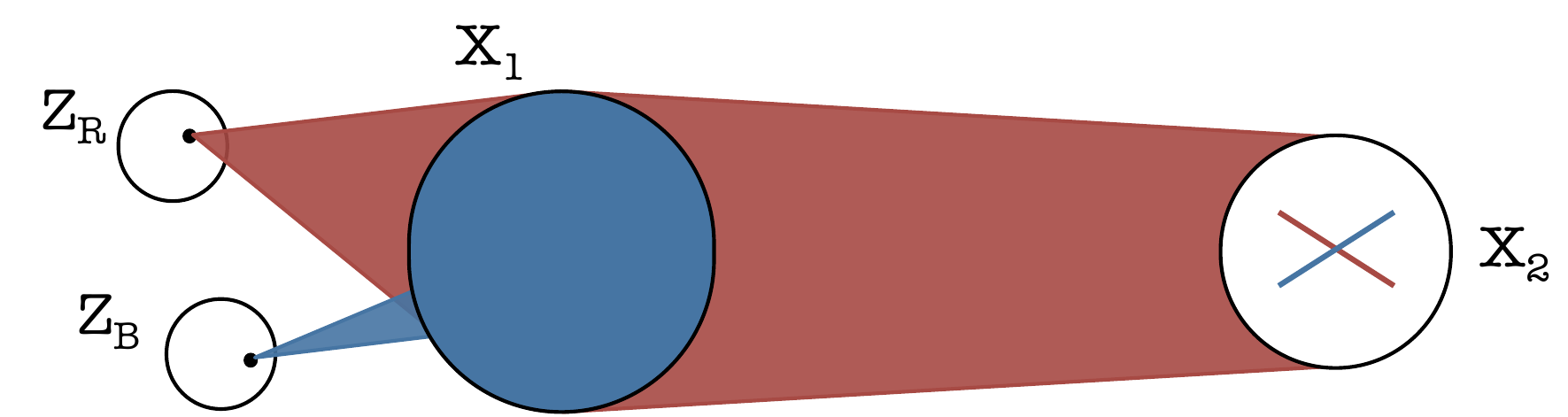}
\vspace{0mm}\caption{Partition of $Z_X$ into $Z_R\cup Z_B$.}
  
\end{figure}

Then, we have $Z=Z_B\cup Z_R$. Thus, either $|X_1\cup Z_B|\geq\llangle \aII n \rrangle$ or $|X_2\cup Z_R|\geq\half\llangle \aI n \rrangle$.

If $|X_1\cup Z_B|\geq\llangle \aII n \rrangle$, then we show that $G_2[X_1\cup Z_B]$ contains a long blue cycle as follows: Let~$X$ be any set of~$\llangle \aII n \rrangle$ vertices from $X_1 \cup Z_B$ consisting of every vertex from~$X_1$ and~$\llangle \aII n \rrangle-|X_1|$ vertices from~$Z_B$. By (\ref{HB6}) and~(\ref{IB7}), the blue graph~$G_2[X]$ has at least $\llangle \aII n \rrangle-11\eta^{1/64}n$ vertices of degree at least $|X|-20\eta^{1/64}n$ and at most $11\eta^{1/64}n$ vertices of degree at least $64\eta^{1/64}n$. Thus, by Theorem~\ref{chv},~$G[X]$ contains a blue cycle on exactly~$\llangle \aII n \rrangle$ vertices.

Thus, we may, instead, assume that $|X_2\cup Z_R|\geq\half\llangle \aI n \rrangle$, in which case, we consider the red graph $G_2[X_1,X_2\cup Z_R]$. Given the relative sizes of~$X_1$ and~$X_2\cup Z_R$ and the large minimum-degree of the graph, we can use Theorem~\ref{moonmoser} to give a red cycle on exactly~$\llangle \aI n \rrangle$ vertices as follows: By (\ref{HB0}) and~(\ref{IB7}), we have $|X_1|\geq\half\llangle\aI n \rrangle$ and may choose subsets $\widetilde{X}_1\subseteq X_1$, $\widetilde{X}_2\subseteq X_2\cup Z_R$ such that $\wX_2$ includes every vertex of $X_2$, $
|\widetilde{X}_1|=|\widetilde{X}_2|=\half\llangle\aI n \rrangle$ and $|\wX_2\cap Z_R|\leq6\eta^{1/64}n$. Recall, from (\ref{HB7}), that $G_1[X_1,X_2]$ is $4\eta^{1/64}k$-almost-complete and that, by definition, all vertices in $Z_R$ have red degree at least $|\widetilde{X}_1|-64\eta^{1/64}n$ in $G[\widetilde{X}_1,\widetilde{X}_2]$. Thus, since $|\wX_2 \cap Z_R|\leq 6\eta^{1/64}n$, for any pair of vertices $x_1\in\widetilde{X}_1$ and $x_2\in\widetilde{X}_2$, we have $d(x_1)+d(x_2)\geq\half\llangle\aI n \rrangle+1$. Therefore, by Theorem~\ref{moonmoser}, $G_1[\wX_1,\wX_2]$ contains a red cycle on exactly~$\llangle \aI n \rrangle$ vertices, thus completing this part of the proof.

\subsection*{Part I.C: $H \in \cH_1, K \in \cH_2$.}
By Theorem~\hyperlink{reB}{B}, this case only occurs when 
\begin{equation}
\tag{HC0}
\label{HC0}
\aII \leq \aI \leq \aII + \eta^{1/16}.
\end{equation}
Recalling that $\cG_1, \cG_2$ and $\cG_3$ are the monochromatic spanning subgraphs of the  reduced-graph $\cG$, the vertex set~$\cV$ of~$\cG$ has a natural partition into $\cX _1 \cup \cX _2 \cup \cY _1 \cup \cY _2\cup \cZ$ with

 ~\vspace{-12mm}
\begin{align*}
(\aI-2\eta^{1/32})k&\leq|\cX _1|=p\leq \aI k, &\quad
(\half\aII-2\eta^{1/32})k&\leq|\cX _2|=q\leq \half\aII k,  \\
(\aII-2\eta^{1/32})k&\leq|\cY _1|=r\leq \aII k,&
(\half\aI-2\eta^{1/32})k&\leq|\cY _2|=q\leq \half\aI k
\end{align*}
\vspace{-1mm}
such that
\begin{itemize}
\labitem{HC1}{HC1} $\cG_1[\cX _1], \cG_2[\cY _1$] are each $(1-2\eta^{1/32})$-complete (and thus connected); 
\labitem{HC2}{HC2} $\cG_2[\cX_1,\cX_2], \cG_1[\cY_1,\cY_2]$ are each $(1-2\eta^{1/32})$-complete (and thus connected);
\labitem{HC3}{HC3} $\cG[\cX_1\cup\cX_2\cup\cY_1\cup\cY_2]$ is $3\eta^4 k$-almost-complete (and thus connected);
\labitem{HC4}{HC4} $\cG[\cX _1]$ is $2\eta^{1/32}$-sparse in blue and contains no green edges, 
\\ \hphantom{.}\quad$\cG[\cY _1$] is $2\eta^{1/32}$-sparse in red and contains no green edges; and
\labitem{HC5}{HC5} $\cG[\cX_1,\cX_2]$ is $2\eta^{1/32}$-sparse in red and contains no green edges,
\\  \hphantom{.}\quad$\cG[\cY_1,\cY_2]$ is $2\eta^{1/32}$-sparse in blue and contains no green edges.
\end{itemize}
\begin{figure}[!h]
\centering
\vspace{-1mm}
\includegraphics[width=64mm, page=3]{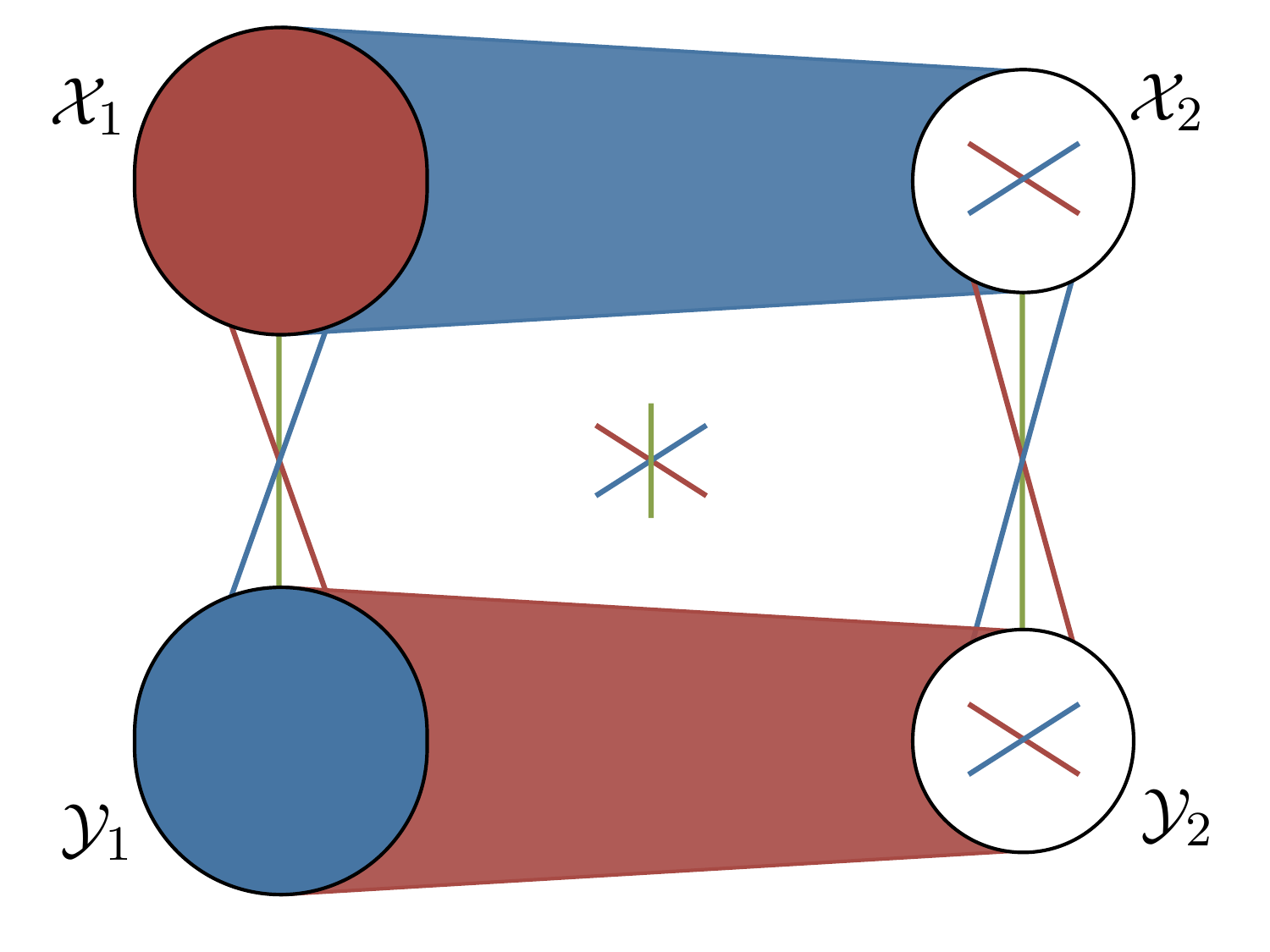}
\vspace{-2mm}\caption{Coloured structure of the  reduced-graph in Part I.C.}
  
\end{figure}

The proof that follows is essentially similar to those in Parts I.A and I.B but with some additional complications. Again, each vertex~$V_{i}$ of~$\cG$ represents a class of vertices of~$G$. We partition the vertices of~$G$ into sets $X_{1}$, $X_{2}$, $Y_{1}$, $Y_{2}$ and~$Z$ corresponding to the partition of the vertices of~$\cG$ into $\cX_1, \cX_2, \cY_1, \cY_2$ and $\cZ$. Then $X_1$ contains~$p$ clusters, $X_2$ contains~$q$ clusters $Y_1$ contains~$r$ clusters, $Y_2$ contains~$s$ clusters and we have
\begin{equation}
\label{IC1}
\left.
\begin{aligned}
\quad\,\,\,\,|X_1|&=p|V_1| \geq (\aI-2\eta^{1/32})n,
\,&\, |X_2|&=q|V_1|\geq (\half \aII-2\eta^{1/32})n,\quad\quad\\
|Y_1|&=r|V_1| \geq (\aII-2\eta^{1/32})n,
\quad&\quad |Y_2|&=s|V_1|\geq (\half \aI-2\eta^{1/32})n.
\end{aligned}
\right\}\!
\end{equation}

We may then show that $G$ has similar coloured structure to the  reduced-graph $\cG$:
\begin{claim}
\label{G-struct-C}
Given~$G$ as described, we can \textit{remove} at most $9\eta^{1/64}n$ vertices from each of $X_1$ and $Y_1$ and at most $4\eta^{1/64}n$ vertices from each of $X_2$ and $Y_2$ so that the following pair of conditions hold:
\end{claim}
\begin{itemize}
\labitem{HC6}{HC6} $G_1[X_1]$ and $G_2[Y_1]$ are each $8\eta^{1/64}n$-almost-complete;
\labitem{HC7}{HC7} $G_2[X_1,X_2]$ and $G_1[Y_1,Y_2]$ are each $4\eta^{1/64}n$-almost-complete.
\end{itemize}

\begin{proof}
The proof for $G_1[X_1]$ and $G_2[X_1,X_2]$ is identical to that of Claim~\ref{G-struct}. The proof for $G_2[Y_1]$ and $G_1[Y_1,Y_2]$ is identical but with the roles of red and blue exchanged.
\end{proof}

Having discarded some vertices, recalling~(\ref{IC1}), we now have
\begin{equation}
\label{IC2}
\left.
\begin{aligned}
\,\,\quad\quad\,\,\,\,\quad|X_1| & \geq (\aI-10\eta^{1/64})n,
\quad\quad&\quad\quad|X_2| & \geq (\half \aII-5\eta^{1/64})n,\quad\quad\quad\,\,\,\,\,
\\|Y_1| & \geq (\aII-10\eta^{1/64})n,
\,\,\,\,\,\quad&\quad\quad|Y_2| & \geq (\half \aI-5\eta^{1/64})n.
\end{aligned}
\right\}\!
\end{equation}
and proceed to consider $G[X,Y]$:
\begin{claim}
\label{nogreen3}
If there exist distinct vertices $x_1,x_2 \in X_1$ and $y_1, y_2\in Y_1$ such~that~$x_1y_1$ and~$x_2y_2$ are coloured red, then~$G$ contains a red cycle of length exactly~$\llangle \aI n \rrangle$.
\end{claim}

\begin{proof}
Suppose there exist distinct vertices $x_1,x_2\in X_1$ and $y_1,y_2\in Y_2$ such that the edges $x_1y_1$ and $x_2y_2$ are coloured red.  Then, letting $\widetilde{X}_1$ be any set of of~{$2\lceil \tfrac{1}{4} \llangle \aI n \rrangle \rceil -1$} vertices in $X_1$ such that $x_1, x_2\in \widetilde{X}_1$, by (\ref{HC6}), every vertex in $\widetilde{X}_1$ has degree at least $|\widetilde{X}_1|-8\eta^{1/64}n$ in $G[\widetilde{X}_1]$. Since $\eta\leq(\aI/100)^{64}$, we have $|\widetilde{X}_1|-8\eta^{1/64}n\geq \half|\widetilde{X}_1|+2$, so, by Corollary~\ref{dirac2}, there exists a red path from~$x_1$ and $x_2$ on exactly $2\lceil \tfrac{1}{4} \llangle \aI n \rrangle \rceil -1$ vertices in $X_1$.

Let $\widetilde{Y}_2$ be any set of $\lfloor \tfrac{1}{4} \llangle \aI n \rrangle \rfloor \geq 4\eta^{1/64} n +2 $ vertices from ${Y}_2$. By~(\ref{HC7}), $y_1$ and $y_2$ each have at least two neighbours in $\widetilde{Y}_2$. Also, by~(\ref{HC0}) and (\ref{HC7}), every vertex in $\widetilde{Y}_2$ has degree at least $\half|Y_1| +\half|\widetilde{Y}_2| +1$ in $G[Y_1,\widetilde{Y}_2]$. Finally, by (\ref{HC0}), we have $|Y_1|\geq\ell+1$. Thus, by Lemma~\ref{bp-dir},  $G_2[Y_1,\widetilde{Y}_2]$ contains a path on exactly $2\lfloor \tfrac{1}{4} \llangle \aI n \rrangle \rfloor +1$ vertices from~$y_1$ to~$y_2$. Then, combining the red edges $x_1y_1$ and $x_2y_2$ with the red paths found in $G[X_1]$ and $G[Y_1,Y_2]$ gives a red cycle on exactly~$\llangle \aI n \rrangle$ vertices. 
\end{proof}

The existence of a red cycle on~$\aIna$ vertices would be sufficient to complete the proof. Thus, there cannot exist such a pair of vertex-disjoint red edges. Similarly, there cannot exist a pair of vertex-disjoint red edges in $G[X_1,Y_2]$ or a pair of vertex-disjoint blue edges in $G[X_1,Y_1]$ or $G[X_2,Y_1]$.

\begin{figure}[!h]
\centering
\includegraphics[width=64mm, page=10]{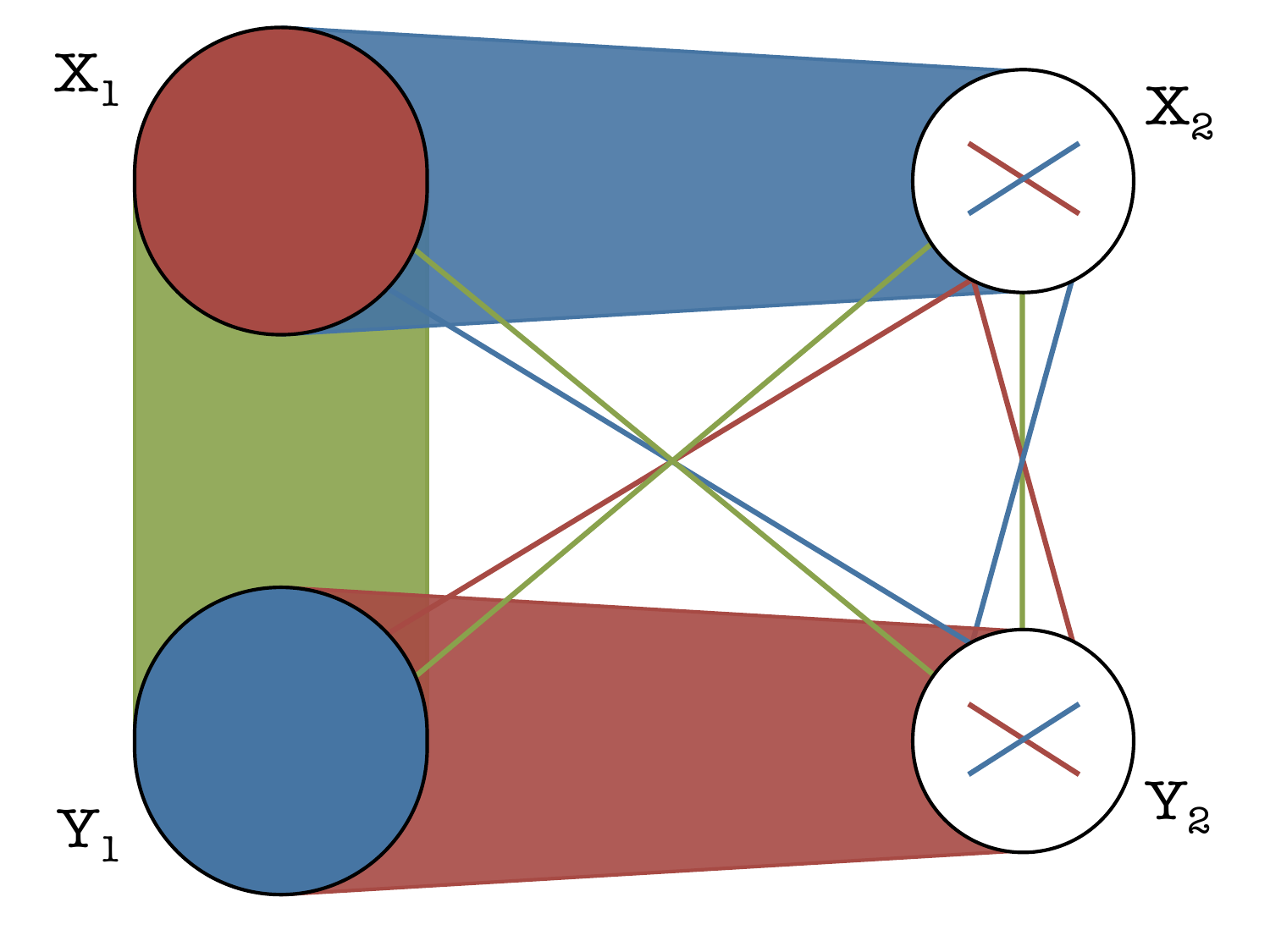}
\vspace{0mm}\caption{Colouring of $G$ after Claim~\ref{nogreen3}.}
\end{figure}

Thus, after removing at most three vertices from each of $X_1$, $Y_1$ and at most one vertex from each of $X_2,Y_2$, we may assume that 
\begin{itemize}
\labitem{HC8a}{HC8a} the green graph $G_3[X_1,Y_1]$ is complete; and
\labitem{HC8b}{HC8b} the are no red edges in $G[X_1,Y_2]$ and no blue edges in $G[X_2,Y_1]$.
\end{itemize}

Then, recalling~(\ref{IC2}), we have
\begin{equation}
\label{IC3}
\left.
\begin{aligned}
\,\,\quad\quad\,\,\,\,\quad|X_1| & \geq (\aI-11\eta^{1/64})n,
\quad\quad&\quad\quad|X_2| & \geq (\half \aII-6\eta^{1/64})n,\quad\quad\quad\,\,\,\,\,
\\|Y_1| & \geq (\aII-11\eta^{1/64})n,
\,\,\,\,\,\quad&\quad\quad|Y_2| & \geq (\half \aI-6\eta^{1/64})n.
\end{aligned}
\right\}\!
\end{equation}

We now consider $W$. Defining $Z_G$ to be the set of vertices in $Z$ having a green edge to both $X_1$ and~$Y_1$, the following claim allows us to assume that $Z_G$ is empty:

\begin{claim}
\label{greenorred3}
If $Z_G$ is non empty, then $G$ contains either a red cycle on exactly~$\llangle \aI n \rrangle$ vertices, a blue cycle on exactly~$\llangle \aII n \rrangle$ vertices or a green cycle on exactly $\langle \aIII n \rangle$ vertices. 
\end{claim}

\begin{proof}
Suppose that $Z_G$ is non-empty. Then, there exists $z\in Z$, $x\in X_1$, $y\in Y_1$ such that $xz$ and $yz$ are both coloured green. Recalling (\ref{HC8a}), $G_3[X_1,Y_1]$ is complete, thus we may obtain an odd cycle of any odd length up to $|X_1|+|Y_1|+1$ in $G[Z_G,X_1\cup Y_1]\cup G[X_1,Y_1]$. Therefore, to avoid having a a green cycle on exactly $\langle \aIII n \rangle$ vertices, we assume that $|X_1|+|Y_1|+1<\langle \aIII n \rangle$.

Then, considering,~(\ref{a3notbig}), (\ref{HC0}) and~(\ref{IC3}), we have
$$|X_1|+|Y_1|\geq(\aI+\aII-22\eta^{1/64})n\geq(\aIII-24\eta^{1/64})n.$$

In that case, suppose that, for some $x_a, x_b\in X_1$ and $y_a, y_b\in Y_1$, there exist green paths 
\begin{align*}
\quad\quad\quad\quad&P_1\text{ from }x_a\in X_1\text{ to }x_b\in X_1&&\text{ on }2\lceil12\eta^{1/64}n\rceil+1\text{ vertices }\text{in } G[X_1,Y_2],
\quad\quad\quad\quad\\
\quad\quad\quad\quad&P_2\text{ from }y_a\in Y_1\,\text{ to }y_b\in Y_1&&\text{ on }2\lceil12\eta^{1/64}n\rceil+1\text{ vertices }\text{in } G[Y_1,Y_2].
\quad\quad\quad\quad
\end{align*}
Then, since $G_3[X_1,Y_1]$ is complete and $Z_G$ is non-empty, $P_1$ and $P_2$ could be used along with edges from $G[X_1,Y_1]$ and $G[z,X_1\cup Y_1]$ to give an odd green cycle on exactly $\langle \aIII n\rangle$ vertices. 
Thus, at most one of $G[X_1,Y_2]$, $G[X_2,Y_1]$ contains a green path on $2\lceil12\eta^{1/64}n\rceil+1$ vertices. Thus, by Theorem~\ref{th:eg}, either (a) $G[X_1,Y_2]$ contains at most $16\eta^{1/64}n^2$ green edges or (b) $G[X_2,Y_1]$ contains at most $16\eta^{1/64}n^2$ green edges.

(a) If $G[X_1,Y_2]$ contains at most $16\eta^{1/64}n^2$ green edges,  then recalling (\ref{HC8b}), we know that $G[X_1,Y_2]$ is complete, contains no red edges and contains at most $16\eta^{1/64}n^2$ green edges. Thus, after removing at most $4\eta^{1/128}$ vertices from each of~$X_1$ and $Y_2$ we may assume that $G_2[X_1,Y_2]$ is $4\eta^{1/128}n$-almost-complete. Recall from (\ref{HC6}) that $G_2[X_1,X_2]$ is $4\eta^{1/64}$-almost-complete. 
Thus, $G_2[X_1,X_2\cup Y_2]$ is~$6\eta^{1/128}$-almost-complete.

Having discarded these vertices, recalling~(\ref{IC3}), since $\eta<10^{-50}$, we have 
\begin{equation}
\label{IC5}
\left.
\begin{aligned}
\,\,\,\quad\quad\,\,\,\,\quad|X_1| & \geq (\aI-7\eta^{1/128})n,
\quad\quad&\quad\quad|X_2| & \geq (\half \aII-2\eta^{1/128})n,\quad\quad\quad\quad
\\|Y_1| & \geq (\aII-2\eta^{1/128})n,
\quad\quad&\quad\quad|Y_2| & \geq (\half \aI-7\eta^{1/128})n.
\end{aligned}
\right\}\!
\end{equation}

Given the bounds in (\ref{IC5}), there exist subsets $\wX_1\subseteq X_1$ and $\wX_2\subseteq X_2\cup Y_2$ such that $|\wX_1|=|\wX_2|=\half\aIIna$. Then, by Theorem~\ref{moonmoser}, $G_2[\wX_1,\wX_2]$ is Hamiltonian and, thus, provides a blue cycle on exactly~$\aIIna$ vertices.

(b) If instead $G[X_2,Y_1]$ contains at most $16\eta^{1/64}n^2$ green edges, then, after removing at most $4\eta^{1/64}n$ vertices from each of~$X_2$ and $Y_1$, we may assume that $G_1[X_2,Y_1]$ is $4\eta^{1/128}n$-almost-complete. Recall from (\ref{HC6}) that $G_2[Y_1,Y_2]$ is $4\eta^{1/64}$-almost-complete. 
Thus, $G_1[Y_1,X_2\cup Y_2]$ is $6\eta^{1/128}$-almost-complete. Having discarded these vertices, recalling~(\ref{IC3}), since $\eta<10^{-20}$, we have 
\begin{equation}
\label{IC5b}
\left.
\begin{aligned}
\,\,\,\quad\quad\,\,\,\,\quad|X_1| & \geq (\aI-2\eta^{1/128})n,
\quad\quad&\quad\quad|X_2| & \geq (\half \aII-7\eta^{1/128})n,\quad\quad\quad\quad
\\|Y_1| & \geq (\aII-7\eta^{1/128})n,
\quad\quad&\quad\quad|Y_2| & \geq (\half \aI-2\eta^{1/128})n.
\end{aligned}
\right\}\!
\end{equation}

Given these bounds, there exist subsets $\wY_1\subseteq X_1$ and $\wY_2\subseteq X_2\cup Y_2$ such that $|\wY_1|=|\wY_2|=\half\aIna$. Then, by Theorem~\ref{moonmoser}, $G_1[\wY_1,\wY_2]$ is Hamiltonian and, thus, provides a red cycle on exactly~$\aIna$ vertices.
\end{proof}

\begin{figure}[!h]
\centering
\includegraphics[height=50mm, page=11]{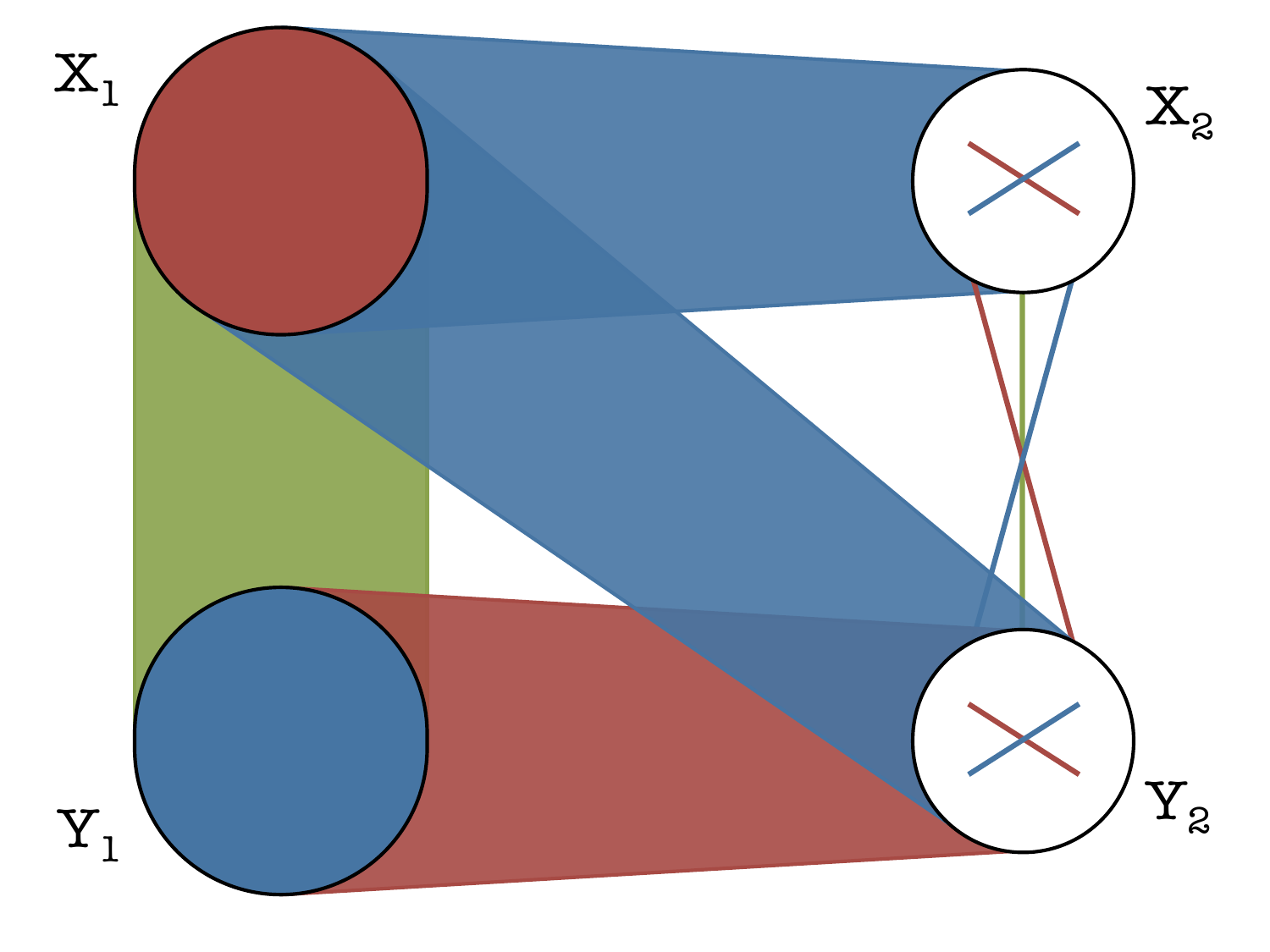}
\vspace{0mm}\caption{Colouring of $G[X_1,Y_2]$ in Claim~\ref{greenorred3}(a).}
\end{figure}

The existence of a red cycle on exactly~$\llangle \aI n \rrangle$ vertices, a blue cycle on exactly~$\llangle \aII n \rrangle$ vertices or a green cycle on exactly $\langle \aIII n \rangle$ vertices, as offered by Claim~\ref{greenorred3}, would be sufficient to complete the proof of Theorem~\hyperlink{thA}{A}. We may, therefore, instead assume that~$Z_G$ is empty.
Then, defining $Z_X$ to be the set of vertices in $Z$ having no green edges to $X_1$ and $Z_Y$ to be the set of vertices in $Z$ having no green edges to $Y_1$, we see that  $Z_X\cup Z_Y$ is a partition of $Z$. Thus either 
\begin{align*}
|X_1 \cup X_2 \cup Z_X| & \geq \llangle \aI n \rrangle + \half \llangle \aII n \rrangle -1 && \text{or} & |Y_1 \cup Y_2 \cup Z_Y| & \geq \llangle \aI n \rrangle + \half \llangle \aII n \rrangle -1.
\end{align*}
In the first case, the proof proceeds exactly as in Part I.A. In the second case, the proof proceeds exactly as in Part I.B. Thus, we obtain either a red cycle on exactly~$\llangle \aI n \rrangle$ or a blue cycle on exactly~$\llangle \aI n \rrangle$, completing Part I of the proof of Theorem~\hyperlink{thA}{A}.
Note that, in the case that $\aI\geq\aIII$,
we have dealt with the only coloured structure that can arise. Thus, since the graph providing the corresponding lower bound has already been seen, we have proved that, given $\aI,\aII,\aIII>0$ such that $\aI\geq\aII,\aIII$, there exists $n_{A^{\prime}}=n_{A^{\prime}}(\aI,\aII,\aIII)$ such that, for $n> n_{A^{\prime}}$,
$$ R(C_{\llangle \alpha_{1}n \rrangle}, C_{\llangle \alpha_{2}n \rrangle}, C_{\langle \alpha_{3}n \rangle }) = 2\llangle \alpha_{1}n \rrangle+ \llangle \alpha_{2}n \rrangle - 3.$$

\section{Proof of the main result  -- Part II -- Case (v)}
\label{s:p12}

Suppose that $\cG$ contains a subgraph from 
\begin{align*}
\cK\left((\half\aI-14000\eta^{1/2})k, (\half\aI-14000\eta^{1/2})k, (\aIII-68000\eta^{1/2})k, 4\eta^4 k\right).
\end{align*}

In that case, by Theorem~\hyperlink{reB}{B}, we may assume that
$
\label{a3big} 
\aIII\geq\tfrac{3}{2}\aI+\half\aII-10\eta^{1/2}.
$

 Recalling that $\cG_1, \cG_2$ and $\cG_3$ are the monochromatic spanning subgraphs of the  reduced-graph~$\cG$, since $\eta<10^{-20}$, the vertex set~$\cV$ of~$\cG$ has partition into $\cX _1 \cup \cX _2 \cup \cX_3\cup \cZ$~with 
\begin{align*}
(\half\aI-\eta^{1/4})k\leq |\cX_1| &=p\leq \half\aI k,\\
(\half\aII-\eta^{1/4})k\leq |\cX_2| &=q\leq \half\aII k,\\
(\aIII-\eta^{1/4})k \leq |\cX_3| &=r\leq \aIII k.
\end{align*}
such that all edges present in $\cG[\cX_1,\cX_3]$ are coloured exclusively red, all edges present in $\cG[\cX_2,\cX_3]$ are coloured exclusively blue, all edges present in $\cG[\cX_3]$ are coloured exclusively green and $\cG[\cX_1\cup \cX_2\cup \cX_3]$ is $4\eta^4 k$-almost-complete.
\begin{figure}[!h]
\centering
\includegraphics[width=64mm, page=15]{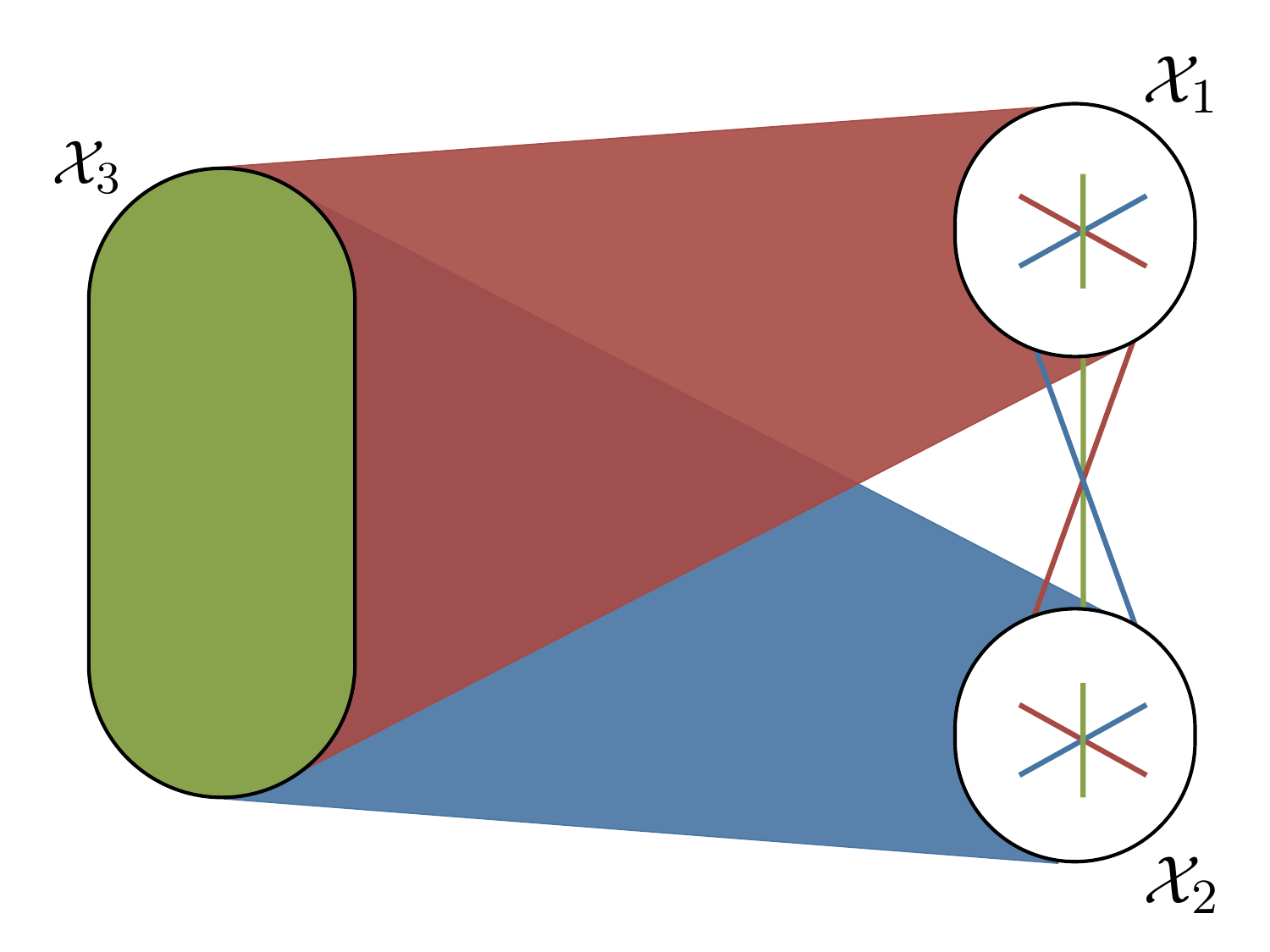}
\vspace{0mm}\caption{Coloured structure of the  reduced-graph in Part II.}
  
\end{figure}

Again, each vertex $V_{i}$ of $\cG=(\cV,\cE)$ represents a cluster of vertices of $G=(V,E)$ with
$$(1-\eta^4)\frac{N}{K}\leq |V_i|\leq \frac{N}{K}$$ 
and, since $n> \max\{n_{\ref{th:blow-up}}(2,1,0,\eta), n_{\ref{th:blow-up}}(\half,\half,1,\eta)\}$, we have $$
|V_i|\geq \left(1+\frac{\eta}{24}\right)\frac{n}{k}> \frac{n}{k}.$$ 
We partition the vertices of~$G$ into sets $X_{1}$, $X_{2}$, $X_{3}$ and~$Z$ corresponding to the partition of the vertices of~$\cG$ into $\cX_1, \cX_2, \cX_3$ and $\cZ$. Then~$X_1$ contains~$p$ clusters of vertices, $X_2$ contains~$q$ clusters, $X_3$ contains~$r$ clusters and we have  
\begin{equation}
\label{Y0}
\left.
\begin{aligned}
\quad\,\,\,
|X_1| &= p|V_1|\geq(\half\aI-\eta^{1/4})n,
\quad\quad&\quad\quad |X_3| & = r|V_1|\geq(\aIII-\eta^{1/4})n,
\quad\quad
\\
 |X_2| & = q|V_1|\geq(\half\aII-\eta^{1/4})n.
\end{aligned}
\right\}\!\!
\end{equation}
In what follows, we will remove vertices from $X_1, X_2, X_3$,  by moving them into~$Z$. We prove the below claim which essentially tells us that $G$ has a similar coloured structure to $\cG$.
\begin{claim}
\label{G-structII}
We can remove at most $7\eta^{1/2}n$ vertices from~$X_1$,
$7\eta^{1/2}n$ vertices from~$X_2$,
and $24\eta^{1/2}n$ vertices from~$X_3$
such that the following holds:
\begin{itemize}
\item[(i)] $G_1[X_1,X_3]$ is $7\eta^{1/2}n$-almost-complete;
\item[(ii)] $G_2[X_2,X_3]$ is $7\eta^{1/2}n$-almost-complete;
\item[(iii)] $G_3[X_3]$ is $10\eta^{1/2}n$-almost-complete.
\end{itemize}
 \end{claim}

\begin{proof} 
Consider $G[X_1,X_3]$. We can bound the the number of non-red edges in $G[X_1,X_3]$ by 
$$4\eta^4 rk \left(\frac{N}{K}\right)^2+2\eta pr\left(\frac{N}{K}\right)^{2},$$
where the first term bounds the number of non-red edges between non-regular pairs, the second bounds the number of non-red edges between pairs of clusters that are joined by a red edge in the  reduced-graph.
Since $K\geq k$, $N\leq 3n$, $p \leq \half\aI k\leq \half k$ and $r\leq \aIII k \leq 2k$, we obtain
$$e(G_2[X_1,X_3])+e(G_3[X_1,X_3])\leq (72 \eta^4+18\eta)\leq 40\eta n^2.$$
Since $G[X_1,X_3]$ is complete and contains at most $40\eta n^2$ non-red edges, there are at most $7\eta^{1/2}$ vertices in~$X_1$ with red degree to~$X_3$ at most $|X_3|-7\eta^{1/2}n$ and at most $7\eta^{1/2}$ vertices in~$X_3$ with red degree to~$X_1$ at most $|X_1|-7\eta^{1/2}n$. Re-assigning these vertices to~$Z$ results in every vertex in~$X_1$ having degree in $G_1[X_1,X_3]$ at least $|X_3|-7\eta^{1/2}n$ and every vertex in~$X_3$ having degree in $G_1[X_1,X_3]$ at least $|X_1|-7\eta^{1/2}n$.
We repeat the above for~$G[X_2,X_3]$, removing vertices such that every (remaining) vertex in~$X_2$ has degree in $G_2[X_2,X_3]$ at least $|X_3|-7\eta^{1/2}n$ and every (remaining) vertex in~$X_3$ has degree in $G_2[X_1,X_3]$ at least $|X_2|-7\eta^{1/2}n$.

Finally, we consider the complete three-coloured graph $G[X_3]$ and recall that $\cG[\cX_3]$ is $4\eta^4k$-almost-complete and has only green edges. Given the structure of~$\cG$, we can bound the number of non-green edges in $G[X_3]$ by
$$r \binom{N/K}{2} + 4\eta^4 rk \left(\frac{N}{K}\right)^2+2\eta\binom{r}{2}\left(\frac{N}{K}\right)^{2},$$ 
where the first term counts the number of non-green edges within the clusters,
the second counts the number of non-green edges between non-regular pairs,
the third counts the number of non-green edges between regular pairs.
%
%
%
%
%
%
Since $K\geq k, \eta^{-1}$, $N\leq 3n$ and $r\leq \aIII k \leq 2k$, we obtain
$$e(G_1[X_3])+e(G_2[X_3])\leq [9\eta + 72\eta^{4} +36 \eta]n^2\leq 50\eta n^2.$$
Since $G[X_3]$ is complete and contains at most $50\eta n^2$ non-green edges, there are at most $10\eta^{1/2}n$ vertices with green degree at most $|X_3|-10\eta^{1/2}n$. After re-assigning these vertices to~$Z$, every vertex in $G[X_3]$ has red degree at least $|X_3|-10\eta^{1/2}n$, thus completing the proof of Claim~\ref{G-structII}.
\end{proof}

We now proceed to the {end-game}: Observe that, by Corollary~\ref{moonmoser2} there exist red cycles in $G[X_1,X_3]$ of every (non-trivial) even length up to twice the size of the smaller part and blue cycles in $G[X_2,X_3]$ of every (non-trivial) even length up to twice the size of the smaller part. Similarly, by Corollary~\ref{dirac1a}, there exist green cycles in $G[X_3]$ of every (non-trivial) length up to $|X_3|$. We will show that it is possible to augment each of $X_1, X_2, X_3$ with vertices from~$Z$ while maintaining this  property. Then, considering the sizes of each part, there must, in fact, be a cycle of appropriate length, colour and parity to complete the proof:

Since $24\eta^{1/2}\leq\eta^{1/4}$, recalling (\ref{Y0}), having discarded some vertices while proving Claim~\ref{G-structII}, we have
\begin{align*}
(\half\aI-2\eta^{1/4})n \leq |X_1|&< \half\llangle \aI n \rrangle, \\
(\half\aII-2\eta^{1/4})n \leq |X_2|&< \half\llangle \aII n \rrangle,\\
(\aIII-2\eta^{1/4})n \leq |X_3|&< \langle \aIII n \rangle,
\end{align*}
and know that $G_1[X_1,X_3]$, $G_2[X_2,X_3]$ and $G_3[X_3]$ are each $\eta^{1/4}n$-almost-complete.

Now, let
\vspace{-1mm}
 $$Z_G=\{z\in Z : w\text{ has at least }4\eta^{1/4}n\text{ green edges to }X_3\}.$$ Suppose that $|X_3\cup Z_G|\geq\langle\aIII n\rangle$. Then, since $|X_3|\leq \langle\aIII n\rangle$, we may choose a subset~$X$ of size $\langle \aIII n \rangle$ from $X_3\cup Z_G$ which includes every vertex from $X_3$ and $\langle \aIII n \rangle-|X_3|$ vertices from $Z_G$. Then, $G[X]$ has at least $(\aIII-2\eta^{1/4})n$ vertices of degree at least $(\aIII-4\eta^{1/4})n$ and at most $2\eta^{1/4}$ vertices of degree at least $4\eta^{1/4} n$, so, by Theorem~\ref{chv}, $G[X]$ is Hamiltonian and, thus, contains a green cycle of length exactly $\langle \aIII n \rangle$. The existence of such a cycle would be sufficient to complete the proof of Theorem~\hyperlink{thA}{A} in this case, so we may assume, instead that $|X_3\cup Z_G|<\langle\aIII n\rangle$.
Thus, letting $Z_{RB}=Z\backslash Z_G$, we may assume that 
$$|X_1|+|X_2|+|Z_{RB}|\geq\half\llangle\aI n\rrangle +\half\llangle \aII n \rrangle -1,$$
and, defining
\vspace{-2mm}
\begin{align*}
Z_R&=\{z\in Z : z\text{ has at least }\half|X_3|-2\eta^{1/4}n\text{ red edges to }X_3\},\\
Z_B&=\{z\in Z : z\text{ has at least }\half|X_3|-2\eta^{1/4}n\text{ blue edges to }X_3\},
\end{align*}
we may assume, without loss of generality, that $|X_1\cup Z_R|\geq \llangle \aI n \rrangle.$
In that case, let $\widetilde{Z}_R\subseteq Z_R$ be such that $|X_1|+|\widetilde{Z}_R|=\half\aIna$. Then, observing that any $z\in Z_R$ and $x\in X_1$ have at least $\half|X_3|-3\eta^{1/4}n\geq|\widetilde{Z}_R|$ common neighbours and that any $x,y\in X_1$ have at least $(\aIII -2\eta^{1/4})n\geq\aI n$ common neighbours, we can greedily construct a red cycle of length~$\llangle \aI n \rrangle$ using all the vertices of $X_1\cup \widetilde{Z}_R$, thus completing this part of the proof of Theorem~\hyperlink{thA}{A}.

\section{Proof of the main result  -- Part III -- Case (vi)}
\label{s:p13}

We now consider the final case, thus, we suppose that $\cG$ contains a subgraph $K^*$ from $\cK_1^{*}\cup \cK_2^*$, where
\begin{align*}
\cK_{1}^{*}=\cK^*\big((\half\aI-97\eta^{1/2})k, (\half\aI-97&\eta^{1/2})k, (\half\aI+102\eta^{1/2})k, \\&(\half\aI+102\eta^{1/2})k, (\aIII-10\eta^{1/2})k, 4\eta^4 k\big),\hphantom{l}
\end{align*}
\vspace{-12mm}
\begin{align*}
\cK_{2}^{*}=\cK^*\big((\half\aI-97\eta^{1/2})k, (\half\aII-97\eta^{1/2})k, (\tfrac{3}{4}&\aIII-140\eta^{1/2})k, \\& 100\eta^{1/2}k, (\aIII-10\eta^{1/2})k, 4\eta^4 k\big).
\end{align*}
Recalling, Theorem~\hyperlink{reB}{B}, we may assume that 
\begin{equation}
\tag{K0}
\label{K0} 
\aIII\geq\tfrac{3}{2}\aI+\half\aII-10\eta^{1/2}.
\end{equation}
Recalling that $\cG_1, \cG_2$ and $\cG_3$ are the monochromatic spanning subgraphs of the  reduced-graph, we have a partition of the vertex set $\cV$ of $\cG$ into $\cX_1\cup\cX_2\cup\cY_1\cup\cY_2\cup\cZ$ 
such that all edges present in $\cG[\cX_1,\cY_1]\cup\cG[\cX_2,\cY_2]$ are coloured exclusively red, all edges present in $\cG[\cX_1,\cY_2]\cup\cG[\cX_2,\cY_1]$ are coloured exclusively blue and all edges present in $\cG[\cX_1,\cX_2]\cup\cG[\cY_1,\cY_2]$ are coloured exclusively green. Also, for any $\cZ\subseteq\cX_1\cup\cX_2\cup\cY_1\cup\cY_2$, $\cG[\cZ]$ is $4\eta^4 k$-almost-complete.

\begin{figure}[!h]
\centering
\includegraphics[width=64mm, page=17]{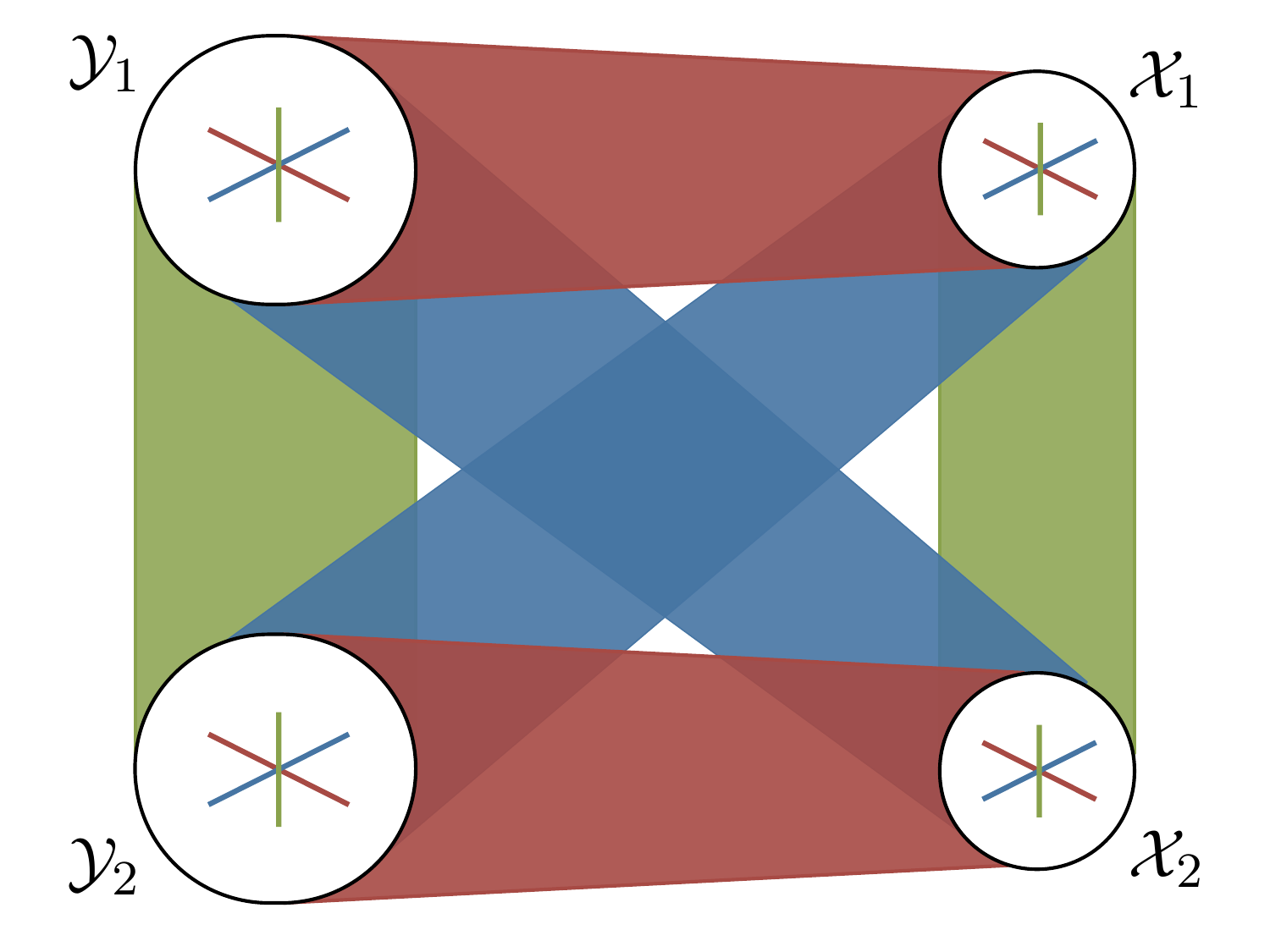}
\vspace{0mm}\caption{Initial coloured structure of the  reduced-graph in Part III.}
\label{fig-III-initial}
\end{figure}

In each case, before proceeding to consider $G$, we we must determine more about the the coloured structure of the  reduced-graph $\cG$. In the process, we will discard further vertices from $\cX_1, \cX_2, \cY_1$ and~$\cY_2$. As in Section~\ref{s:stabp2}, these discarded vertices are considered as having been re-assigned to $\cZ$.

\subsection*{Part III.A:  $K^*\in\cK_1^{*}$.}
In this case, we have a partition of the vertex set $\cV$ of $\cG$ into $\cX_1\cup\cX_2\cup\cY_1\cup\cY_2\cup\cZ$ with
\begin{align*}
|\cX_1|,|\cX_2|\geq(\half\aI-97\eta^{1/2})k, \quad\quad
|\cY_1|,\,|\cY_2|\geq(\half\aI+102\eta^{1/2})k, \quad\quad
|\cY_1|\!+\!|\cY_2|\geq(\aIII-10\eta^{1/2})k.
\end{align*}
 Observe that, since $|\cX_1|,|\cY_1|\geq(\half\aI-97\eta^{1/2})k$ and $\cG_1[\cX_1\cup\cY_1]$ is $4\eta^4 k$-complete, $\cG[\cX_1\cup\cY_1]$ has a single red component.  Similarly, $\cG[\cX_2\cup\cY_2]$ has a single red component and each of $\cG[\cX_1\cup\cY_2]$ and $\cG[\cX_2\cup\cY_1]$ has a single blue component.  

Consider~$\cG[\cY_1]$ and suppose that there exists a red matching $\cR_1$ on $198\eta^{1/2}k$ vertices in $G[\cY_1]$. Then, we have $|\cY_1\backslash\cV(\cR_1)|,|\cX_1|\geq(\half\aI-97\eta^{1/2})k$, so, by Lemma~\ref{l:eleven}, $\cG[\cX_1,\cY_1\backslash \cV(\cR_1)]$ contains a red connected-matching on at least $(\aI-196\eta^{1/2})k$ vertices, which combined with $\cR_1$ gives a red connected-matching on at least~$\aI k$ vertices.  
Thus, there can be no such red matching in $\cG[\cY_1]$. Similarly,~$\cG[\cY_1]$ cannot contain a blue matching on $198\eta^{1/2}k$ vertices. 
Thus, after discarding at most $396\eta^{1/2}k$ vertices from $\cY_1$, we may assume that all edges present in~$\cG[\cY_1]$ are coloured exclusively green. Similarly, after discarding at most $396\eta^{1/2}k$ vertices from $\cY_2$, we may assume that all edges present in~$\cG[\cY_2]$ are coloured exclusively green.

\begin{figure}[!h]
\centering
\includegraphics[width=64mm, page=18]{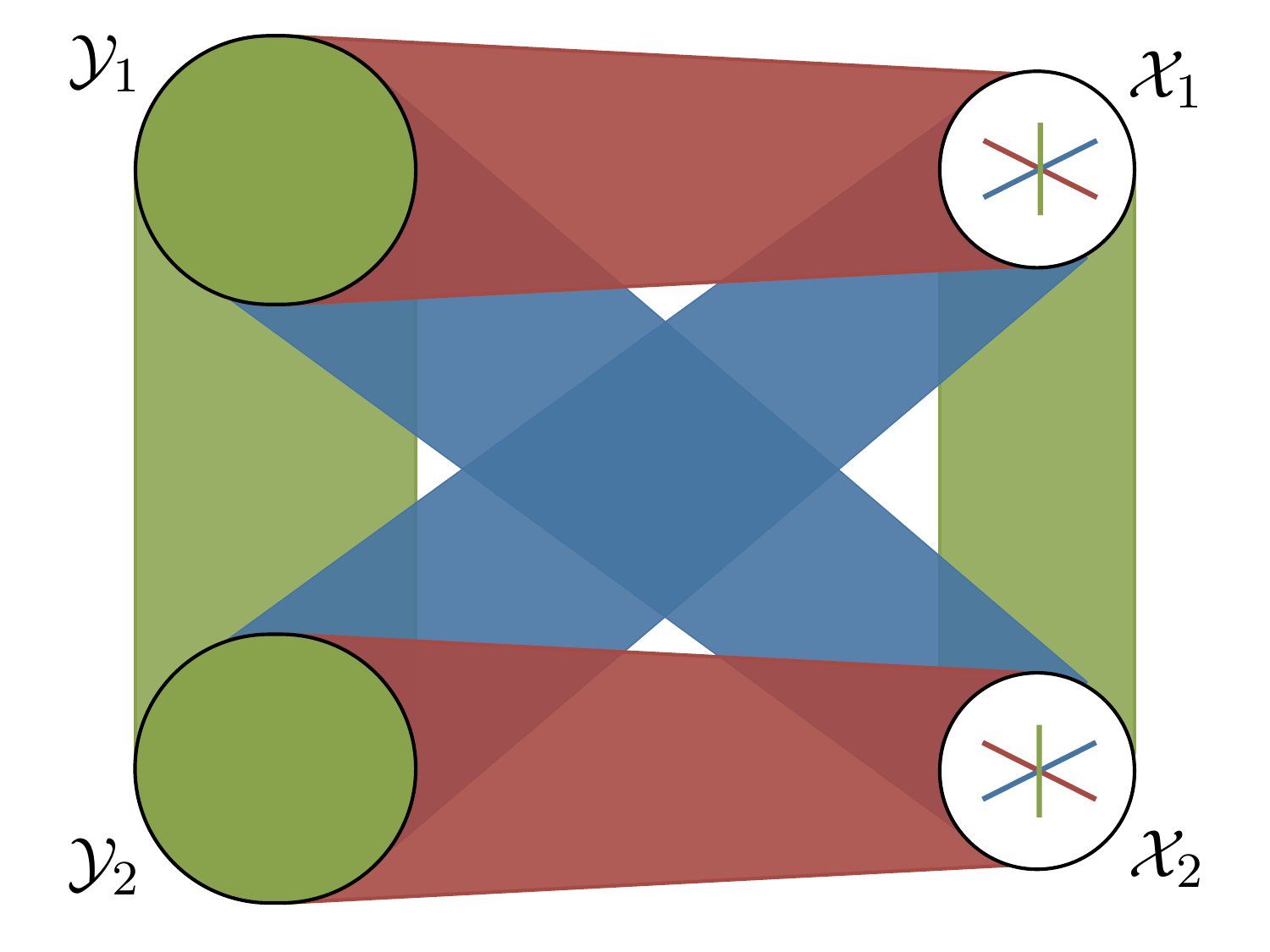}
\vspace{0mm}\caption{Colouring of $\cG[\cY_1]\cup\cG[\cY_2]$.}
  
\end{figure}

Thus, we may assume that we have a partition $\cX_1\cup\cX_2\cup\cY_1\cup\cY_2\cup\cZ$  such that 
\begin{subequations}
\begin{equation}
\left.
\begin{aligned}
\quad\quad\quad\quad\quad\quad\quad\quad\quad\quad
(\half\aI-97\eta^{1/2})k&\leq|\cX_1|=|\cX_2|=p\leq\half\aI k,
\quad\quad\quad\quad\,\\
(\half\aI-294\eta^{1/2})k&\leq|\cY_1|=r, \\ 
(\half\aI-294\eta^{1/2})k&\leq|\cY_2|=s.
\end{aligned}
\right\}\!
\end{equation}

Additionally, writing $\cY$ for $\cY_1\cup\cY_2$, we have
\begin{align}
(\aIII-802\eta^{1/2})k&\leq |\cY|=r+s\leq \aIII k \,\,
\end{align}
\end{subequations}
and know that all edges present in $\cG[\cX_1,\cY_1]\cup\cG[\cX_2,\cY_2]$ are coloured exclusively red, all edges present in $\cG[\cX_1,\cY_2]\cup\cG[\cX_2,\cY_1]$ are coloured exclusively blue and all edges present in $\cG[\cX_1,\cX_2]\cup\cG[\cY_1\cup\cY_2]$ are coloured exclusively green.

Thus far, we have obtained information about the structure of the  reduced-graph~$\cG$. The remainder of this section focuses on showing that the original graph must have a similar structure which can then be exploited to force a cycle of appropriate length, colour and parity. Again, each vertex $V_{i}$ of $\cG=(\cV,\cE)$ represents a cluster of vertices of $G=(V,E)$ with
$$(1-\eta^4)\frac{N}{K}\leq |V_i|\leq \frac{N}{K}$$ 

and that, since $n> \max\{n_{\ref{th:blow-up}}(2,1,0,\eta), n_{\ref{th:blow-up}}(\half,\half,1,\eta)\}$, we have $$
|V_i|\geq \left(1+\frac{\eta}{24}\right)\frac{n}{k}> \frac{n}{k}.$$ 

We partition the vertices of~$G$ into sets $X_{1}$, $X_{2}$, $Y_{1}$, $Y_{2}$ and~$Z$ corresponding to the partition of the vertices of~$\cG$ into $\cX_1, \cX_2, \cY_1, \cY_2$ and $\cZ$. Then~$X_1$ and $X_2$ contain~$p$ clusters of vertices, $Y_1$ contains~$q$ clusters and $Y_2$ contains~$r$ clusters. Note that, we write~$Y$ for $Y_1\cup Y_2$ and $\cY$ for $\cY_1\cup \cY_2$. Thus, we have
\begin{equation}
\label{Z0}
\left.
\begin{aligned}
\,\quad\quad\quad\quad\quad\quad\quad\quad
|X_1|=|X_2|&=p|V_1|\geq(\half\aI-97\eta^{1/2})n,
\quad\quad\quad\quad\quad\quad\quad\quad\\
|Y_1|&=r|V_1|\geq(\half\aI-294\eta^{1/2})n,\\
|Y_2|&=s|V_1|\geq(\half\aI-294\eta^{1/2})n,\\
|Y|=|Y_1|+|Y_2|&=(r+s)|V_1|\geq(\aIII-802\eta^{1/2})n.
\end{aligned}
\right\}
\end{equation}
Again, we will remove vertices from $X_1\cup X_2\cup Y_1\cup Y_2$,  by moving them into~$Z$.

The following claim  tells us that the graph $G$ has essentially the same coloured structure as the  reduced-graph $\cG$:
\begin{claim}
\label{G-structIIIA}
We can remove at most $14\eta^{1/2}n$ vertices from~$X_1$,
$14\eta^{1/2}n$ vertices from~$X_2$ and $38\eta^{1/2}n$ vertices from~$Y$
such that each of the following holds:
\end{claim}
\begin{itemize}
\labitem{KA1}{xy-red} $G_1[X_1,Y_1]$ and $G_1[X_2,Y_2]$ are each $7\eta^{1/2}n$-almost-complete;
\labitem{KA2}{xy-blue} $G_2[X_1,Y_2]$ and $G_2[X_2,Y_1]$ are each $7\eta^{1/2}n$-almost-complete;
\labitem{KA3}{y-green} $G_3[Y]$ is $10\eta^{1/2}n$-almost-complete.
\end{itemize}

\begin{proof}
Consider the complete three-coloured graph $G[Y]$ and recall that $\cG[\cY]$ is $4\eta^4k$-almost-complete and that all edges present in $\cG[\cY]$ are coloured exclusively green. Given the construction of~$\cG$, we can bound the number of non-green edges in $G[Y]$ by 
%
%
%
%
$$(r+s) \binom{N/K}{2} + 4\eta^4 (r+s)k \left(\frac{N}{K}\right)^2+2\eta\binom{r+s}{2}\left(\frac{N}{K}\right)^{2},$$
where
where the first term counts the number of non-green edges within the clusters,
the second counts the number of non-green edges between non-regular pairs,
the third counts the number of non-green edges between regular pairs.

Since $K\geq k, \eta^{-1}$, $N\leq 3n$ and $r+s\leq\aIII k\leq 2k$, we obtain
$$e(G_1[Y])+e(G_2[Y])\leq [9\eta + 72\eta^{4} +36 \eta]n^2\leq 50\eta n^2.$$
Since $G[Y]$ is complete and contains at most $50\eta n^2$ non-green edges, there are at most $10\eta^{1/2}n$ vertices with green degree at most $|Y|-10\eta^{1/2}n$. Re-assigning these vertices to~$Z$ gives a new~$Y$ 
such that every vertex in $G[Y]$ has red degree at least $|Y|-10\eta^{1/2}n$. 

Next, we consider $G[X_1,Y_1]$, bounding the the number of non-red edges in $G[X_1,Y_1]$ by 
$$4\eta^4 pk \left(\frac{N}{K}\right)^2+2\eta pr\left(\frac{N}{K}\right)^{2}, $$
where the first term bounds the number of non-red edges between non-regular pairs and the second bounds the number of non-red edges between regular pairs.

Since $K\geq k$, $N\leq 3n$, $p \leq \half\aI\leq \half k$ and $r\leq \aIII k \leq 2k$, we obtain
$$e(G_2[X_1,Y_1])+e(G_3[X_1,Y_1])\leq (18\eta^4+18\eta)n^2 \leq 40 \eta n^2.$$
Since $G[X_1,Y_1]$ is complete and contains at most $40\eta n^2$ non-red edges, there are at most $7\eta^{1/2}n$ vertices in~$X_1$ with red degree to~$Y_1$ at most $|Y_1|-7\eta^{1/2}n$ and at most $7\eta^{1/2}n$ vertices in~$Y_1$ with red degree to~$X_1$ at most $|X_1|-7\eta^{1/2}n$. Re-assigning these vertices to~$Z$ results in every vertex in~$X_1$ having degree in $G_1[X_1,Y_1]$ at least $|Y_1|-7\eta^{1/2}n$ and every vertex in~$Y_1$ having degree in $G_1[X_1,Y_1]$ at least $|X_1|-7\eta^{1/2}n$.
We repeat the above for each of~$G[X_1,Y_2], G[X_2,Y_1]$ and $G[X_2,Y_2]$, thus completing the proof of the claim. \end{proof}

\begin{figure}[!h]
\centering
\includegraphics[width=64mm, page=19]{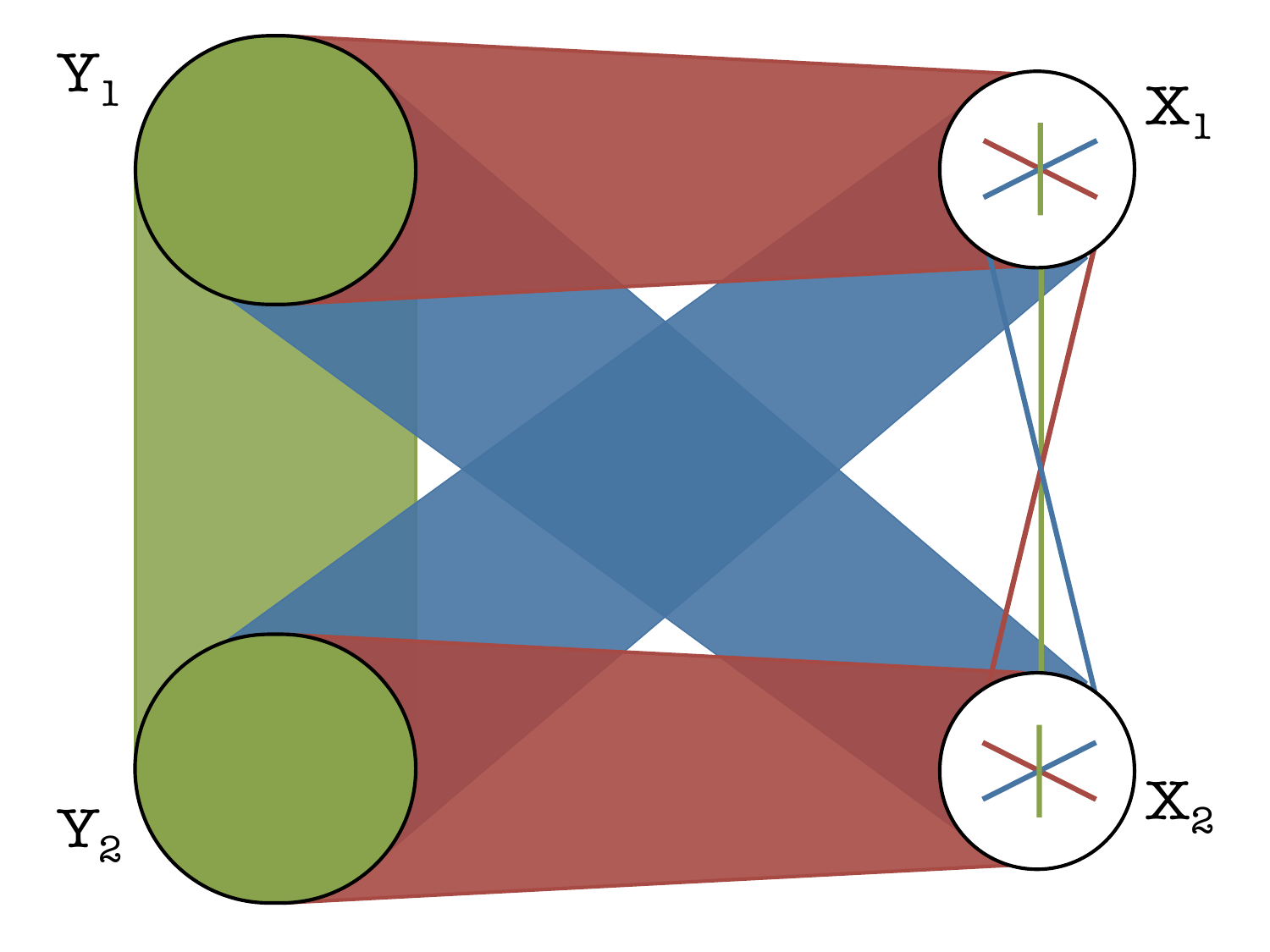}
\vspace{0mm}\caption{Colouring of $G$ after Claim~\ref{G-structIIIA}.}
\end{figure}

Having removed some vertices from $X_1,X_2,Y_1$ and $Y_2$, recalling (\ref{Z0}), we have
\begin{equation}
\label{Z1}
\left.
\begin{aligned}
\quad\quad\quad
|X_1|=|X_2|&\geq(\half\aI-112\eta^{1/2})n,\quad\quad&
|Y_1|&\geq(\half\aI-332\eta^{1/2})n,\,\,\quad\quad\\
|Y|&\geq(\aIII-840\eta^{1/2})n, &
|Y_2|&\geq(\half\aI-332\eta^{1/2})n.
\end{aligned}
\right\}
\end{equation}
Notice also, that, since $\eta\leq(\aII/2500)^2$, without loss of generality, we have 
\begin{equation}
\label{y1bigish}
\left.
\begin{aligned}
\quad\,\,\,|Y_1|=\max\{|Y_1|,|Y_2|\}&\geq \half(\aIII-840\eta^{1/2})n
\\&\geq(\tfrac{3}{4}\aI+\tfrac{1}{4}\aII-430\eta^{1/2})n\geq(\half\aI+100\eta^{1/2})n.\quad\,\,\,
\end{aligned}
\right\}
\end{equation}
Then, recalling (\ref{xy-red}) and (\ref{xy-blue}), by Corollary~\ref{moonmoser2}, in order to avoid having a red cycle on exactly~$\aIna$ vertices in $G[X_1,Y_1]$ or a blue cycle on exactly~$\aIIna$ vertices in $G[X_2,Y_1]$, we may assume that $|X_1|<\half\aIna$ and $|X_2|< \half\aIIna$. Also, recalling (\ref{y-green}), by Corollary~\ref{dirac1a}, in order to avoid having a green cycle on exactly $\aIIIna$, we may assume that $|Y|<\aIIIna$. 

Now, let 
$$Z_G=\{z\in Z : z\text{ has at least $850\eta^{1/2}n$ green edges to }Y\}.$$
Suppose that $|Y\cup Z_G|\geq\langle\aIII n\rangle$. Then, we may choose a set $Y^{\prime}$ of vertices from $Y\cup Z_G$, including every vertex from $Y$ and at most $840\eta^{1/2}n$ vertices from $Z_G$. In that case, $G[Y^{\prime}]$ has at least $(\aIII-840\eta^{1/2})n$ vertices of degree at least $(\aIII-850\eta^{1/4})n$ and at most $840\eta^{1/2}n$ vertices of degree at least~$850\eta^{1/2} n$, so, by Theorem~\ref{chv}, $G[Y^{\prime}]$ is Hamiltonian and thus contains a green cycle of length exactly $\langle \aIII n \rangle$. The existence of such a cycle would be sufficient to complete the proof of Theorem~\hyperlink{thA}{A}. Thus, we may assume, instead that $|Y\cup Z_G|<\langle\aIII n\rangle$.
Thus, letting $Z_{RB}=Z\backslash Z_G$, we may assume that 
\begin{equation}
\label{x1x2wrb}
|X_1|+|X_2|+|Z_{RB}|\geq\half\llangle\aI n\rrangle +\half\llangle \aII n \rrangle -1.
\end{equation}
By definition every vertex in $Z_{RB}$ has  at least $|Y|-850\eta^{1/2}n$ red or blue edges to $Y$.

 Observe that, given any pair of vertices $y_{11}, y_{12}$ in $Y_1$, we can use Lemma~\ref{bp-dir} to establish the existence of a long red path in $G[X_1,Y_1]$ from $y_{11}$ to $y_{12}$. Likewise, given $y_{21}, y_{22}$ in~$Y_2$, we can use Lemma~\ref{bp-dir} to establish the existence of a long red path in $G[X_2,Y_2]$ from $y_{21}$ to $y_{22}$. Thus, we may prove the following claim:
 
\begin{claim}
\label{redblueW}
If there exist distinct vertices $y_{11}, y_{12}\in Y_1$,  $y_{21}, y_{22}\in Y_2$ and $z_1,z_2\in Z$ such that the edges $z_1y_{11}, z_2y_{12}, z_1y_{21}$ and $z_2y_{22}$ are all coloured red, then $G$ contains a red cycle on exactly $\llangle \aI n \rrangle$.
\end{claim}

\begin{proof}
\begin{figure}[!h]
\centering
{\includegraphics[height=50mm, page=13]{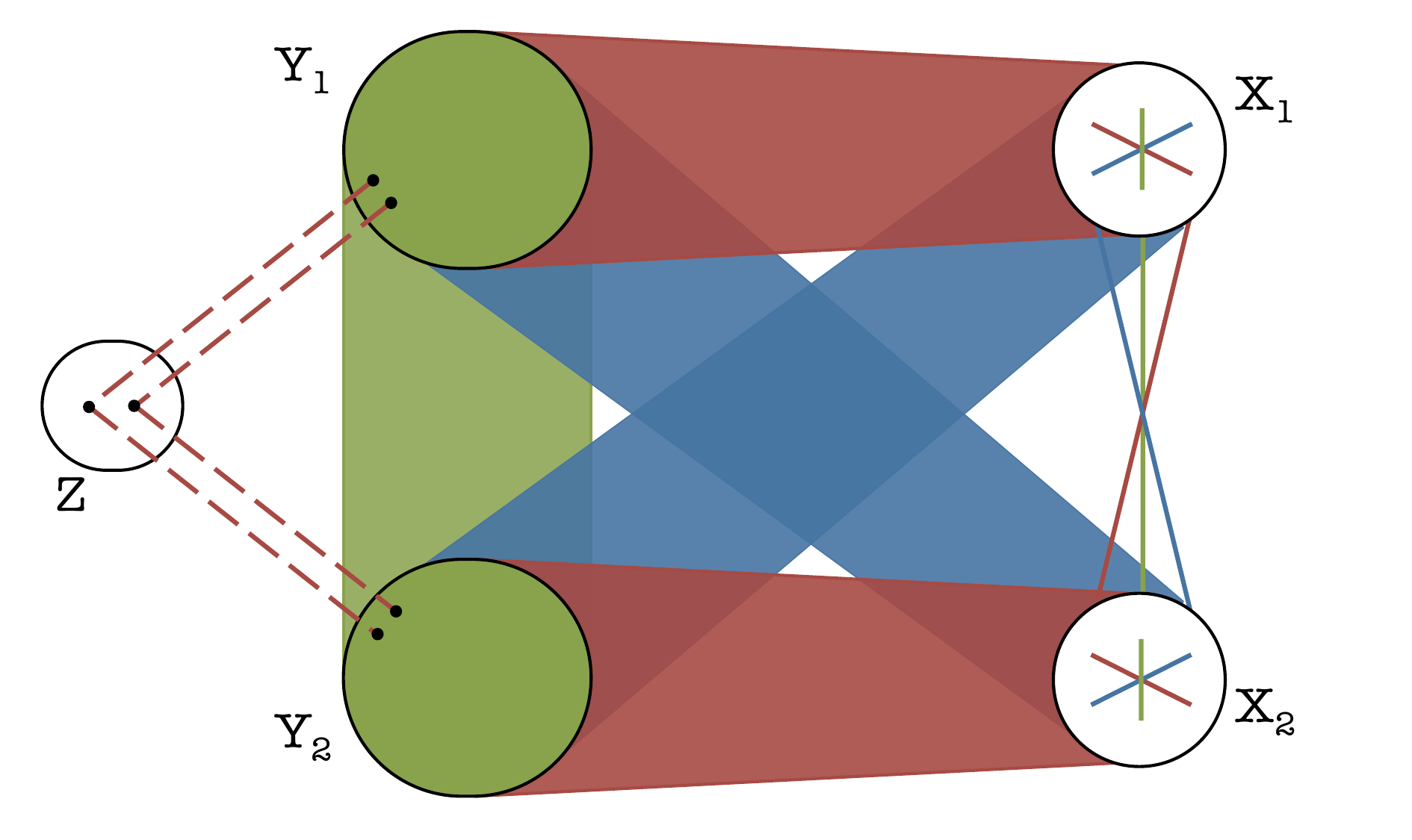}\hspace{18mm}}
\vspace{0mm}\caption{Existence of two red vertex-disjoint paths between $X_1$ and $Y_1$.}
  
\end{figure}
Suppose there exist distinct vertices $y_{11}, y_{12}\in Y_1$,  $y_{21}, y_{22}\in Y_2$ and $z_1, z_2\in Z$ such that the edges $z_1y_{11}, z_2y_{12}, z_1y_{21}$ and $z_2y_{22}$ are all coloured red. 
Then, let $\widetilde{X}_1$ be any set of 
$$\ell_1=\left\lfloor \frac{\llangle \aI n \rrangle-4}{4} \right\rfloor\geq 7\eta^{1/2}n+2$$ vertices from $X_1$.
 By~(\ref{xy-red}), $y_{11}$ and $y_{12}$ each have at least two neighbours in $\widetilde{X}_1$ and, since $\eta\leq(\aI/100)^{2}$, every vertex in $\widetilde{X}_1$ has degree at least 
$|Y_1|-7\eta^{1/2}n\geq \half|Y_1| +\half|\widetilde{X}_1| +1$ in $G[\widetilde{X}_1,Y]$. Then, Since $|Y_1|>\ell_1+1$, by Lemma~\ref{bp-dir},  $G_1[\widetilde{X}_1,Y_1]$ contains a red path $R_1$ on exactly $2\ell_1+1$ vertices from~$y_{11}$ to~$y_{12}$.
Similarly, letting $\widetilde{X}_2$ be any set of $$\ell_2=\left\lceil \frac{\llangle \aI n \rrangle-4}{4} \right\rceil\geq 7\eta^{1/2}n+2$$ vertices from $X_2$, by Lemma~\ref{bp-dir},  $G_1[\widetilde{X}_2,Y_2]$ contains a red path $R_2$ on exactly $2\ell_2+1$ vertices from~$y_{21}$ to~$y_{22}$. Then, combining $R_1$ and $R_2$ with $y_{11} w_1 y_{21}$ and $y_{12} w_2 y_{22}$ gives a red cycle on exactly $2\ell_1+2\ell_2+4=\llangle\aI n \rrangle$ vertices. 
\end{proof}

Similarly, the existence of two such vertex-disjoint paths blue paths from $Y_1$ to $Y_2$ via~$Z$ would result in a blue cycle on exactly~$\llangle \aII n \rrangle$ vertices. The existence of such a red or blue cycle would be sufficient to complete the proof of Theorem~\hyperlink{thA}{A}. Therefore, we may assume that there can be at most two vertices in $W_{RB}$ with red edges to both $Y_1$ and~$Y_2$ or blue edges to both $Y_1$ and $Y_2$. 

We denote by $Z^{*}$ the (possibly empty) set of vertices having either red edges to both $Y_1$ and $Y_2$ or blue edges to both $Y_1$ and $Y_2$.
Then, letting $Z_{RB}^{*}=Z_{RB}\backslash Z^{*}$, observe that we may partition $Z_{RB}^{*}$ into $Z_1\cup Z_2$ such that there are no blue edges in $G[Z_1,Y_1]\cup G[Z_2,Y_2]$ and no red edges in $G[Z_1,Y_2]\cup G[Z_2,Y_1]$. Then, recalling that every vertex in $Z_{RB}$ has at most $850\eta^{1/2}n$ green edges to $Y$, we know that

\begin{itemize}
\labitem{KA4a}{IIIA4a} every vertex in $Z_1$ has red degree at least $|Y_1|-850\eta^{1/2}n$ in $G[Z_1,Y_1]$;
\labitem{KA4b}{IIIA4b} every vertex in $Z_1$ has blue degree at least $|Y_2|-850\eta^{1/2}n$ in $G[Z_1,Y_2]$;
\labitem{KA4c}{IIIA4c} every vertex in $Z_2$ has red degree at least $|Y_2|-850\eta^{1/2}n$ in $G[Z_2,Y_2]$;
\labitem{KA4d}{IIIA4d} every vertex in $Z_2$ has blue degree at least $|Y_1|-850\eta^{1/2}n$ in $G[Z_2,Y_1]$.
\end{itemize}

\begin{figure}[!h]
\centering
{\includegraphics[height=53mm, page=2]{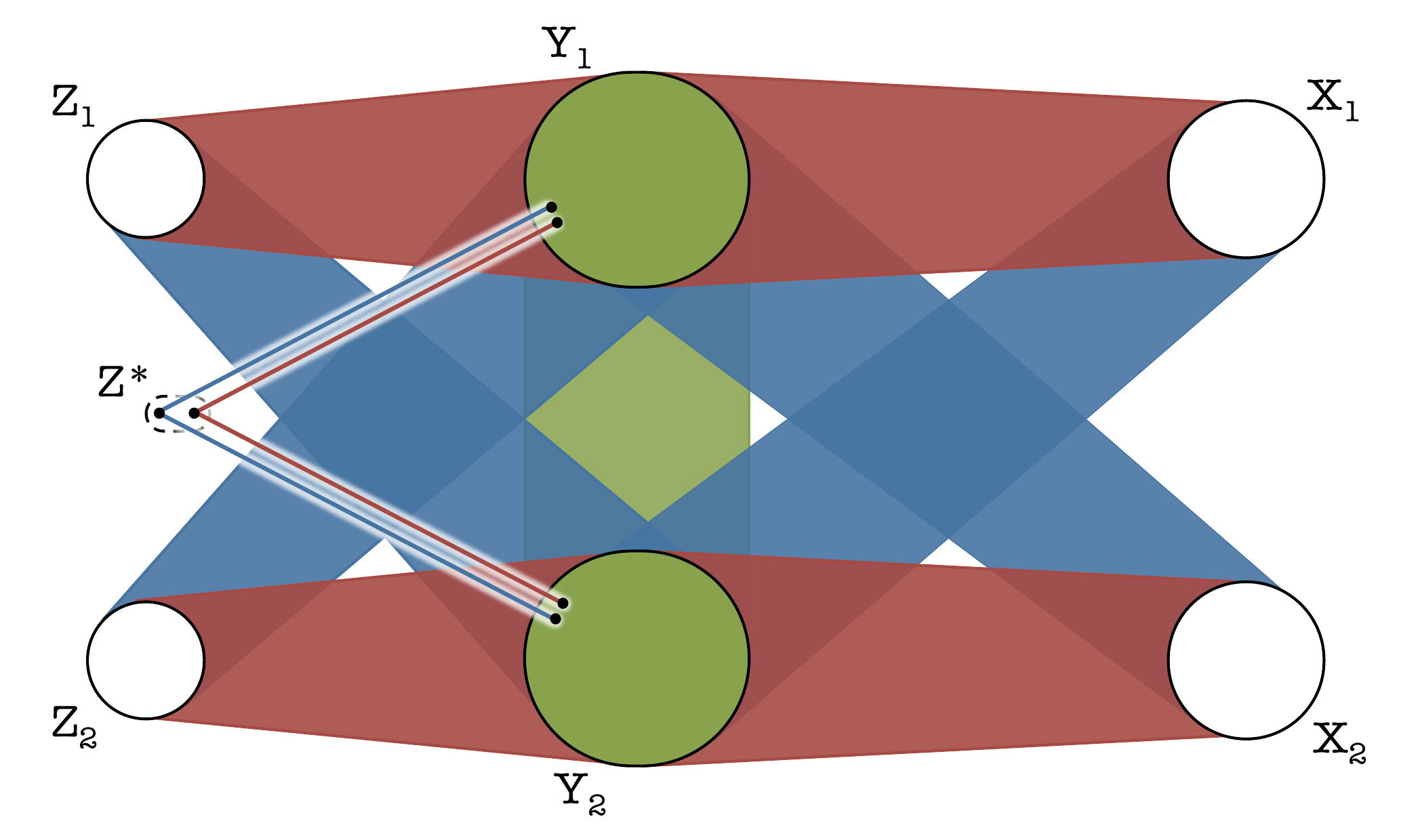}
\vspace{0mm}\caption{Partition of $Z_{RB}$ into $Z_1\cup Z_2\cup Z^{*}$.}}
\end{figure}

\FloatBarrier
Then, since $Z_{RB}=Z_1\cup Z_2\cup Z^{*}$, by~(\ref{x1x2wrb}), we have
$$|X_1|+|X_2|+|Z_1|+|Z_2|+|Z^{*}|\geq\half\llangle\aI n\rrangle +\half\llangle \aII n \rrangle -1.$$
Thus, one of the following must occur:
\begin{itemize}
\item[(i)] $|X_1|+|Z_1|\geq \half\llangle \aI n \rrangle$;
\item[(ii)] $|X_2|+|Z_2|\geq \half\llangle \aII n \rrangle$;
\item[(iii)] $|X_1|+|Z_1|=\half\llangle \aI n \rrangle-1$, $|X_2|+|Z_2|= \half\llangle \aII n \rrangle-1$, and $|Z^{*}|=1$;
\item[(iv)] $|X_1|+|Z_1|=\half\llangle \aI n \rrangle-2$, $|X_2|+|Z_2|= \half\llangle \aII n \rrangle-1$, and $|Z^{*}|=2$; 
\item[(v)] $|X_1|+|Z_1|=\half\llangle \aI n \rrangle-1$, $|X_2|+|Z_2|= \half\llangle \aII n \rrangle-2$, and $|Z^{*}|=2$.
\end{itemize}

In each case, we can show that $G$ contains either a red cycle on exactly~$\aIna$ vertices or a blue cycle on exactly~$\aIIna$ vertices as follows:

(i) Suppose $|X_1|+|Z_1|\geq\half\aIna$ and recall that, by~(\ref{y1bigish}), $|Y_1|\geq\half\aIna$. Then,  by~(\ref{Z1}), we may choose $\wX_1
 \in X_1\cup Z_1$ and $\wY_1 \in Y_1$ such that $|\wX_1|=|\wY_1|=\half\aIna$ and $|\wX_1\cap Z_1|\leq 114\eta^{1/2} n$.
By~(\ref{xy-red}), every vertex in $\wY_1$ has at least 
$|\wX_1\cap X_1|-7\eta^{1/2}n$ red neighbours in $\wX_1\cap X_1$, that is, at least $|\wX_1|-121\eta^{1/2}n$ red neighbours in $\wX_1$. By~(\ref{xy-red}) and~(\ref{IIIA4a}), every vertex in $\wX_1$ has at least $|\wY_1|-850\eta^{1/2}n$ red neighbours in $\wY_1$. Thus, for any $x\in \wX_1, y\in \wY_1$, $d(x)+d(y)\geq|\wX_1|+|\wY_1|-971\eta^{1/2}n\geq\half|\wX_1|+\half|\wY_1|+1$. So, by Theorem~\ref{moonmoser}, there exists a red cycle on exactly~$\aIna$ vertices in $G[\wX_1,\wY_1]$.

(ii) Suppose $|X_2|+|Z_2|\geq\half\aIIna$ and recall that, by~(\ref{y1bigish}), $|Y_1|\geq\half\aIIna$. Then,  by~(\ref{Z1}), we may choose $\wX_2
 \in X_2\cup Z_2$ and $\wY_1 \in Y_1$ such that $|\wX_2|=|\wY_1|=\half\aIIna$ and $|\wX_2\cap Z_2|\leq 114\eta^{1/2} n$.
By~(\ref{xy-blue}), every vertex in $\wY_1$ has at least 
$|\wX_2\cap X_2|-7\eta^{1/2}n$ blue neighbours in $\wX_2\cap X_2$, that is, at least $|\wX_2|-121\eta^{1/2}n$ neighbours in $\wX_2$. By~(\ref{xy-blue}) and~(\ref{IIIA4d}), every vertex in $\wX_2$ has at least $|\wY_1|-850\eta^{1/2}n$ blue neighbours in $\wY_1$. Thus, for any $x\in \wX_2, y\in \wY_1$, $d(x)+d(y)\geq|\wX_2|+|\wY_1|-971\eta^{1/2}n\geq\half|\wX_2|+\half|\wY_1|+1$. So, by Theorem~\ref{moonmoser}, there exists a blue cycle on exactly~$\aIIna$ vertices in $G[\wX_2,\wY_2]$.

(iii) Suppose that $|X_1|+|Z_1|=\half\llangle \aI n \rrangle-1$, $|X_2|+|Z_2|= \half\llangle \aII n \rrangle-1$ and $|Z^{*}|=1$. Consider $z\in Z^{*}$. Since $Z^{*}\subset Z_{RB}$, we know that $z$ has green edges to at most $850\eta^{1/2}n$ of the vertices of $Y=Y_1\cup Y_2$. Thus, $z$ either has red edges to at least two vertices in $Y_1$ or blue edges to at least two vertices in $Y_1$. We denote two of these as $y_1$ and $y_2$. 

In the former case, by (\ref{xy-red}), $y_1$ and $y_2$ each have at least two red neighbours in $X_1$. By (\ref{y1bigish}), we have $|Y_1|>|Z_1|+|X_1|+1$. By~(\ref{xy-red}),~(\ref{IIIA4a}) and (\ref{y1bigish}), since $\eta\leq(\aII/5000)^2$, every vertex in $Z_1\cup X_1$ has at least $|Y_1|-850\eta^{1/2}n\geq\half(|Z_1|+|X_1|+|Y_1|)+1$ red neighbours in $Y_1$. Thus, by Lemma~\ref{bp-dir}, there exists a red path in $G[X_1\cup Z_1,Y_1]$ from $y_1$ to $y_2$ which visits every vertex of $X_1\cup Z_1$. This path, together with the red edges $y_1z$ and $y_2z$, forms a red cycle on exactly~$\aIna$ vertices.

In the latter case, by (\ref{xy-blue}), $y_1$ and $y_2$ each have at least two blue neighbours in $X_2$. By (\ref{y1bigish}), we have $|Y_1|>|Z_2|+|X_2|+1$. By~(\ref{xy-blue}),~(\ref{IIIA4d}) and~(\ref{y1bigish}), since $\eta\leq(\aII/5000)^2$, every vertex in $Z_2\cup X_2$ has at least $|Y_1|-850\eta^{1/2}n\geq\half(|Z_2|+|X_2|+|Y_1|)+1$ blue neighbours in $Y_1$. Thus, by Lemma~\ref{bp-dir}, there exists a blue path in $G[X_2\cup Z_2,Y_1]$ from $y_1$ to $y_2$ which visits every vertex of $X_2\cup Z_2$. This path, together with the blue edges $y_1z$ and $y_2z$, forms a blue cycle on exactly~$\aIIna$ vertices.

(vi) Suppose that $|X_1|+|Z_1|=\half\llangle \aI n \rrangle-2$, $|X_2|+|Z_2|= \half\llangle \aII n \rrangle-1$ and $|Z^{*}|=2$. Then, considering $z_1,z_2\in Z^{*}$, since $Z^{*}\subset Z_{RB}$, we know that $z_1$ and $z_2$ each have green edges to at most $850\eta^{1/2}n$ of the vertices of $Y=Y_1 \cup Y_2$. Thus, either one of $z_1$,~$z_2$ has blue edges to two distinct vertices in $Y_1$ or both have at least $|Y_1|-900\eta^{1/2}n$ red neighbours in $Y_1$. In the former case, the situation is identical to one already considered in (iii): We have $|X_2|+|Z_2|=\half\aIIna-1$ and know of the existence of a vertex in $Z$ with blue edges to two distinct vertices in $Y_1$.

In the latter case, the situation is similar to the one considered in (i): We have $|X_1|+|Z_1|+|Z^{*}|\geq\half\aIna$ and $|Y_1|\geq\half\aIna$. By~(\ref{Z1}), we may choose $\wX_1
 \in X_1\cup Z_1\cup Z^{*}$ and $\wY_1 \in Y_1$ such that $|\wX_1|=|\wY_1|=\half\aIna$ and $|\wX_1\cap (Z_1\cup Z^{*})|\leq 114\eta^{1/2} n$.
By~(\ref{xy-red}) and~(\ref{IIIA4a}), every vertex in $\wY_1$ at least $|\wX_1|-121\eta^{1/2}n$ red neighbours in~$\wX_1$ and every vertex in $\wX_1$ has at least $|\wY_1|-900\eta^{1/2}n$ red neighbours in $\wY_1$. Thus, for any $x\in \wX_1, y\in \wY_1$, $d(x)+d(y)\geq|\wX_1|+|\wY_1|-1021\eta^{1/2}n\geq\half|\wX_1|+\half|\wY_1|+1$. So, by Theorem~\ref{moonmoser}, there exists a red cycle on exactly~$\aIna$ vertices in $G[\wX_1,\wY_1]$.

(v) Suppose that $|X_1|+|Z_1|=\half\llangle \aI n \rrangle-1$, $|X_2|+|Z_2|= \half\llangle \aII n \rrangle-2$ and $|Z^{*}|=2$. Then, considering $z_1,z_2\in Z^{*}$, since $Z^{*}\subset Z_{RB}$, we know that $z_1$ and $z_2$ each have green edges to at most $850\eta^{1/2}n$ of the vertices of $Y$. Thus, either one of~$z_1$,~$z_2$ has red edges to two distinct vertices in $Y_1$ or both have at least $|Y_1|-900\eta^{1/2}n$ blue neighbours in $Y_1$. In the former case, the situation is identical to one already considered in (iii): We have $|X_1|+|Z_1|=\half\aIna-1$ and know of the existence of a vertex in $Z$ with red edges to two distinct vertices in $Y_1$.

In the latter case, the situation is similar to the one considered in (ii): We have $|X_2|+|Z_2|+|Z^{*}|\geq\half\aIIna$ and $|Y_1|\geq\half\aIIna$. By~(\ref{Z1}), we may choose $\wX_2
 \in X_2\cup Z_2\cup Z^{*}$ and $\wY_1 \in Y_1$ such that $|\wX_2|=|\wY_1|=\half\aIIna$ and $|\wX_2\cap (Z_2\cup Z^{*})|\leq 114\eta^{1/2} n$.
By~(\ref{xy-blue}) and~(\ref{IIIA4d}), every vertex in $\wY_1$ at least $|\wX_2|-121\eta^{1/2}n$ blue neighbours in~$\wX_2$ and every vertex in $\wX_2$ has at least $|\wY_1|-900\eta^{1/2}n$ blue neighbours in $\wY_1$. Thus, for any $x\in \wX_2, y\in \wY_1$, $d(x)+d(y)\geq|\wX_1|+|\wY_1|-1021\eta^{1/2}n\geq\half|\wX_2|+\half|\wY_1|+1$. So, by Theorem~\ref{moonmoser}, there exists a red cycle on exactly~$\aIna$ vertices in $G[\wX_2,\wY_1]$.
The existence of such a red or blue cycle would be sufficient to complete the proof of Theorem~\hyperlink{thA}{A}. Thus,  we have completed Part III.A.

\subsection*{Part III.B: $K^*\in\cK_2^{*}$.}
In this case, we have a partition the vertex set $\cV$ of $\cG$ into $\cX_1\cup\cX_2\cup\cY_1\cup\cY_2\cup\cZ$ with
\begin{subequations}
\begin{align}
\label{ZZ0a}
|\cX_1|&\geq(\half\aI-97\eta^{1/2})k, & |\cX_2|&\geq(\half\aII-97\eta^{1/2})k
\end{align}
and, writing $\cY$ for $\cY_1\cup\cY_2$,
\begin{align}
\label{ZZ0b}
|\cY_1|&\geq(\tfrac{3}{4}\aIII-140\eta^{1/2})k, & |\cY_2|&\geq 100\eta^{1/2}k, &
|\cY|&\geq(\aIII-10\eta^{1/2})k,
\end{align}
\end{subequations}
such that all edges present in $\cG[\cX_1,\cY_1]\cup\cG[\cX_2,\cY_2]$ are coloured exclusively red, all edges present in $\cG[\cX_1,\cY_2]\cup\cG[\cX_2,\cY_1]$ are coloured exclusively blue and all edges present in $\cG[\cX_1,\cX_2]\cup\cG[\cY_1,\cY_2]$ are coloured exclusively green (see Figure~\ref{fig-III-initial}). Also, for any $\cZ\subseteq\cX_1\cup\cX_2\cup\cY_1\cup\cY_2$, $\cG[\cZ]$ is $4\eta^4 k$-almost-complete.
 Observe that, since 
 $\cG_1[\cX_1\cup\cY_1]$ is $4\eta^4 k$-complete, it
has a single red component.  Similarly, $\cG[\cX_2\cup\cY_2]$ has a single red component and each of $\cG[\cX_1\cup\cY_2]$ and $\cG[\cX_2\cup\cY_1]$ has a single blue component.  
 
Consider~$\cG[\cY_1]$ and suppose that there exists a red matching $\cR_1$ on $198\eta^{1/2}k$ vertices in $G[\cY_1]$. Then, we have $|\cY_1\backslash\cV(\cR_1)|,|\cX_1|\geq(\half\aI-97\eta^{1/2})k$, so, by Lemma~\ref{l:eleven}, $\cG[\cX_1,\cY_1\backslash \cV(\cR_1)]$ contains a red connected-matching on at least $(\aI-196\eta^{1/2})k$ vertices, which combined with $\cR_1$ gives a red connected-matching on at least~$\aI k$ vertices.  
Thus, there can be no such red matching in $\cG[\cY_1]$. Similarly,~$\cG[\cY_1]$ cannot contain a blue matching on $198\eta^{1/2}k$ vertices. Thus, after discarding at most $396\eta^{1/2}k$ vertices from~$\cY_1$, we may assume that all edges present in~$\cG[\cY_1]$ are coloured exclusively green. 
After discarding these vertices, recalling (\ref{ZZ0a}) and (\ref{ZZ0b}), we may assume that we have a partition $\cX_1\cup\cX_2\cup\cY_1\cup\cY_2\cup\cZ$  such that 
\begin{subequations}
\begin{equation}
\left.
\begin{aligned}
\label{ZZ0ca}
\quad\quad\quad\quad\quad\quad\quad\,\,\,\,
(\half\aI-97\eta^{1/2})k&\leq|\cX_1|=p\leq\half\aI k,
\quad\quad\quad\quad\quad\quad\quad\quad\quad\,\,\\
(\half\aII-97\eta^{1/2})k&\leq|\cX_2|=q\leq\half\aII k,\\
(\tfrac{3}{4}\aIII-536\eta^{1/2})k&\leq|\cY_1|=r\leq(\aIII-100\eta^{1/2})k, \\ 
100\eta^{1/2}k&\leq|\cY_2|=s\leq(\tfrac{1}{4}\aIII+536\eta^{1/2})k.
\end{aligned}
\right\}\!\!
\end{equation}
Additionally, writing $\cY$ for $\cY_1\cup\cY_2$, we have
\begin{align}
\label{ZZ0cb}
(\aIII-406\eta^{1/2})k&\leq |\cY|=r+s\leq \aIII k \,\,
\end{align}
\end{subequations}
and know that all edges present in $\cG[\cX_1,\cY_1]\cup\cG[\cX_2,\cY_2]$ are coloured exclusively red, all edges present in $\cG[\cX_1,\cY_2]\cup\cG[\cX_2,\cY_1]$ are coloured exclusively blue and all edges present in $\cG[\cX_1,\cX_2]\cup\cG[\cY_1,\cY_2]\cup\cG[\cY_1]$ are coloured exclusively green.

\begin{figure}[!h]
\centering
\vspace{-2mm}
\includegraphics[width=64mm, page=21]{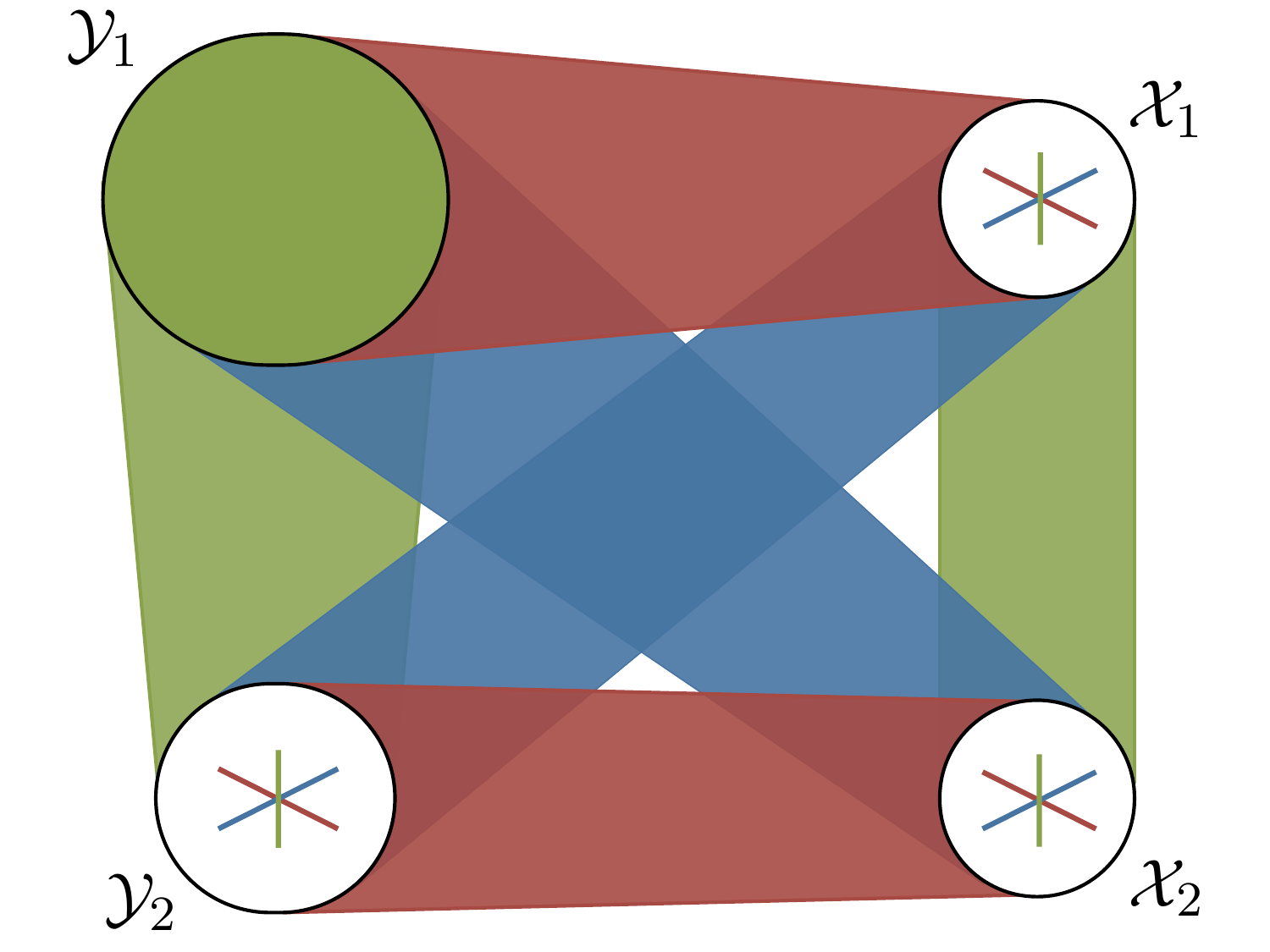}
\vspace{0mm}\caption{Colouring of $\cG[\cY_1]$.}
\end{figure}

The remainder of this section focuses on showing that the original graph must have a similar structure which can then be exploited to force a cycle of appropriate length, colour and parity. Again, each vertex $V_{i}$ of $\cG=(\cV,\cE)$ represents a cluster of vertices of $G=(V,E)$ with
$$(1-\eta^4)\frac{N}{K}\leq |V_i|\leq \frac{N}{K}$$ 
and, since $n> \max\{n_{\ref{th:blow-up}}(2,1,0,\eta), n_{\ref{th:blow-up}}(\half,\half,1,\eta)\}$, we have $$
|V_i|\geq \left(1+\frac{\eta}{24}\right)\frac{n}{k}> \frac{n}{k}.$$ 
Similarly, we may show that $$|V_i|\leq\big(1+\eta\big)\frac{n}{k}.$$

We partition the vertices of~$G$ into sets $X_{1}$, $X_{2}$, $Y_{1}$, $Y_{2}$, and~$Z$ corresponding to the partition of the vertices of~$\cG$
. Then~$X_1$ contains~$p$ clusters of vertices, $X_2$ contains $p$ clusters, $Y_1$ contains~$r$ clusters and $Y_2$ contains~$s$ clusters. 
Writing $Y$ for $Y_1\cup Y_2$ and recalling (\ref{ZZ0ca}) and (\ref{ZZ0cb}), we have
\begin{equation}
\label{ZZ0d}
\left.
\begin{aligned}
|X_1|&=p|V_1|\geq(\half\aI-97\eta^{1/2})n,& \quad\,\,
|Y_1|&=r|V_1|\geq(\tfrac{3}{4}\aIII-536\eta^{1/2})n,\\
|X_2|&=p|V_1|\geq(\half\aII-97\eta^{1/2})n,&
|Y_2|&=s|V_1|\geq100\eta^{1/2}n,\\
&&|Y|&=(r+s)|V_1|\geq(\aIII-406\eta^{1/2})n.
\quad\,\,
\end{aligned}
\right\}\!
\end{equation}

In what follows, we will remove vertices from $X_1\cup X_2\cup Y_1\cup Y_2$  by moving them into~$Z$. The following claim  tells us that the graph $G$ has essentially the same coloured structure as the  reduced-graph $\cG$:
\begin{claim}
\label{G-structIIIB}
We can remove at most $7\eta^{1/2}n$ vertices from~$X_1$,
$7\eta^{1/2}n$ vertices from~$X_2$, $31\eta^{1/2}n$ vertices from~$Y_1$ and $7\eta^{1/2}n$ vertices from~$Y_2$
such that each of the following holds:
\end{claim}
\begin{itemize}
\labitem{KB1}{Bxy-red} $G_1[X_1,Y_1]$ is $7\eta^{1/2}n$-almost-complete;
\labitem{KB2}{Bxy-blue} $G_2[X_2,Y_1]$ is $7\eta^{1/2}n$-almost-complete;
\labitem{KB3}{By-green} $G_3[Y_1]$ is $10\eta^{1/2}n$-almost-complete and $G_3[Y_1,Y_2]$ is $7\eta^{1/2}n$-almost-complete.
\end{itemize}

\begin{proof}
Consider the complete three-coloured graph $G[Y_1]$ and recall that~$\cG[\cY_1]$ is $4\eta^4k$-almost-complete and that all edges present in~$\cG[\cY_1]$ are coloured exclusively green. Given the construction of~$\cG$, we can bound the number of non-green edges in $G[Y_1]$ by
%
%
%
%
$$r \binom{N/K}{2} + 4\eta^4 rk \left(\frac{N}{K}\right)^2+2\eta\binom{r}{2}\left(\frac{N}{K}\right)^{2},$$ where the first term counts the number of non-green edges within the clusters,
the second counts the number of non-green edges between non-regular pairs,
the third counts the number of non-green edges between regular pairs.

Since $K\geq k, \eta^{-1}$, $N\leq 3n$ and $r\leq \aIII k\leq 2k$, we obtain
$$e(G_1[Y])+e(G_2[Y])\leq [9\eta + 72\eta^{4} +36 \eta]n^2\leq 50\eta n^2.$$
Since $G[Y_1]$ is complete and contains at most $50\eta n^2$ non-green edges, there are at most $10\eta^{1/2}n$ vertices with green degree at most $|Y_1|-10\eta^{1/2}n$. Re-assigning these vertices to~$Z$ gives a new~$Y_1$ 
such that every vertex in $G[Y_1]$ has red degree at least $|Y_1|-10\eta^{1/2}n$. 

Now, consider $G[X_1,Y_1]$, bounding the the number of non-red edges in $G[X_1,Y_1]$ by 
$$4\eta^4 pk \left(\frac{N}{K}\right)^2+2\eta pr\left(\frac{N}{K}\right)^{2}. $$
Where the first term bounds the number of non-red edges between non-regular pairs and the second bounds the number of non-red edges between regular pairs.

Since $K\geq k$, $N\leq 3n$, $p \leq \half\aI k\leq \half k$ and $r\leq \aIII k \leq 2k$, we obtain
$$e(G_2[X_1,Y_1])+e(G_3[X_1,Y_1])\leq (18\eta^4+18\eta)n^2 \leq 40 \eta n^2.$$
Since $G[X_1,Y_1]$ is complete and contains at most $40\eta n^2$ non-red edges, there are at most $7\eta^{1/2}n$ vertices in~$X_1$ with red degree to~$Y_1$ at most $|Y_1|-7\eta^{1/2}n$ and at most $7\eta^{1/2}n$ vertices in~$Y_1$ with red degree to~$X_1$ at most $|X_1|-7\eta^{1/2}n$. Re-assigning these vertices to~$Z$ results in every vertex in~$X_1$ having degree in $G_1[X_1,Y_1]$ at least $|Y_1|-7\eta^{1/2}n$ and every vertex in~$Y_1$ having degree in $G_1[X_1,Y_1]$ at least $|X_1|-7\eta^{1/2}n$.
By the same argument we can show that each of~$G_2[X_2,Y_1]$ and $G_3[Y_1,Y_2]$, are $7\eta^{1/2}n$-almost-complete, thus completing the proof of the claim. \end{proof}

\begin{figure}[!h]
\centering
\includegraphics[width=64mm, page=22]{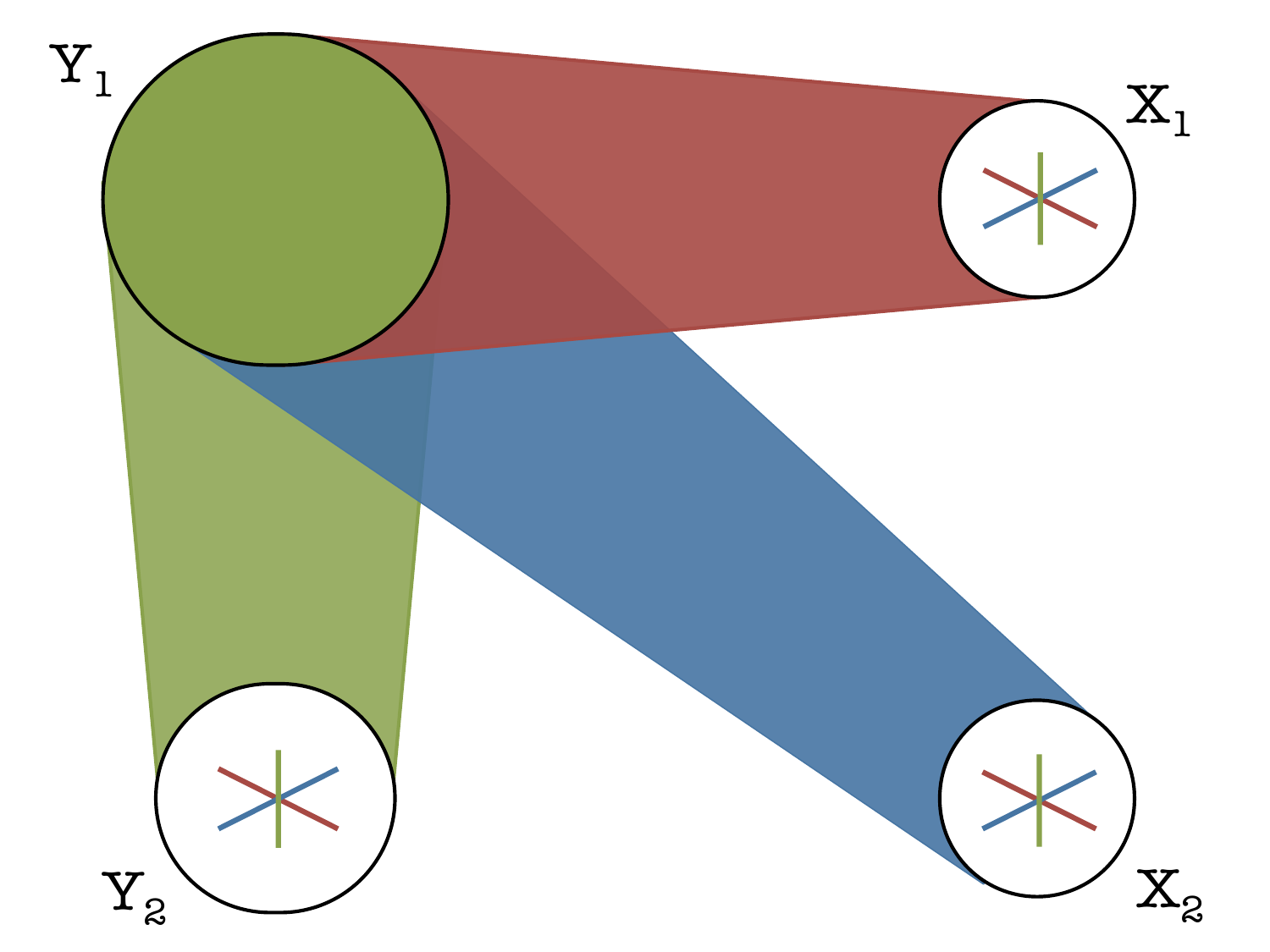}
\vspace{0mm}\caption{Colouring of $G$ after Claim~\ref{G-structIIIB}.}
 \end{figure}

Following the removals in Claim~\ref{G-structIIIB}, recalling (\ref{ZZ0d}), we have
\begin{equation}
\label{ZZ1}
\left.
\begin{aligned}
\quad\quad\quad\,\,\,\,
|X_1|&\geq(\half\aI-104\eta^{1/2})n,&
\quad\quad\quad\quad
|Y_1|&\geq(\tfrac{3}{4}\aIII-567\eta^{1/2})n,
\quad\quad\quad\\
|X_2|&\geq(\half\aII-104\eta^{1/2})n,&
|Y_2|&\geq93\eta^{1/2}n,\\
&&|Y|&\geq(\aIII-444\eta^{1/2})n.
\end{aligned}
\right\}
\end{equation}
and know that $G_1[X_1,Y_1]$, $G_2[X_2,Y_1]$ and $G_3[Y_1,Y_2]$ are each $7\eta^{1/2}n$-almost-complete and $G_3[Y_1]$ is $10\eta^{1/2}n$-almost-complete.
Observe that, by Corollary~\ref{moonmoser2}, there exist red cycles in $G[X_1,Y_1]$ of every (non-trivial) even length up to twice the size of the smaller part and  blue cycles in $G[X_2,Y_1]$ of every (non-trivial) even length up to twice the size of the smaller part. Additionally, we may use use Corollary~\ref{dirac2} and Lemma~\ref{bp-dir} to obtain a green cycle in $G[Y]$ of any 
length up to $|Y|=|Y_1|+|Y_2|$ as follows:

Recalling (\ref{ZZ0ca}), we have $s\leq(\tfrac{1}{4}\aIII+536\eta^{1/2})k$. Then, $|Y_2|= s|V_1|\leq (1+\eta)(\tfrac{1}{4}\aIII+536\eta^{1/2})n\leq(\tfrac{1}{4}\aIII+538\eta^{1/2})n$. Thus, we have $|Y_1|\geq(\tfrac{3}{4}\aIII-567\eta^{1/2})\geq|Y_2|+14\eta^{1/2}n+2$. Thus, since $G_3[Y_1]$ is $10\eta^{1/2}n$-almost-complete and $G_3[Y_1,Y_2]$ is $7\eta^{1/2}n$-almost-complete, every vertex in $Y_1$ has degree at least two in $G_3[Y_1,Y_2]$ and every vertex in $Y_2$ has degree at least $\half(|Y_1|+|Y_2|)+1$ in $G_3[Y_1,Y_2]$. Therefore, by Lemma~\ref{bp-dir}, given any two vertices $y_1$ and $y_2$ in $Y_1$, there exists a green path $P_1$ on~$2|Y_2|+1$ vertices from $y_1$ to $y_2$ in $G[Y_1,Y_2]$. 
Let $Y_1^{\prime}$ be a subset of $\big(Y_1\backslash V(P_1)\big)\cup\{y_1,y_2\}$ such that $y_1,y_2\in Y_1^{\prime}$ and
$|Y_1^{\prime}|
\geq 20\eta^{1/2}n +2.$
Then, since $G_3[Y_1]$ is $10\eta^{1/2}n$-almost-complete, every vertex in $Y_1^{\prime}$ has degree at least $\half|Y_1^{\prime}|+1$ in $G_3[Y_1^{\prime}]$ and so, by Corollary~\ref{dirac2}, there exists a green path $P_2$ on $|Y_1^{\prime}|$ vertices from $y_1$ to $y_2$ in $G[Y_1^{\prime}]$. Together, the green paths $P_1$ and $P_2$ form a green cycle on exactly $|Y_1^{\prime}|+|Y_2|$ vertices.

Thus, we may assume that
\begin{align}
\label{291}
|X_1|&\leq\half\aIna, & |X_2|&\leq\half\aIIna, & |Y|&\leq \aIIIna.
\end{align}
 We will show that it is possible to augment each of $X_1, X_2, Y$ with vertices from~$Z$ and that, considering the sizes of each part, there must in fact be a cycle of appropriate length, colour and parity to complete the proof.
 Observe that, by (\ref{By-green}) and (\ref{ZZ1}), every vertex in $Y_1$ has degree at least 
$$(\tfrac{3}{4}\aIII-577\eta^{1/2})n\geq \aIIIna -(\tfrac{1}{4}\aIII+577\eta^{1/2})n$$
in $G_3[Y_1]\subseteq G_3[Y]$ and every vertex in $Y_2$ has degree at least
$$(\tfrac{3}{4}\aIII-574\eta^{1/2})n\geq \aIIIna -(\tfrac{1}{4}\aIII+574\eta^{1/2})n$$ in $G_3[Y_1,Y_2]\subseteq G_3[Y]$.
Then, let $$Z_G=\{z\in Z : z\text{ has at least }(\tfrac{1}{4}\aIII+578\eta^{1/2})n \text{ green edges to }Y_1\}$$ and suppose that $|Z_G\cup Y_1\cup Y_2|\geq\langle\aIII n\rangle$. 

In that case, since $|Y|\leq \aIIIna$, we we may choose a subset $\wY$ of size $\langle \aIII n \rangle$ from $Z_G\cup Y_1\cup Y_2$ which includes every vertex of $Y_1\cup Y_2$ and $\aIIIna - |Y|$ vertices from~$Z_G$. Then, by (\ref{ZZ1}), $\wY$ includes at least $(\aIII -444\eta^{1/2})n$ vertices from $Y_1\cup Y_2$ and at most $445\eta^{1/2}n$ vertices from $Z_G$. Thus, $G[Y]$ has at least $(\aIII -444\eta^{1/2})n$ vertices of degree at least $\aIIIna -(\tfrac{1}{4}\aIII+577\eta^{1/2})n$ and at most $445\eta^{1/2}$ vertices of degree at least $(\tfrac{1}{4}\aIII+578\eta^{1/2})n$. Therefore, by Theorem~\ref{chv}, $G[Y]$ is Hamiltonian and thus contains a green cycle of length exactly $\langle \aIII n \rangle$. The existence of such a cycle would be sufficient to complete the proof of Theorem~\hyperlink{thA}{A} in this case so we may assume instead that $|X_3\cup Z_G|<\langle\aIII n\rangle$.

Thus, letting $Z_{RB}=Z\backslash Z_G$, we may assume that 
$$|X_1|+|X_2|+|Z_{RB}|\geq\half\llangle\aI n\rrangle +\half\llangle \aII n \rrangle -1$$
and, defining
\begin{align*}
Z_R&=\{w\in Z : w\text{ has at least }(\half\aI-575\eta^{1/2})n\text{ red edges to }Y_1\},\\
Z_B&=\{w\in Z : w\text{ has at least }(\half\aII-575\eta^{1/2})n\text{ blue edges to }Y_1\},
\end{align*}
may assume, without loss of generality, that $|X_1\cup Z_R|\geq \llangle \aI n \rrangle.$

In that case, by (\ref{291}), we may choose a subset $\widetilde{Z}_R\subseteq Z_R$ such that $|X_1|+|\widetilde{Z}_R|=\half\aIna$. 
By (\ref{ZZ1}), we have $|\widetilde{Z}_R|\leq106\eta^{1/2}n$. Then, observing that any $z\in Z_R$ and $x\in X_1$ have at least $(\half\aI-582\eta^{1/2})n\geq|\widetilde{Z}_R|$ common neighbours and that any $x,y\in X_1$ have at least $(\tfrac{3}{4}\aIII -581\eta^{1/2})n\geq\aI n$ common neighbours, we can greedily construct a red cycle of length~$\llangle \aI n \rrangle$ using all the vertices of $X_1\cup \widetilde{Z}_R$ and $\half\aIna$ vertices from~$Y_1$, completing Part III of the proof of Theorem~\hyperlink{thA}{A}.

Observing that we have exhausted all the possibilities arising from Theorem~\hyperlink{thB}{B} and that the graphs providing the corresponding lower bounds have already been seen (in Section~\ref{ram:low}), we have proved that, given $\aI,\aII,\aIII>0$ such that $\aI\geq\aII$, there exists $n_A=n_A(\aI,\aII,\aIII)$ such that, for $n > n_{A}$, 
 \begin{align*}
 R(C_{\llangle \alpha_{1} n \rrangle},C_{\llangle \alpha_{2} n \rrangle}, C_{\langle \alpha_{3} n \rangle }) = \max\{ 2\llangle \alpha_{1} n \rrangle \!+\! \llangle \alpha_{2} n \rrangle  \!-\!\text{\:}3,\text{\:}\half\llangle  \alpha_{1} n \rrangle  \!+\! \half\llangle \alpha_{2} n \rrangle  \!+\! \langle \alpha_{3} n \rangle \!-\! \text{\:}2\},
\end{align*}
 thus completing the the proof of Theorem~\hyperlink{thA}{A}.
\qed

\renewcommand{\baselinestretch}{1.15}\small\normalsize

\cleardoublepage
\phantomsection
\addcontentsline{toc}{section}{References}
\clearpage
\bibliographystyle{halpha2}
\bibliography{test}

\end{document}

%% file: packagesetc.tex
\usepackage[slovak, english]{babel}

\usepackage{amsthm}
\usepackage{textcomp}

\newtheorem*{thA}{Theorem A}
\newtheorem*{thB}{Theorem B}
\newtheorem*{thC}{Theorem C}

\newtheorem{theorem}{Theorem}[section]
\newtheorem{claim}[theorem]{Claim}
\newtheorem{lemma}[theorem]{Lemma}

\newtheorem{corollary}[theorem]{Corollary}

\newtheorem{definition}[theorem]{Definition}

\newcommand{\aI}{{\alpha_{1}}}
\newcommand{\aII}{{\alpha_{2}}} 
\newcommand{\aIII}{{\alpha_{3}}}

\newcommand{\half}{\tfrac{1}{2}}

\newcommand{\llangle}{{\langle \! \langle}}
\newcommand{\rrangle}{{\rangle \! \rangle}}

\newcommand{\cE}{{\mathcal E}}

\newcommand{\cG}{{\mathcal G}}
\newcommand{\cH}{{\mathcal H}}

\newcommand{\cK}{{\mathcal K}}

\newcommand{\cR}{{\mathcal R}}

\newcommand{\cV}{{\mathcal V}}
\newcommand{\cX}{{\mathcal X}}
\newcommand{\cY}{{\mathcal Y}}
\newcommand{\cZ}{{\mathcal Z}}

\usepackage{amsfonts}
\usepackage{amsmath}
\usepackage{amssymb} 
\usepackage{mathrsfs}
\usepackage{setspace}
\usepackage{paralist}

\usepackage{geometry}
\usepackage{srcltx}
\usepackage{graphicx}
\usepackage{color}
\usepackage{upref}
\usepackage{enumerate}
\usepackage{latexsym}
\usepackage{amsmath}
\usepackage{amsfonts}
\usepackage{amssymb}
\usepackage[ansinew]{inputenc}
\usepackage{varioref}
\usepackage{subfig} 
\usepackage[final, notcite,notref]{showkeys} 
\usepackage{verbatim} 
\usepackage{soul}
\usepackage{complexity}